\documentclass[10pt,reqno]{amsart}

\usepackage[colorlinks=true, pdfstartview=FitV, linkcolor=blue,citecolor=red, urlcolor=blue]{hyperref}

\usepackage{amsmath,amsthm,amssymb}
\usepackage{color}
\usepackage{hyperref}
\usepackage{url}
\newcommand{\bburl}[1]{\textcolor{blue}{\url{#1}}}

\newtheorem{thm}{Theorem}[section]
\newtheorem{cor}[thm]{Corollary}
\newtheorem{lem}[thm]{Lemma}
\newtheorem{prop}[thm]{Proposition}

\newtheorem{cla}[thm]{Claim}

\theoremstyle{definition}

\theoremstyle{definition}
\newtheorem{defi}[thm]{Definition}

\newtheorem{rem}[thm]{Remark}
\theoremstyle{remark}

\newcommand\be{\begin{equation}}
\newcommand\ee{\end{equation}}
\newcommand\bee{\begin{equation*}}
\newcommand\eee{\end{equation*}}
\newcommand\ben{\begin{enumerate}}
\newcommand\een{\end{enumerate}}

\newcommand{\A}{\ensuremath{\mathbb{A}}}
\newcommand{\R}{\ensuremath{\mathbb{R}}}
\newcommand{\C}{\ensuremath{\mathbb{C}}}
\newcommand{\Z}{\ensuremath{\mathbb{Z}}}
\newcommand{\Q}{\mathbb{Q}}

\numberwithin{equation}{section}

\DeclareMathOperator{\Proj}{Proj}
\DeclareMathOperator{\sgn}{sgn}

\DeclareMathOperator{\re}{Re}
\DeclareMathOperator{\im}{Im}

\DeclareMathOperator{\diag}{diag}

\newcommand{\ovA}[1]{\overline{#1\raisebox{3mm}{}}}
\newcommand{\ovB}[1]{\overline{#1\raisebox{4mm}{}}}
\newcommand{\ovC}[1]{\overline{#1\raisebox{5mm}{}}}

\usepackage{empheq}

\usepackage{mathdots}
\usepackage{bm}

\usepackage{mathabx,epsfig}

\usepackage{pst-node}
\usepackage{tikz-cd}
\usepackage{verbatim}
\usepackage{mathdots}

\usepackage{scalerel,stackengine}
\stackMath

\newcommand\reallywidehat[1]{%
	\savestack{\tmpbox}{\stretchto{%
			\scaleto{%
				\scalerel*[\widthof{\ensuremath{#1}}]{\kern.1pt\mathchar"0362\kern.1pt}%
				{\rule{0ex}{\textheight}}%WIDTH-LIMITED CIRCUMFLEX
			}{\textheight}% 
		}{2.4ex}}%
	\stackon[-6.9pt]{#1}{\tmpbox}%
}
\usepackage{listings}
\title[Spectral Moment Formulae for $GL(3)\times GL(2)$ $L$-functions I: The Cuspidal Case]{Spectral Moment Formulae for $GL(3)\times GL(2)$ $L$-functions I: The Cuspidal Case}

\author[C.-H. Kwan]{Chung-Hang Kwan}
\email{\textcolor{blue}{\href{mailto: ucahckw@ucl.ac.uk}
{ucahckw@ucl.ac.uk}}}
\address{University College London}

\subjclass[2010]{11F55 (Primary) 11F72 (Secondary)}

\keywords{Automorphic Forms,  Automorphic $L$-functions, Maass forms,  Moments of $L$-functions,  Rankin-Selberg $L$-functions,  Period Integrals, Whittaker Functions, Hypergeometric Functions,  Poincar\'e series}

\date{\today}

\begin{document}

\maketitle

\begin{abstract}
	Spectral moment formulae  of various shapes have proven to be  very successful  in studying the statistics of central $L$-values.    In this article, we establish, in a completely explicit fashion,   such formulae for the family of  $GL(3)\times GL(2)$  Rankin-Selberg $L$-functions using the  period integral method.  The Kuznetsov and the Voronoi formulae are not needed in our argument.   We also prove the essential analytic properties and explicit formulae for the integral transform of our moment formulae.  It is hoped that our method  will provide insights into moments of $L$-functions for higher-rank groups.

	\end{abstract}
	
	%	 involving  a Poincar\'{e} series.
	
	%Favourable features include the precision of the results and the structural elegance of the argument. 
	
	%	with generalizations  to  higher-rank instances being anticipated. 

%\tableofcontents

\section{Introduction}

\subsection{Background}

The study of  $L$-values at the central point $s=1/2$ has taken center stage for many branches of  number theory in the past  decades because of the great arithmetic significance behind them. There have been a variety of interesting perspectives furnishing our understanding of the nature of central $L$-values. As an example, one may wish to take a statistical look at them. Fundamental questions in this direction include the determination of  (non-)vanishing and sizes of these $L$-values.    A particularly effective way to approach problems of this sort is via \textbf{Moments of $L$-functions}. Techniques from analytic number theory have proven to be very successful in  estimating the sizes of moments of all kinds. Moreover, spectacular results can be obtained when moment estimates join forces with arithmetic geometry and   automorphic representations.

\begin{comment}
 In the context of elliptic curves, the celebrated BSD conjecture asserts that the order of vanishing of the Hasse-Weil $L$-function at the central point coincides with the algebraic rank of the elliptic curve in question. 
\end{comment}

This line of investigation is nicely exemplified by the landmark result of Conrey-Iwaniec \cite{CI00}. Let $\chi$ be a real primitive Dirichlet character $ (\bmod\, q)$ with $q$  odd and square-free. The main object of  \cite{CI00} is   the cubic moment of $GL(2)$ automorphic $L$-functions of the congruence subgroup $\Gamma_{0}(q)$ twisted by $\chi$. An \textit{upper bound}  of Lindel\"{o}f strength in the $q$-aspect was established therein. When combining this upper bound with the celebrated Waldspurger formula \cite{Wa81}, the famous Burgess $3/16$-bound for Dirichlet $L$-functions was improved  for the first time since the 1960's. In fact,  \cite{CI00} proved the bound
\begin{align}
L\left( \frac{1}{2}, \chi\right) \ \ll_{\epsilon} \ q^{\frac{1}{6}+\epsilon}. 
\end{align}

\begin{comment}
 the central values of such $L$-functions are non-negative. Altogether, 
 \end{comment}
 
 %as well as many other work in the area, 
 
 %An important part of \cite{CI00} is the 
 
 %by following the method of  closely with great extra care,
 
% to understand the sequence 
 
 Understanding the effects of a sequence of intricate transformations (of both arithmetic and analytic nature) constitutes a significant part of moment calculations as seen in \cite{CI00}.  Surprisingly, such a sequence of \cite{CI00} ends up in a single elegant \textit{identity} showcasing a duality between the cubic average over a basis of  $GL(2)$ automorphic forms (Maass or holomorphic) and the fourth moment of  $GL(1)$ $L$-functions.  This remarkable phenomenon was uncovered relatively recently by Petrow \cite{Pe15}. His work consists of new elaborate analysis (see also Young \cite{Y17}) building upon the foundation of \cite{CI00}.  Let us also mention Frolenkov \cite{Fr20} and the earlier works of Ivi\'{c} \cite{Iv01, Iv02} for other aspects of the problem.  In its basic form, the identity roughly takes the shape 
\begin{align}\label{basicmoto}
\sum_{f: GL(2)} \  L\left( \frac{1}{2}, f\right)^3 \  \ = \   \int\limits_{-\infty}^{\infty}  \ \left|\zeta\left(\frac{1}{2}+it\right)\right|^4   \ dt \ + \  (***),
\end{align}
where the weight functions for the moments are suppressed and  $(***)$ represents certain  polar contributions.

\begin{comment}
	Important!!!!!!!!!
	
	and has the benefit of being able to extract the full set of main terms according to the philosophy of Conrey et. al. \cite{CFKRS05}.  The latter turns out to be  rather non-trivial due to the presence of the many off-diagonal  contributions which requires careful combinations and cancellations to yield meaningful answers.  In particular, this interesting task (as commented by the  authors of \cite{CI00}, see  pp. 1177 therein) was unfortunately left out in \cite{CI00} due to the mentioned difficulty. 
\end{comment}

%Also, it transforms a  moment problem about the \textit{spectral aspect} to a different one about the \textit{$t$-aspect}.

Besides its structural elegance, the  identity (\ref{basicmoto}) comes with immediate applications. It leads to sharp moment estimates as a consequence of exact evaluation. As an extra benefit,  it cleans up the analysis in the traditional but approximate approach. In \cite{Pe15}, such an identity was termed a \textbf{`Motohashi-type identity'}.  Indeed, Motohashi \cite{Mo93, Mo97} discovered an identity of this sort but with the choice of test function made on the fourth moment side instead, i.e., in the reverse direction of  \cite{CI00,  Pe15, Y17,  Iv01, Iv02}. It greatly enhances our understanding of the fourth moment of the $\zeta$-function.  There are also the recent works of Young \cite{Y11},  Blomer-Humphries-Khan-Milinovich \cite{BHKM20},  Topacogullari  \cite{To21} and Kaneko \cite{Ka21+} extending Motohashi's work to Dirichlet $L$-functions.

%\footnote{ In \cite{Mo93b, Mo97}, the continuous moment under consideration is the fourth moment of  the Riemann $\zeta$-function. There is a short note of Motohashi  \cite{Mo93a} announcing the corresponding result in the case of Dirichlet $L$-functions.  In principle, the proof is along the same line of  \cite{Mo93b, Mo97}. Nevertheless, the non-archimedean aspect of the problem deserves more elaboration  and is  recently addressed thoroughly in the work of .   }

%It can be encapsulated by

%(for the $GL(2)$ Eisenstein series) 

In the introduction of \cite{CI00},  Conrey-Iwaniec further envisioned the possibilities and challenges of extending their method to a setting involving a $GL(3)$ automorphic form. 
This is reasonable because the cubic moment of $GL(2)$ $L$-functions can be regarded as the first moment of $GL(3)\times GL(2)$ Rankin-Selberg $L$-functions averaging over a basis of $GL(2)$ automorphic forms, where the $GL(3)$ automorphic form is an Eisenstein series of minimal parabolic type.  It is natural  to anticipate some form of  `harmonic analysis of $GL(3)$'  would provide new perspectives towards the Conrey-Iwaniec method. Also, it is worthwhile to point out the  $GL(3)$ set-up consists of an important new example of the method: the first moment of $GL(3)\times GL(2)$ type with a $GL(3)$ \textit{cusp form} (which certainly requires genuine $GL(3)$ machinery).

In the ten years following  \cite{CI00}, it took two breakthrough works to make the Conrey-Iwaniec method  possible for the group $GL(3)$.  Firstly, there was   Miller-Schmid \cite{MS06} (see also \cite{GoLi06, IT13})) who developed the \textit{$GL(3)$ Voronoi formula} for the first time and made it usable for a variety of analytic applications. In particular, the Hecke combinatorics of $GL(3)$ associated to twisting and ramifications  are considerably more involved than the classical $GL(2)$ counterpart. Secondly, Xiaoqing Li \cite{Li11} succeeded to apply  the $GL(3)$ Voronoi formula together with new techniques of her own to obtain good \textit{upper bounds} for the first moment of $GL(3)\times GL(2)$ Rankin-Selberg $L$-functions in the $GL(2)$ spectral aspect (for the cuspidal case).   As a Corollary, she obtained the first instance of subconvexity for $GL(3)$ automorphic $L$-functions.

%%%%%%%%%%%%%%%%%%%%%%%%%%%%%%%%%%%%%%%%%%%%%%%%%%%%%%%%%%%%%%%%%%%%%%%%%%%%%%%%%%%%%%%%%%%%%%%%%%%%%%%%%%%%%%%%%%%%%%%%%%%%%%%%%%%%%%%%%%%%%%%%%%%%%%%%%%%%%%%%%%%%%%%%%

\subsection{Main Results}

The purpose of this article is to further the investigation of $GL(3)\times GL(2)$ moments of $L$-functions. However, we shall deviate from the line of attack of the aforementioned literature.  We are interested in understanding the \textit{intrinsic mechanisms} and examining the \textit{essential ingredients} that would lead more directly towards the full  \textit{exact structures} (main terms and off-diagonals) of moments of such kind. It is important to address these points carefully so as to open up possibilities for generalizations to higher-rank groups.  The formalism of  \textbf{period integrals} for $GL(3)$ was found to be convenient in achieving our goals. 

%and working with the \textit{exactly} (as opposed to \textit{approximate}) structures rather 

%(see our discussions in Section \ref{fea})

%(see our discussions in Section \ref{comCI})
 
We are ready to state the main result of this article, which is the moment identity of Motohashi type behind the work of \cite{Li11}.

\begin{thm}\label{maingl3gl2}
	Let 
	\begin{itemize}
		\item 	$\Phi$ be a fixed, Hecke-normalized Maass cusp form of $SL_{3}(\Z)$ with the Langlands parameters $(\alpha_{1}, \alpha_{2}, \alpha_{3}) \in (i\R)^3$, and $\widetilde{\Phi}$ be  the dual form of $\Phi$;   \newline
		
		\item  $(\phi_{j})_{j=1}^{\infty}$ be an orthogonal basis of  \textbf{even},  Hecke-normalized  Maass cusp forms of $SL_{2}(\Z)$ which satisfy $\Delta\phi_{j}= \left( \frac{1}{4}-\mu_{j}^2\right) \phi_{j}$;  \newline
		
		\item $L\left( s , \phi_{j} \otimes  \Phi \right)$ and $L\left( s , \Phi \right)$ be the Rankin-Selberg $L$-function of  the pair $(\phi_{j}, \Phi)$ and the standard $L$-function of $\Phi$ respectively,  where $\Lambda$ denotes the corresponding complete $L$-functions;  \newline
		
		\item  $\mathcal{C}_{\eta}$ ($\eta>40$) be the class of holomorphic functions $H$ defined on the vertical strip  $|\re \mu|< 2\eta$ such that  $H(\mu)=H(-\mu)$ and has rapid decay:
		\begin{align*}
		H(\mu) \ \ll \ e^{-2\pi|\mu|} \hspace{50pt} (|\re \mu| \ < \  2\eta). 
		\end{align*}

		\item For  $H\in \mathcal{C}_{\eta}$,  $\left(\mathcal{F}_{\Phi} H\right)\left(s_{0},  s \right) $ is the integral transform defined  in equation  (\ref{doubmbtrans}) and   it only depends on the Langlands parameters of $\Phi$. 
		\end{itemize}
	  Then on the domain   $\frac{1}{4}+ \frac{1}{200}  <  \sigma  < \frac{3}{4}$, we have the following moment identity: 
	\begin{align}\label{momaid}
	    \sum_{j=1}^{\infty} \  H\left(\mu_{j}\right) \ & \   \frac{\Lambda\left(s, \phi_{j}\otimes \widetilde{\Phi}\right) }{\langle \phi_{j}, \phi_{j}\rangle}  \ + \   \int_{(0)} \   H\left(\mu\right) 
	\frac{ \Lambda\left( s+\mu,   \widetilde{\Phi} \right)\Lambda\left( 1-s+ \mu,   \Phi \right)}{\left|\Lambda(1+2\mu)\right|^2} \  \frac{d\mu}{4\pi i } \nonumber\\
	\ &= \   \frac{ \pi^{-3s}}{2} \   L(2s, \Phi)\   \int_{(0)}  \ \frac{H(\mu)}{|\Gamma(\mu)|^2} \cdot  \prod_{i=1}^{3} \  \Gamma\left(\frac{s+\mu-\alpha_{i}}{2}\right)\Gamma\left(\frac{s-\mu-\alpha_{i}}{2}\right) \ \frac{d\mu}{2\pi i} \nonumber\\
	& \hspace{90pt}+ \frac{1}{2} \ L(2s-1, \Phi)   \left(\mathcal{F}_{\Phi} H\right)\left(2s-1,  s\right) \nonumber \\
	&  \hspace{130pt} + \frac{1}{2}  \  \int_{(1/2)} \zeta\left(2s-s_{0}\right) L\left(s_{0},  \Phi\right) \left(\mathcal{F}_{\Phi} H\right)\left(s_{0},  s \right) \ \frac{ds_{0}}{2\pi i}. 
	\end{align}
The function $s\mapsto \left(\mathcal{F}_{\Phi} H\right)\left(2s-1,  s\right)$ can be computed explicitly, see  Theorem \ref{CFKGamma} below. 

\end{thm}

The temperedness assumption $(\alpha_{1}, \alpha_{2}, \alpha_{3})\in (i\R)^{3}$ for our fixed Maass cusp form $\Phi$ is very mild ---  it merely serves as a simplification of our exposition (when applying Stirling's formula in Section \ref{Stirl}) and can be removed with a little more effort. In fact,  all Maass cusp forms of $SL_{3}(\Z)$ are conjectured to be tempered and it was proved in  \cite{Mil01} that the non-tempered forms constitute a density zero set.  

We have made no attempt to enlarge the class of test functions for Theorem  \ref{maingl3gl2} since this is not the focus of this article (but is certainly doable by more refined analysis). The regularity assumptions of $\mathcal{C}_{\eta}$ essentially follow from those of the Kontorovich-Lebedev inversion (see Section \ref{prelimwhitt}). As in \cite{GK13, GSW21, GSW23+,  Bu20}, the class $\mathcal{C}_{\eta}$ already includes good test functions that are useful in a number of applications and allows us to deduce a version of Theorem \ref{maingl3gl2} for incomplete $L$-functions (see Remark \ref{testfunc}).

Also, we have obtained the analytic properties and several explicit expressions for the integral transform  $\left(\mathcal{F}_{\Phi} H\right)\left(s_{0},  s \right)$.   They are written in terms of Mellin-Barnes integrals or hypergeometric functions as in \cite{Mo93, Mo97}.  For ease of expositions, we do not record the full formulae here but refer the readers  to Section \ref{expevatra} for  the detailed discussions.  However, we record an interesting identity of special functions as follow:

	\begin{thm}[Theorem \ref{secMTcom}]\label{CFKGamma}
		For  $\frac{1}{2}+ \frac{1}{100}  <  \sigma  < 1$, we have
			\begin{align}
			\left(\mathcal{F}_{\Phi} H\right)\left(2s-1,  s \right) 
			\ &= \ \pi^{\frac{1}{2}-s} \  \prod_{i=1}^{3} \ \frac{\Gamma\left(s-\frac{1}{2}+ \frac{\alpha_{i}}{2}\right)}{\Gamma\left(1-s- \frac{\alpha_{i}}{2}\right)} \cdot \int_{(0)} \  \frac{H(\mu)}{\left| \Gamma(\mu)\right|^2}   \cdot \prod\limits_{i=1}^{3}  \ \prod\limits_{\pm} \  \Gamma\left( \frac{1-s+ \alpha_{i}\pm \mu}{2}\right) \ \frac{d\mu}{2\pi i}. 
			\end{align}
		\end{thm}
	
	% In fact,  the four identities  in Theorem \ref{CFKGamma} will be shown to  
	
	There are actually two additional identities of Barnes type that account for the origins and the combinatorics of six (out of eight) of the off-diagonal main terms for the cubic moment of $GL(2)$ $L$-functions. The results align nicely with the predictions of the \textit{`Moment Conjecture'} (or  \textit{`Recipe'}) of \cite{CFKRS05}.   We refer the interested readers to our  papers \cite{Kw23a+, Kw23b+}.

	\subsection{Follow-up Works}
	
	The current work aims to illustrate the key ideas and address the main analytic issues of our period  integral approach. It is the simplest to illustrate all these using the cuspidal case for $\Phi$.  However, this is  by no means the end of the scope of our method.  In our upcoming works \cite{Kw23a+, Kw23b+} (some parts are contained in the arxiv preprint \cite{Kw23}), we demonstrate the versatility of our method by: 
	\begin{enumerate}
		\item  	 Providing a new proof of the cubic moment identity (\ref{basicmoto}) (actually for the more general  \textit{`shifted moment'}) with a number of technical advantages, as well as a new unified way of extracting the full set of main terms.  There are  considerable recent interests in understanding the deep works of \cite{Mo93, Mo97} and \cite{CI00} from different perspectives, e.g.,  Nelson \cite{Ne20+}, Wu \cite{Wu21+}, Balkanova-Frolenkov-Wu \cite{BFW21+}.   %in line with the set-up of \cite{CFKRS05}
		
		\item  Establishing  a Motohashi's formula of $GL(3)$ in the non-archimedean aspect which  dualizes   $GL(2)$ twists  of Hecke eigenvalues into  $GL(1)$ twists by Dirichlet characters.  This should offer insights into the celebrated works of  Young \cite{Y11} and Blomer et. al.  \cite{BHKM20} on the fourth moment of Dirichlet $L$-functions.  In their works, this kind of change of structures was the result of a long sequence of spectral/ harmonic transformations and it was surprising (and useful) to observe such a nice  phenomenon.

	  %(pp. 45 therein)
	\end{enumerate}

 	\section{Outline}\label{prelim}
 	In Section \ref{fea}, we discuss the technical features of the method used in this article and draw comparisons with the current literature. In Section \ref{sketch},  we include a sketch of our arguments to demonstrate the essential ideas of our method  and sidestep the technical points. In Section \ref{prel}, we collect the essential notions and results for later parts of the article. 
 	
 	The proof of Theorem \ref{maingl3gl2} is divided into four sections. In Section \ref{Basicide}, we  prove the key identity of this article (see Corollary 
 	\ref{incomexpf}). In Section \ref{separaOD},  we develop such an identity into moments of $L$-functions on the  region of absolute convergence. In particular,  the intrinsic structure of the problem allows one to easily see the shape of the dual moment (see Proposition \ref{structure}).   In Section \ref{Stirl}, we obtain the region of holomorphy and growth of the archimedean transform. In Section \ref{2stepana}, a step-by-step analytic continuation argument is performed based on the analytic information obtained in Section \ref{Stirl}.  
 	
 	In Section \ref{expevatra}, we prove Theorem \ref{CFKGamma}. and  provide several explicit formulae of the integral transforms.

   %This is possible because unipotent integration and method of analytic continuation serve as  the basis of our approach.   Moreover,  the absolute convergence of the period integral and the two Rankin-Selberg unfoldings remain in place when $\re s\gg 1$. 

%%%%%%%%%%%%%%%%%%%%%%%%%%%%%%%%%%%%%%%%%%%%%%%%%%%%%%%%%%%%%%%%%%%%%%%%%%%%%%%%%%%%%%%%%%%%%%%%%%%%%%%%%%%%%%%%%%%%%%%%%%%%%%%%%%%%%%%%%%%%%%%%%%%%%%%%%%%%%%%%%%%%%%%%%

\section{Technical Features of Our Method}\label{fea}

\subsection{Period Reciprocity}

Our work adds a new instance to the recent banner  \textit{`Period Reciprocity'} which aims at revealing the underlying  structures of moments of $L$-functions through the lenses of period integrals.  The general philosophy of this method is to evaluate a period integral  in two distinct manners. Under favorable circumstances, the intrinsic structures of period integrals would lead to interesting, non-trivial moment identities, say connecting two different-looking families of $L$-functions. 

In our case, the generalized Motohashi-type phenomenon of Theorem \ref{maingl3gl2} at $s=1/2$ will be shown to be an intrinsic property of a given Maass cusp form  $\Phi$ of $SL_{3}(\Z)$ via  the following  trivial identity 
\begin{align}\label{trivstart}
\int_{0}^{1} \left[ \int_{0}^{\infty} \Phi\begin{pmatrix} \begin{pmatrix}
		y_{0} & \\
		          & y_{0}
	\end{pmatrix}\begin{pmatrix}
1 & u\\
& 1
\end{pmatrix} & \\
& 1\end{pmatrix}  \ d^{\times} y_{0}\right] \ & e(-u)  \ du  \nonumber\\
\ & \hspace{-80pt} = \   \int_{0}^{\infty} \left[\int_{0}^{1}  \Phi\begin{pmatrix}\begin{pmatrix}
1 & u\\
& 1
\end{pmatrix} \begin{pmatrix}
y_{0} & \\
& y_{0}
\end{pmatrix} & \\
& 1\end{pmatrix} e(-u) \  du\right] \ d^{\times} y_{0}.
\end{align}
Roughly speaking,  Theorem \ref{maingl3gl2}  follows from (1). spectrally-expanding the innermost integral on the left in terms of a basis of $GL(2)$ automorphic forms, and  (2). computing the innermost integral on the right in terms of the $GL(3)$ Fourier-Whittaker period.  A sketch of this will be provided in Section \ref{sketch}. In practice, it turns out to be convenient to work with a more general set-up
\begin{align}\label{mainobj}
	\int_{SL_{2}(\Z)\setminus GL_{2}(\R)} \ P(g;h) \Phi\begin{pmatrix}
		g & \\
		& 1
	\end{pmatrix} |\det g|^{s-\frac{1}{2}} \ dg
\end{align}
 so as to  bypass certain technical difficulties, where $P(*;h)$ is a Poincar\'e series of  $SL_{2}(\Z)$. 
%(with a `nice' test function $h$)

%One must admit that 

The current examples for Period Reciprocity occur rather sporadically and we do not have a systematic way to construct new examples yet. Also, techniques differ greatly in each known instance (see \cite{MV06, MV10, Ne20+}, \cite{Bl12a}, \cite{Nu20+}, \cite{JN21+}, \cite{Za21, Za20+}). This marks a stark contrast with the more traditional  `Kuznetsov-Voronoi' framework (see Section \ref{comCI}).  However, Period Reciprocity seems to address some of the technical complications more softly than the Kuznetsov-Voronoi approach. We shall elaborate more in the upcoming subsections.   

%a number of vs/ some of the

Regarding the  `classical' Motohashi phenomenon (\ref{basicmoto}), there was  the Michel-Venkatesh strategy \cite{MV06, MV10}  and Nelson \cite{Ne20+} very recently developed such a regularized period method fully and rigorously with new inputs from automorphic representations. This article provides another strategy that includes (\ref{basicmoto}) together with several generalized instances of such phenomenon. We address the structural and analytic aspects of the formulae rather differently using unipotent integration for $GL(3)$ and method of analytic continuation. (So we would work with (\ref{mainobj}) initially for  $\re s  \gg   1$.)  For further discussions, see Section \ref{sketch}.

We would also like to mention the works of  Wu \cite{Wu21+}  and  Balkanova-Frolenkov-Wu \cite{BFW21+}   in which  an interesting framework in terms of tempered distributions and  relative trace formula of Godement-Jacquet type was developed to address the phenomenon (\ref{basicmoto}). %Notice that ramifications are allowed in \cite{Ne20+}, \cite{Wu21+, BFW21+}. 

%An important feature of our method  is that we are able to uncover the dual moment, i.e., the right side of (\ref{momaid}),   quickly and naturally thanks to  the structural advantages provided by the  period integral: 

\subsection{Comparisons with the Conrey-Iwaniec-Li Method}\label{comCI}

The celebrated works of Conrey-Iwaniec \cite{CI00} and Li \cite{Li09, Li11} are known for their successful analysis based on the Kuznetsov trace formulae and summation formulae of Poisson/ Voronoi type.  Their accomplishments include  the  delicate treatment of the arithmetic of exponential sums as well as the stationary phase analysis.

The Kuznetsov trace formula (or more generally the relative trace formula) has been a cornerstone in the analytic theory of $L$-functions  during the past few decades. For $PGL_{2}(\Z)\setminus PGL_{2}(\R)$ (i.e., the context of Theorem \ref{maingl3gl2}, summing over a basis of \textit{even} Maass forms of $SL_{2}(\Z)$), it is an equality of the shape
\begin{align}
	\sum_{j} \  H(\mu_{j}) \ \frac{\lambda_{j}(n) \overline{\lambda_{j}(m)}}{L(1, \text{Ad}^2 \phi_{j})} + (\text{cts}) \ = \  \delta_{m=n} \int_{\R} \ H(\mu) \ d_{\text{spec}} \mu \ + \ \sum_{\pm}\sum_{c} \ \frac{S(\pm m,n;c)}{c} \mathcal{J}^{\pm}\left(\frac{4\pi\sqrt{mn}}{c}\right). 
\end{align}
between the spectral bilinear form of Hecke eigenvalues and the geometric expansion consisting of Kloosterman sums $S(m,n;c)$ and  oscillatory integrals $\mathcal{J}^{+}$ and $\mathcal{J}^{-}$ involving the $J$-Bessel and $K$-Bessel function in their kernels respectively.  These two pieces have to be treated separately.

As noticed by \cite{CI00, Li09, Li11, Bl12b} and a number of subsequent works, the $J$-Bessel piece turns out to be rather interesting ---  it contains  remarkable technical features that are crucial in gaining sufficient cancellations in geometric sums and integrals. This  seems to be distinctive to the settings  of higher-rank. (In view of this, readers may wish to  compare with  Liu-Ye \cite{LY02}'s analysis in the $GL(2)$ settings.) More concretely, Li \cite{Li11} was able to apply the $GL(3)$  Voronoi formula \textit{twice}, which were surprisingly  non-involutary,  because of a subtle cancellation taking place between the \textit{arithmetic phase} coming from Voronoi and the \textit{analytic phase} coming from the $J$-Bessel transform. 

%incorporate extra \textit{shifts} to the central point and 

% the  $J$-Bessel piece of  Kuznetsov  \textit{cannot} be dispense with

  In this set of method, the treatment of the  $J$-Bessel piece of Kuznetsov  is essential especially if one wishes to handle moments of $L$-functions of greater generality, say  for $\Phi$ not necessarily self-dual or for non-central $L$-values (e.g.,  $s=1/2+i\tau$) as  in Theorem \ref{maingl3gl2}.

   In our period integral approach, the Kuznetsov formula, the Voronoi formula, and the approximate functional equation,  which belong to the standard toolbox in analytic number theory, are completely avoided altogether. Indeed, there are  conceptual reasons for all these as we now explain: 
   
   \begin{itemize}
   \item 	Firstly,  since the  $GL(3) \times GL(2)$ $L$-functions on the spectral side are interpreted as period integrals,  we never need to open up those $L$-functions in the form of Dirichlet series. As a result, we do not need to average over the  Hecke eigenvalues of our basis of  $GL(2)$ Maass forms using the Kuznetsov formula. 
   	
   \item Secondly, as shown in our  moment identity (\ref{momaid}),  the dual $GL(3)$ object turns out to be the \textit{standard $L$-function}.  The construction of  the standard $L$-function  involves only the $GL(3)$ Hecke eigenvalues,  whereas  the $GL(3)$ Voronoi formula is known for the presence of \textit{general} Fourier coefficients of $GL(3)$ due to the arithmetic twisting. It is thus reasonable to expect a proof of (\ref{momaid})  that does not use the $GL(3)$ Voronoi formula (nor the full Fourier expansions of \cite{JPSS}). The set-up (\ref{trivstart}) already suggests that  our method meets such an expectation, but see Proposition \ref{incomf} for full details. 
   	
   \item	Thirdly,  we do not encounter any intermediate exponential sums (e.g., Kloosterman/ Ramanujan sums), slow-decaying/ very oscillatory special functions, nor shifted convolution sums  (which are necessary components in \cite{Iv01, Iv02, Fr20} for (\ref{basicmoto})). Also, we handle the archimedean component of  (\ref{momaid}) in one piece (instead of handling the $J$- and $K$-Bessel pieces separately in \cite{CI00, Li09, Li11}) and we directly work with the $GL(3)$ Whittaker function associated to the automorphic form $\Phi$. 
   
   \item Fourthly, we take advantage of the equivariance of the Whittaker functions under unipotent translations which helps to simplify  many formulae. 
   \end{itemize}

Not only did we gain many technical benefits  in our period integral approach, our approach is distinct from the Kuznetsov/ Voronoi approach  in nature.  Indeed, our approach is  \textit{local} and the key result Proposition \ref{incomf}  can be easily phrased in terms of adeles (see (\ref{adeles})), whereas the Kuznetsov/ Voronoi approach is \textit{global} and \textit{non-adelic}. In this article, we focus on the level $1$ case (and the spectral aspect) as a proof of concept and thus we use the classical language of real groups.  In our upcoming work, we wish to extend our approach in various non-archimedean aspects.

  %In fact, due to the degenerate structure of the $GL(3)\times GL(2)$ period, the transform we encounter ... Mellin? But compare with GL(2)*GL(2)---> can also take Mellin and get 2F1

%role of test functions?

%no dual $GL(2)$ components/ transforms unless explicate. 

 %%%%%%%%%%%%%%%%%%%%%%%%%%%%%%%%%%%%%%%%%%%%%%%%%%%%%%%%%%%%%%%%%%%%%%%%%%%%%%%%%%%%%%%%%%%%%%% %%%%%%%%%%%%%%%%%%%%%%%%%%%%%%%%%%%%%%%%%%%%%%%%%%%%%%%%%%%%%%%%%%%%%%%%%%%%%%%%%%%%%%%%%%%%%%% %%%%%%%%%%%%%%%%%%%%%%%%%%%%%%%%%%%%%%%%%%%%%%%%%%%%%%%%%%%%%%%%%%%%%%%%%%%%%%%%%%%%%%%%%%%%%%% %%%%%%%%%%%%%%%%%%%%%%%%%%%%%%%%%%%%%%%%%%%%%%%%%%%%%%%%%%%%%%%%%%%%%%%%%%%%%%%%%%%%%%%%%%%%%%% %%%%%%%%%%%%%%%%%%%%%%%%%%%%%%%%%%%%%%%%%%%%%%%%%%%%%%%%%%%%%%%%%%%%%%%%%%%%%%%%%%%%%%%%%%%%%%%

 \subsection{Prospects for Higher-Rank}
 
 Once we reach $GL(3)$, the geometric expansion for the Kuznetsov formula becomes substantially more involved  and presents a number of obstacles in generalizing the  Kuznetsov-based approaches to moments of $L$-functions of higher-rank:

   \begin{rem}[\textbf{Oscillatory Integrals}]\label{cha1}
 	In $GL(2)$,  a couple of coincidences  allow us to identify the oscillatory integrals with some well-studied special functions, see \cite{Mo97}, \cite{I02}. However,  such a phenomenon does not exist in $GL(3)$ and  there turn out to be many unexpected analytic difficulties, see Buttcane  \cite{Bu13, Bu16}.  The complicated formulae for the oscillatory integrals make the Kuznetsov trace formula for $GL(3)$  challenging to apply, see Blomer-Buttcane \cite{BlBu20}. 
 	
 \end{rem}

   %Upon further reflection, the author of this paper believes that the \textit{Bruhat decomposition} is a  source of complications.  Therefore, this prompts us  to  re-think the strategies towards moment problems. 

  %which is a component intrinsic to the Kuznetsov formula. The mentioned oscillatory integrals and exponential sums are the outcomes of such a decomposition, and 

 %Non-archimedean (non-adelic approach already!!!)
 
 \begin{rem}[\textbf{Kloosterman Sums}]\label{cha2}
 	    The $GL(3)$ Kloosterman sums, e.g.,
 	 \begin{align}\label{GL3Kloost}
 	 S(m_{1}, m_{2}, n_{1}, n_{2};  \ D_{1}, D_{2}) 
 	 \ &:= \  \sideset{}{^\dagger}{\sum}_{\substack{B_{1} \ (D_{1}), \  B_{2} \ (D_{2}) \\ C_{1} \ ( D_{1}), \ C_{2} \ (D_{2}) }} \   e\left( \frac{m_{1}B_{1}+n_{1}(Y_{1}D_{2}-Z_{1}B_{2})}{D_{1}}\right)   e\left( \frac{m_{2}B_{2}+n_{2}(Y_{2}D_{1}-Z_{2}B_{1})}{D_{2}}\right),
 	 \end{align}
 	 are clearly much harder to work with than the usual one, where the definitions of  $Y_{i}, Z_{i}$'s along with a couple of  congruence and coprimality conditions are suppressed.  There are two other Kloosterman sums for $GL(3)$. See \cite{Bu13} for details.   
 	\end{rem}

As already mentioned in Section \ref{comCI}, further transformations of the exponential sums from the Kuznetsov formulae encode important arithmetic information of the moment of $L$-functions in question.  Blomer-Buttcane \cite{BlBu20} has provided an instance when this can be done for (\ref{GL3Kloost}) (after a four-fold Poisson summation!), but other than that it remains unclear what are the useful manipulations of (\ref{GL3Kloost}) in general.  On the other hand, applications of Voronoi formulae for $GL(3)$ (see \cite{CI00, Li09, Li11, Bl12b, BK19a, BK19b}) and for $GL(4)$ (see \cite{BLM19, CL20}) are currently limited to the usual Kloosterman sums of $GL(2)$, but already the situations get tricky very quickly.

%delicate. complicated tricky

 \begin{comment}
  		\begin{enumerate}
  			\item While the explication of the Kloosterman sums is  doable \textit{in principle}  (see Chapter 11 of \cite{Gold}), the end results are inevitably involved.  In fact,  the task of getting good bounds for these Kloosterman sums is  already far from trivial, see \cite{GSW21} (Appendix B)  and \cite{M21}.
  		\end{enumerate}
  		\end{comment}

 Conceptually speaking, the objects in Remark  \ref{cha1}-\ref{cha2} and the associated  issues  are caused by the    \textbf{Bruhat decomposition},  which is fundamental to the framework of relative trace formulae in general. However, ideas from Period Reciprocity have offered ways to bypass the Bruhat decomposition and any geometric sums and integrals, which is certainly a welcoming feature.

 Regarding Remark \ref{cha1}, the advantages of our method are visible even in the context of Theorem \ref{maingl3gl2}.  Although we work over the group $GL(3)$ on the dual side, the oscillatory factor in our approach (see (\ref{secondcc})) is actually simpler than the ones encountered  in the `Kuznetsov-Voronoi' approaches (cf. \cite{Li11}) and is more structured in the sense that (1). it  arises naturally from the definition of the archimedean Whittaker function, and (2).  it serves as an important constituent of the exact Motohashi structure, the exact structures of the main terms predicted by \cite{CFKRS05}, as well as for the analytic continuation past  $\re s=1/2$.  
Furthermore, our approach is devoid of integrals over non-compact subsets of the unipotent subgroups (or  the complements)   which are known to result in intricate dual calculations and exponential phases in  case of  $GL(3)$ Voronoi formula (cf. Section 4 of \cite{IT13}) and Kuznetsov formulae (cf. Chapter 11 of \cite{Gold}). 
 
 % because of the final Motohashi/ CFKRS structure L-functions/ automorphic nature. 
 %cannot be avoided in the respective/individual.

It is worth pointing out the crucial archimedean ingredient in our proof generalizes to  $GL(n)$. It is known as \textit{Stade's formula} (see \cite{St01}), which allows us to rewrite the archimedean part  completely in terms of integrals $\Gamma$-functions.  It turns out to be sufficient to work with such representation for our purposes. Stade's formula  possesses remarkable recursive structures which are useful for further analytic manipulations. We carry out such calculations in the last part of Section \ref{expevatra}.  Another notable  recent  application of Stade's formula can be found in  \cite{GSW21, GSW23+}. We hope that our method will also shed light on the origins, constituents and structures of the archimedean transforms, as well as open up generalizations to moments of higher-rank  (which should sidestep the technical difficulties illustrated above for the `Kuznetsov-Voronoi' method). We shall return to this subject in our upcoming works, together with treatment of the non-archimedean places.

 	\section{Informal Sketch and Discussion}\label{sketch}

 	To assist the readers, we illustrate in a simple fashion  the main ideas of this article in this section before diving into any of the analytic subtleties of our actual argument. In fact, this represents the most intrinsic picture of our method and will facilitate comparisons with the strategy of  Michel-Venkatesh \cite{MV06} along the way. The style of this section will be largely informal ---  we shall suppress the absolute constant multiples (say those $2$'s and $\pi$'s), pretend everything converges, and ignore the treatment of the main terms.

 	We begin by recalling the idea outlined in  \cite{MV06}. The classical Motohashi formula  can be understood as an  intrinsic property of the  $GL(2)$ Eisenstein series (denoted by $E^{*}$ below) via the (`regularized') geodesic period 
 	 \begin{align*}
 	 \int_{0}^{\infty} \ \left|E^{*}(iy)\right|^2 \ d^{\times} y,
 	 \end{align*}
 	 which can be evaluated in two ways according to $|E^{*}|^2$ and  $E^{*} \cdot \ovA{E^{*}}$ respectively: 
 	 \begin{enumerate}
 	 	\item  ($GL(2)$ spectral expansion)
 	 	\begin{align}\label{MVspec}
 	 	\sum_{\phi: \ GL(2)} \ \left\langle \ |E^{*}|^2, \phi \right\rangle \ \int_{0}^{\infty} \ \phi(iy) \ d^{\times} y \ = \   \sum_{\phi: \ GL(2)}\  \Lambda\left(\frac{1}{2}, \phi\right)^2\cdot \Lambda\left(\frac{1}{2}, \phi\right)  \ + \  (\cdots)
 	 	\end{align}

 	 	\item ($GL(1)\times GL(1)$ expansion, or the Mellin-Plancherel formula)
 	 	\begin{align}\label{MVmellin}
 	 	\int_{(1/2)} \ \left|\widetilde{E^{*}}(s)\right|^2 \ \frac{ds}{2\pi i} \ = \ \int_{\R} \ \left|\Lambda\left(\frac{1}{2}+it\right)^2\right|^{2} \ \frac{dt}{2\pi}. 
 	 	\end{align}
 	 	
 	 \end{enumerate}
 	 This seemingly simple sketch  turns out to require rather sophisticated regularizations but  was   skillfully executed  by Nelson \cite{Ne20+} very recently.

 	 We now turn to our sketch of the (generalized) Motohashi  phenomenon  as  described in Theorem \ref{maingl3gl2}.  Let  $\Phi$  be a   Maass cusp form of $SL_{3}(\Z)$. As already mentioned in the introduction, our starting point is  the trivial identity 
 	 \begin{align}\label{trivbasid}
 	 \int_{0}^{1} \left[ \int_{0}^{\infty} \Phi\begin{pmatrix}y_{0}\begin{pmatrix}
 	 1 & u\\
 	 & 1
 	 \end{pmatrix} & \\
 	 & 1\end{pmatrix}  \ d^{\times} y_{0}\right] \ e(-u)  \ du \ = \   \int_{0}^{\infty} \left[\int_{0}^{1}  \Phi\begin{pmatrix}\begin{pmatrix}
 	 1 & u\\
 	 & 1
 	 \end{pmatrix} y_{0} & \\
 	 & 1\end{pmatrix} e(-u) \  du\right] \ d^{\times} y_{0}. 
 	 \end{align}
 	For better symmetry,   it is not hard to observe that the right side of (\ref{trivbasid}) can be written as 
 	 \begin{align}\label{dualexpr}
 	 \int_{0}^{\infty} \ \left[ \int_{0}^{1}  \widetilde{\Phi} \left[\begin{pmatrix}
 	 1 &  &\\
 	 & 1   & u \\
 	 &      & 1  
 	 \end{pmatrix} 
 	 \begin{pmatrix}
 	 y_{0} & & \\
 	 & 1& \\
 	 &   & 1
 	 \end{pmatrix} \right] e(-u) \  du \right] \ d^{\times} y_{0}
 	 \end{align}
 	with  $\widetilde{\Phi}(g):= \Phi(^{t}g^{-1})$  being  the dual form of $\Phi$.

 	 \begin{rem}\label{dualcal}
 	 	Indeed, the  center-invariance of  $\Phi$ implies that
 	 	\begin{align*}
 	 	(\ref{trivbasid}) \ = \  \int_{0}^{\infty} \int_{0}^{1}  \Phi \left[\begin{pmatrix}
 	 	1 & u &\\
 	 	& 1   & \\
 	 	&      & 1
 	 	\end{pmatrix} 
 	 	\begin{pmatrix}
 	 	1 & & \\
 	 	& 1& \\
 	 	&   & y_{0}
 	 	\end{pmatrix}\right] e(-u) \  du \ d^{\times} y_{0}. 
 	 	\end{align*}
 	 	Let $w_{\ell}:= \begin{psmallmatrix}
 	 	& & -1\\
 	 	& 1& \\
 	 	1 &    &	
 	 	\end{psmallmatrix}$.  The observation 
 	 	\begin{align*}
 	 	\begin{pmatrix}
 	 	1 & & \\
 	 	& 1& \\
 	 	&   & y_{0}
 	 	\end{pmatrix} \ = \ w_{\ell}^{-1}
 	 	\begin{pmatrix}
 	 	y_{0} & & \\
 	 	& 1& \\
 	 	&   & 1
 	 	\end{pmatrix} w_{\ell} \hspace{15pt} \text{ and } \hspace{15pt} \begin{pmatrix}
 	 	1 &  &\\
 	 	& 1   & \\
 	 	&    -u  & 1  
 	 	\end{pmatrix}  \ = \  w_{\ell} \begin{pmatrix}
 	 	1 & u &\\
 	 	& 1   & \\
 	 	&      & 1
 	 	\end{pmatrix}  w_{\ell}^{-1}
 	 	\end{align*}
 	 	together with the left and right invariance of $\Phi$ by $w_{\ell}$ further rewrite (\ref{trivbasid}) as
 	 	\begin{align*}
 	 	\int_{0}^{\infty} \int_{0}^{1}  \ \Phi \left[\begin{pmatrix}
 	 	1 &  &\\
 	 	& 1   & \\
 	 	&    -u  & 1  
 	 	\end{pmatrix} 
 	 	\begin{pmatrix}
 	 	y_{0} & & \\
 	 	& 1& \\
 	 	&   & 1
 	 	\end{pmatrix} \right] &  e(-u) \  du \ d^{\times} y_{0} \nonumber\\
 	 	\ &= \ 	\int_{0}^{\infty} \int_{0}^{1}  \ \widetilde{\Phi} \left[\begin{pmatrix}
 	 	1 &  &\\
 	 	& 1   & u \\
 	 	&      & 1  
 	 	\end{pmatrix} 
 	 	\begin{pmatrix}
 	 	y_{0} & & \\
 	 	& 1& \\
 	 	&   & 1
 	 	\end{pmatrix} \right] e(-u) \  du \ d^{\times} y_{0}. 
 	 	\end{align*}
 	 \end{rem}

 	 %%%%%%%%%%%%%%%%%%%%%%%%%%%%%%%%%%%%%%%%%%%%%%%%%%%%%%%%%%%%%%%%%%%%%%%%%%%%%%%%%%%%%%%%%%%%%%%%%%%%%%%%%%%%%%%%%%%%%%%%%%%%%%%%%%%%%%%%%%%%%%%%%%%%%%%%%%%%%%%%%%%%%

 	 As an overview of our strategy,
 	 \begin{enumerate}
 	 	\item \textbf{Similar to} Michel-Venkatesh's strategy, the integral over $(0, \infty)$ (or the center $Z_{GL_{2}}^{+}(\R)$ in view of the structure of the problem) will give rise to  certain complete $L$-functions upon spectral-expanding and will turn into  the dual $t$-integral;

 	 	\item \textbf{Different from} Michel-Venkatesh's strategy, an extra integral over $[0,1]$ (or the quotient $U_{2}(\Z)\setminus U_{2}(\R)$ of the unipotent subgroup of $GL(2)$) is present in our construction which will  lead to  Whittaker functions as the spectral weight functions  and (essentially) a product of two distinct $L$-functions on the dual side; 
 	 	
 	 	\item \label{harmon} The Mellin-Plancherel of (\ref{MVmellin}) is \textbf{replaced by} two Fourier expansions over $\Z\setminus \R$ below. 
 	 \end{enumerate}
  In fact,  the \textit{unipotent} nature of our period method is crucial in realizing the spectral duality for the fourth moment of Dirichlet $L$-functions (see \cite{Kw23b+}), as well as in ensuring the abundance of admissible test functions on the spectral side,  but these features will not be displayed in this section.

 	 \subsection{ The $GL(2)$ (spectral) side} 
 	 This side is relatively straight-forward and gives the desired $GL(3)\times GL(2)$ moment. Regard $\Phi$ as a function of $L^{2}(\Gamma_{2}\setminus \mathfrak{h}^{2})$ via
 	 \begin{align*}
 	 (\Proj_{2}^{3} \Phi)(g) \ := \ \int_{0}^{\infty}  \Phi\begin{pmatrix}
 	 y_{0} g & \\
 	 & 1
 	 \end{pmatrix}  \ d^{\times} y_{0} \hspace{30pt}  (g \ \in  \ \mathfrak{h}^{2}),
 	 \end{align*} 
 	 which in turn can be expanded spectrally as
 	 \begin{align*}
 	 (\Proj_{2}^{3} \Phi)(g) \ = \  \sum_{j} \ \frac{\langle  \Proj_{2}^{3}\Phi \ ,   \phi_{j} \rangle}{||\phi_{j}||^{2}} \ \phi_{j}(g) \ + \    \frac{\langle  \Proj_{2}^{3}\Phi \ ,  1 \rangle}{||1||^{2}}  \cdot 1 \ + \  (\text{cont}).
 	 \end{align*}
  The spectral coefficients $\langle  \Proj_{2}^{3}\Phi \ ,   \phi_{j} \rangle$ are precisely  the  $GL(3)\times GL(2)$ Rankin-Selberg $L$-functions.  Hence,
 	 \begin{align}
 	 \text{LHS of } \ (\ref{trivbasid}) \ = \ 	\int_{0}^{1} \ \left(\Proj_{2}^{3} \Phi\right)
 	 \begin{pmatrix}
 	 1 & u \\
 	 & 1
 	 \end{pmatrix} e(-u) \ du
 	 \ = \ \sum_{j} \ W_{\mu_{j}}(1)\cdot  \frac{\Lambda\left(1/2, \phi_{j}\otimes \Phi\right)}{||\phi_{j}||^2} \ + \ (\text{cont}),
 	 \end{align}
 	 where  $\mu\mapsto W_{\mu}(1)$ is a certain kind of weight function.

 	 \begin{comment} %%%% Conversion of conventions
 	 \begin{rem}
 	 If  $ s \gg 1$ and $\Phi$ is Hecke-Maass, then
 	 \begin{align}
 	 L\left(s, \phi_{j} \otimes \widetilde{\Phi}\right) \ := \  \sum_{m_{1}, m_{2}} \ \frac{\mathcal{B}_{j}(m_{2}) \mathcal{B}_{\widetilde{\Phi}}(m_{1}, m_{2})}{(m_{1}^2 m_{2})^{s}} \ = \  \sum_{m_{1}, m_{2}} \ \frac{\mathcal{B}_{j}(m_{2}) \ovA{\mathcal{B}_{\Phi}(m_{1}, m_{2})}}{(m_{1}^2 m_{2})^{s}}, 
 	 \end{align}
 	 and thus, 
 	 \begin{align}
 	 \ovA{L\left(s, \phi_{j} \otimes \widetilde{\Phi}\right)} \ = \   \sum_{m_{1}, m_{2}} \ \frac{\mathcal{B}_{j}(m_{2}) \mathcal{B}_{\Phi}(m_{1}, m_{2})}{(m_{1}^2 m_{2})^{s}}  \ := \ L(s, \phi_{j} \otimes \Phi).  
 	 \end{align}
 	 \end{rem}
 	 \end{comment}

 	 \subsection{The $GL(1)$ (dual) side}
 	 
 	 In view of  Point (\ref{harmon}) above,   we evaluate the innermost integral of  (\ref{dualexpr}) in terms of the \textbf{Fourier-Whittaker periods} for $\widetilde{\Phi}$ (denoted by $(\widehat{\widetilde{\Phi}})_{(\cdot, \cdot)}$, see Definition \ref{fourcoeff}).  Indeed, (\ref{dualexpr}) is given by 
 	 \begin{align}\label{intrexp}
 	 \int_{0}^{\infty} \int_{0}^{1} \int_{0}^{1} &  \ \widetilde{\Phi}\left[ \begin{pmatrix}
 	 1 &  & u_{1,3}\\
 	 & 1   & u_{2,3} \\
 	 &      & 1  
 	 \end{pmatrix} 
 	 \begin{pmatrix}
 	 y_{0} & & \\
 	 & 1& \\
 	 &   & 1
 	 \end{pmatrix}\right] \   e(-u_{2,3}) \ du_{1,3} \ du_{2,3} \ d^{\times} y_{0} \nonumber\\
 	 \ &\hspace{40pt} + \ \sum_{a_{0}\in \Z-\{0\}} \ \sum_{a_{1}\in \Z-\{0\}}  \ \int_{0}^{\infty} \ \left(\widehat{\widetilde{\Phi}}\right)_{(1, a_{1})}\left[\begin{pmatrix}
 	 1       & & \\
 	 a_{0} & 1 & \\
 	 &  & 1
 	 \end{pmatrix} \begin{pmatrix}
 	 y_{0} & & \\
 	 & 1& \\
 	 &   & 1
 	 \end{pmatrix}  \right] \ d^{\times} y_{0}.
 	 \end{align}
  (See Proposition \ref{incomf}.)

 	 %(cf. Chapter 6.5 of \cite{Gold}) 
 
 	 The first line of (\ref{intrexp}) contributes to the diagonal and is precisely the integral representation of the standard $L$-function of $\widetilde{\Phi}$. It is equal to $L\left(1, \widetilde{\Phi}\right)  \cdot  Z_{\infty}\left(1, \widetilde{\Phi}\right)$, where  $Z_{\infty}( \ \cdot \  , \widetilde{\Phi})$ is the $GL(3)$ local zeta integral at $\infty$. The second line of  (\ref{intrexp}) is the off-diagonal (denoted by $OD_{\Phi}$ below).  In terms of  Fourier coefficients of $\widetilde{\Phi}$, 
 	 	\begin{align}\label{adeles}
 	OD_{\Phi} \ = \   \sum_{a_{0}\in \Z-\{0\}} \ \sum_{a_{1}\in \Z-\{0\}} \ \frac{\mathcal{B}_{\widetilde{\Phi}}(1,a_{1})}{|a_{1}|}   \ \int_{0}^{\infty} \   \left(\widehat{\widetilde{\Phi}}\right)_{(1,1)}\left[
 	 	\begin{pmatrix}
 	 		a_{1}/a_{0}  &   &     \\
 	 		1            & 1 &      \\
 	 		&    &  1
 	 	\end{pmatrix}
 	 	\begin{pmatrix}
 	 		y_{0} &   & \\
 	 		& 1 & \\
 	 		&    & 1
 	 	\end{pmatrix}\right] \ d^{\times} y_{0}. 
 	 \end{align}
 	  It can be further explicated as
 	 \begin{align}
 	OD_{\Phi} \ = \   \sum_{a_{0}\in \Z-\{0\}} \ \sum_{a_{1}\in \Z-\{0\}} \ \frac{\mathcal{B}_{\widetilde{\Phi}}(1,a_{1})}{|a_{1}|}  \cdot \int_{0}^{\infty} \ 
 	 W_{\alpha(\Phi)}\left( \left| \frac{a_{1}}{a_{0}}\right|\frac{y_{0}}{1+y_{0}^2}, \  1\right) \cdot e\left( \frac{a_{1}}{a_{0}}\frac{y_{0}^2}{1+y_{0}^2}\right) \ d^{\times} y_{0} \label{offdiagsketch}
 	 \end{align}
  using the spherical Whittaker function $W_{\alpha(\Phi)}$, where the oscillatory factor $e(\cdots)$ originates from the unipotent translate  of Whittaker functions.  
  
  Roughly speaking,    (\ref{offdiagsketch}) suggests some forms of (multiplicative) convolutions between the $GL(3)$ and $GL(1)$ data at both the archimedean and the non-archimedean places: 
 	 
 	 \begin{comment}
 	 \begin{rem}
 	 	If one uses the adelic language and let $W_{\widetilde{\Phi}}$ be the global Whittaker function of $\widetilde{\Phi}$, then $OD_{\Phi}$ can be more compactly written as
 	 	\begin{align}\label{adeles}
 	 	\sum_{a_{0} \in \Q^{\times}} \sum_{a_{1}\in \Q^{\times}} \ \int_{\Q^{\times} \setminus \A^{\times}} W_{\widetilde{\Phi}}\left[
 	 	\begin{pmatrix}
 	 	a_{1}/a_{0}  &   &     \\
 	 	1            & 1 &      \\
 	 	&    &  1
 	 	\end{pmatrix}
 	 	\begin{pmatrix}
 	 	y_{0} &   & \\
 	 	& 1 & \\
 	 	&    & 1
 	 	\end{pmatrix}\right] \ d^{\times} y_{0}
 	 	\end{align}
 	 	in which the occurence of the ratio $a_{1}/a_{0}$ in (\ref{offdiagsketch}) is more natural. 
 	 \end{rem}
 	  \end{comment}

 	 \begin{comment}
 	 \begin{align}
 	 \int_{\Q\setminus \A} \   \widetilde{\Phi} \left[\begin{pmatrix}
 	 1 &  &\\
 	 & 1   & u \\
 	 &      & 1  
 	 \end{pmatrix} 
 	 \begin{pmatrix}
 	 a_{0}^{-1}y_{0} & & \\
 	 & 1& \\
 	 &   & 1
 	 \end{pmatrix} \right] & \psi(-u) \  du \nonumber\\
 	 \ &\hspace{-60pt} = \  \sum_{a_{0} \in \Q^{\times}} \sum_{a_{1}\in \Q^{\times}} \   W_{\widetilde{\Phi}}\left[\begin{pmatrix}
 	 a_{1} & & \\
 	 & 1 & \\
 	 &    & 1
 	 \end{pmatrix}
 	 \begin{pmatrix}
 	 1       & & \\
 	 a_{0} & 1 & \\
 	 &  & 1
 	 \end{pmatrix} \begin{pmatrix}
 	 a_{0}^{-1}y_{0} & & \\
 	 & 1& \\
 	 &   & 1
 	 \end{pmatrix} 
 	 \right] \nonumber\\
 	 \end{align}
 	 \end{comment}
  
   %In expression (\ref{offdiagsketch}), 

 	 \begin{enumerate}
 	 	\item (Archimedean) We apply Mellin inversion to $W_{\alpha(\Phi)}$ (a standard result for the $GL(3)$ theory) and the local functional equation of  $GL(1)$ of the form
 	 	\begin{align}
 	 	e(x)+ e(-x) \ = \     \int_{-i\infty}^{i\infty} \   \frac{\Gamma_{\R}(u)}{\Gamma_{\R}\left(1-u\right)} \  |x|^{-u} \ \frac{du}{2\pi i} \hspace{20pt} (x \ \neq \ 0);
 	 	\end{align} 
 	 	
 	 	\item (Non-archimedean) Observe the following identity of the double Dirichlet series:  
 	 	\begin{align}
 	 	\sum_{a_{0}\neq 0} \sum_{a_{1}\neq 0} \frac{\mathcal{B}_{\Phi}(a_{1}, 1)}{|a_{1}|} \left| \frac{a_{1}}{a_{0}}\right|^{1-s_{0}-u} \ = \  L\left(s_{0}+u, \widetilde{\Phi}\right) \  \zeta\left(1-s_{0}-u\right).
 	 	\end{align}
 	 \end{enumerate}
 	 We thus arrive at
 	 \begin{align}
 	 OD_{\Phi}  \ = \  \int_{(1/2)} \  \zeta\left(1-s_{0}\right) L(s_{0},  \widetilde{\Phi}) \cdot  (\cdots)  \ \frac{ds_{0}}{2\pi i}, 
 	 \end{align}
 	 where `$(\cdots)$' stands for a certain integral transform can be described purely in terms of $\Gamma$-functions.

 	 \begin{rem}\label{BFWremk}
 	 \ 
 	 \begin{enumerate}
 	 	\item In  (\ref{mainobj}), the test function $h$ of the Poincar\'{e} series $P(*;h)$ will be  transformed into the Kontorovich-Lebedev transform $h^{\#}$ on the $GL(2)$ side  (see (\ref{comspec}))  and into the Mellin transform $\widetilde{h}$ on the $GL(1)$ side (see (\ref{doubmbtrans})). This is consistent with the sketch above. 
 	 	
 	 	\item  Readers may wish to compare the integral transforms obtained in the sketch with the one described in Section 1.3 of \cite{BFW21+}. 
 	 \end{enumerate}
 	 \end{rem}

 	 \begin{comment}
 	 \begin{rem}
 	 	Simply looking at the residue at $u=0$, which is  given by
 	 	\begin{align*}
 	 	\Gamma_{\R}\left( 1-s_{0}\right) \prod_{i=1}^{3} \ \Gamma_{\R}(s_{0}-\alpha_{i}(\Phi))\cdot \int_{-i\infty}^{i\infty} (\cdots) \ \frac{ds_{1}}{2\pi i},
 	 	\end{align*}
 	 	the  first four  $\Gamma$-factors happen to correspond to the completion factors for $ \zeta\left(1-s_{0}\right) L(s_{0},  \widetilde{\Phi})$.
 	 \end{rem}
 	 \end{comment}

 	  \begin{rem}
 	  	The choices of unipotent subgroups have been important  in the constructions of various  $L$-series for the group $GL(3)$:
 	  	\begin{itemize}
 	  	\item $\left\{  \begin{psmallmatrix}
 	  	1 &    &   * \\
 	  	& 1 &  * \\
 	  	&     &    1
 	  	\end{psmallmatrix} \right\}$ or  $\left\{  \begin{psmallmatrix}
 	  	1 &   * &   * \\
 	  	   & 1  &   \\
 	  	   &     &    1
 	  	\end{psmallmatrix} \right\}$ for the standard $L$-function;

 	  \item 	$\left\{\begin{psmallmatrix}
 	  	1 &   &   *\\
 	  	   & 1 &    \\
 	  	   &    &  1
 	  	\end{psmallmatrix}  \right\}$ for Bump's double Dirichlet series (\cite{Bump84});

 	  \item 	$\left\{ \begin{psmallmatrix}
 	  	1 &    &    \\
 	  	& 1 & *   \\
 	  	&     &    1
 	  	\end{psmallmatrix} \right\}$	or 
 	  	$\left\{ \begin{psmallmatrix}
 	  	1 & *   &    \\
 	  	& 1 &   \\
 	  	&     &    1
 	  	\end{psmallmatrix} \right\}$ for the Motohashi phenomenon of this article. 
 	  
 	  \end{itemize}
 	  \end{rem}
 	  
 	  %Readers might be interested in comparing the approach of  \cite{Li11} with ours. In \cite{Li11},  the $GL(3)$ Voronoi formula (\cite{GoLi06})    was used and the unipotent subgroup of interest was . In our case,  the natural unipotent subgroup (upon unfolding)  turns out to be the  complementary one:  . 
 	  
 	  %From (\ref{incomexpform}),   observe that the diagonal and off-diagonal terms are encapsulated simultaneously. 

 	 %Actually, Theorem \ref{maingl3gl2}  is an equality of two Rankin-Selberg unfoldings of $GL(2)$ in two distinct directions. This is possible because of the `disparity of ranks' present in our case.  

 	  % The first unfolding uses the Fourier expansion of the form $\Phi$, which has the effect of  `projecting'  $\Phi$   from $GL(3)$ to $GL(2)$,  and gives the desired spectral average for $GL(3)\times GL(2)$ $L$-functions. This is a $GL(2)$ calculation in essence. 
 	   
 	   %The second unfolding is performed using the definition of the Poincar\'e series $P(*;h)$. We are then led to a $GL(3)$ calculation regarding an  \textit{incomplete unipotent integration} of the Maass form $\Phi$.  

 		%%%%%%%%%%%%%%%%%%%%%%%%%%%%%%%%%%%%%%%%%%%%%%%%%%%%%%%%%%%%%%%%%%%%%%%%%%%%%%%%%%%%%%%%%%%%%%%%%%%%%%%%%%%%%%

\section{Preliminary}\label{prel}

  The analytic theory of   automorphic forms for the group $GL(3)$ have undergone considerable development in the past decade.  Readers should beware that the recent articles in the field (e.g., \cite{Bu13, Bu16, Bu20, GSW21}) have adopted a different set of  conventions and normalizations than the ones in the standard text \cite{Gold}.  (Nevertheless,  \cite{Gold} remains a useful reference as it documents plenty of standard results and their proofs thoroughly.)

  In this article,  we shall follow this recent shift of conventions (closest to \cite{Bu20}) which is convenient in many ways and is better aligned with  the theory of automorphic representation. We shall summarize the essential notions and results below with extra attention on the archimedean calculations involving Whittaker functions as they will play important roles in this article.

  \subsection{Notations and Conventions}\label{notconv}
  
  	Throughout this article, we use the following notations: $\Gamma_{\R}(s):= \pi^{-s/2}\Gamma(s/2)$  $(s\in \C)$;  $e(x):=e^{2\pi i x}$ \ $(x\in \R)$; \  $\Gamma_{n}:=SL_{n}(\Z)$ \ $(n\ge 2)$. Without otherwise specified, our test function $H$ lies in the class $ \mathcal{C}_{\eta}$ and $H= h^{\#}$.    We will often use the same symbol to denote a function (in $s$) and its analytic continuation.
  
  We will frequently encounter contour integrals of the shape
  \begin{align*}
  	\int_{-i\infty}^{i\infty} \ \cdots \   \int_{-i\infty}^{i\infty} \ (\cdots) \  \frac{ds_{1}}{2\pi i} \ \cdots \  \frac{ds_{k}}{2\pi i}
  \end{align*}
  where the contours involved should follow Barnes' convention: they pass to the right of all of the poles of the gamma functions in the form $\Gamma(s_{i}+ a)$ and to the left of all of the poles of the gamma functions in the form $\Gamma(a -s_{i})$.

  We also adopt the following set of conventions:
  \begin{enumerate}
  	\item \label{hecke} All Maass cusp forms will be  simultaneous eigenfunctions of the Hecke operators and will be either even or odd. Also, their first Fourier coefficients are  equal to $1$. In this case, the forms are said to be \textbf{Hecke-normalized}. Note that there are no odd form for  $SL_{3}(\Z)$, see Proposition 9.2.5 of \cite{Gold}. 
  	
  	%\textbf{symmetric}, see Definition 9.2.4 of \cite{Gold}. 

  	\item Our fixed Maass cusp form $\Phi$ of $SL_{3}(\Z)$ is assumed to be \textbf{tempered at $\infty$}, i.e., its Langlands parameters are purely imaginary.

  	\item Denote by $\theta $ the best progress towards the Ramanujan conjecture for the Maass cusp forms of $SL_{3}(\Z)$.  We have $\theta\le \frac{1}{2}- \frac{1}{10}$, see Theorem 12.5.1 of \cite{Gold}.  
  	
  \end{enumerate}

%Also, we want to ensure the compatibility of the conventions between the  $GL(2)$ and  the $GL(3)$ theories. 

\subsection{(Spherical) Whittaker Functions \& Transforms}\label{prelimwhitt}

In the rest of this article, all Whittaker functions are referred to the spherical ones.  The Whittaker function of $GL_{2}(\R)$ is more familiar and is given by 
 \begin{align}\label{gl2wh}
 	W_{\mu}(y) \  := \ 2\sqrt{y} K_{\mu}(2\pi y)
 \end{align}
 for $\mu\in \C$ and $y>0$.  Under this normalization, the following holds:

 \begin{prop}
 	For  $\re (w+\frac{1}{2}\pm \mu)>0$,  we have
 	\begin{align}\label{melwhitgl2}
 		\int_{0}^{\infty} W_{\mu}(y) y^{w} \ d^{\times} y \ = \  \frac{\pi^{-w-\frac{1}{2}}}{2} \  \Gamma\left(\frac{w+\frac{1}{2}+\mu}{2}\right) \Gamma\left(\frac{w+\frac{1}{2}-\mu}{2}\right). 
 	\end{align}
 \end{prop}

\begin{proof}
	Standard, see equation (2.5.2) of \cite{Mo97} for instance.   
	\end{proof}

 For the group $GL_{3}(\R)$, we first  introduce the function 
 \begin{align*}
 I_{\alpha}(y_{0}, y_{1}) \ = \ I_{\alpha}\begin{pmatrix}
 y_{0}y_{1} &          &       \\
 & y_{0} &       \\
 &          &  1
 \end{pmatrix} \ := \ y_{0}^{1-\alpha_{3}} y_{1}^{1+\alpha_{1}}
 \end{align*}
 for $y_{0}, y_{1}>0$ and  $\alpha\in\mathfrak{a}_{\C}^{(3)}:=  \left\{ (\alpha_{1},\alpha_{2}, \alpha_{3})\in \C^{3}: \alpha_{1}+\alpha_{2}+\alpha_{3}=0 \right\} $.   Then the Whittaker function for $GL_{3}(\R)$, denoted by $W_{\alpha}(y_{0}, y_{1})  =   W_{\alpha}\begin{psmallmatrix}
 	y_{0}y_{1} &          &       \\
 	& y_{0} &       \\
 	&          &  1
 \end{psmallmatrix}$, is defined in terms of \textit{Jacquet's integral}: 
	\begin{align}\label{jac}
 \prod_{1\le j<k\le 3} \Gamma_{\R}(1+\alpha_{j}-\alpha_{k}) \  \int_{\R}  \int_{\R}   \int_{\R} \   I_{\alpha}\left[\begin{psmallmatrix}
	&    & 1\\
	& -1 &    \\
	1 &  &
	\end{psmallmatrix} \begin{psmallmatrix}
	1 & u_{1,2} & u_{1,3} \\
	&     1      & u_{2,3} \\
	&             & 1
	\end{psmallmatrix} 
	\begin{psmallmatrix}
	y_{0}y_{1} &          &       \\
	& y_{0} &       \\
	&          &  1
	\end{psmallmatrix}\right] \cdot  e(-u_{1,2}-u_{2,3}) \ du_{1,2} \ du_{1,3} \ du_{2,3} 
	\end{align}
	for $y_{0}, y_{1}>0$ and $\alpha\in \mathfrak{a}_{\C}^{(3)}$.   See Chapter 5.5 of \cite{Gold} for details.  
	
	\begin{rem}
		 Notice the  differences in the normalizations of $I_{\alpha}$ here and the one given by equation 5.1.1 of  \cite{Gold}.  Also, the Whittaker functions here are actually the \textit{complete} Whittaker functions of \cite{Gold}. 
	\end{rem}

	Moreover,  the Whittaker function of $GL_{3}(\R)$ admits the following useful Mellin-Barnes representation commonly known as the \textit{Vinogradov-Takhtadzhyan formula}:

	\begin{prop}\label{vtmellin}
		Assume $\alpha \in \mathfrak{a}_{\C}^{(3)}$ is tempered, i.e., $\re \alpha_{i}=0$ ($i=1,2,3$).  Then for any $\sigma_{0}, \sigma_{1}>0$,  
		\begin{align}\label{whittaker}
		W_{-\alpha}(y_{0}, y_{1}) \ =  \ \frac{1}{4} \  \int_{(\sigma_{0})}\int_{(\sigma_{1})} G_{\alpha}(s_{0},s_{1}) y_{0} ^{1-s_{0}} y_{1}^{1-s_{1}} \  \frac{ds_{0}}{2\pi i } \frac{ds_{1}}{2\pi i }, \hspace{20pt} y_{0}, y_{1} \ > \ 0,
		\end{align}
		where
		\begin{align}\label{vtgamm}
		G_{\alpha}(s_{0},s_{1})  \ := \ \frac{\prod\limits_{i=1}^{3} \Gamma_{\R}\left( s_{0}+\alpha_{i}\right) \Gamma_{\R}\left( s_{1}-\alpha_{i}\right)}{\Gamma_{\R}(s_{0}+s_{1})}. 
		\end{align}
	
		\end{prop}
		
		\begin{proof}
This can be verified (up to the constant $1/4$) by a brute force yet elementary calculation, i.e., checking the right side of (\ref{whittaker}) satisfies the differential equations of $GL(3)$ (see pp. 38-39 of \cite{Bump84}).  For a cleaner proof starting from (\ref{jac}), see  Chapter X of \cite{Bump84}.   
			\end{proof}

\begin{rem}
	Notice the sign convention of the $\alpha_{i}$'s in formula (\ref{whittaker}) --- it is consistent with \cite{Bu20} but is opposite to that of  (6.1.4)$-$(6.1.5) in \cite{Gold}.
%	It is important to keep track of this convention, especially in  Proposition \ref{stadediff1}, \ref{secMTcom},  \ref{Eissimpl} and Section \ref{CFKRS}. 
\end{rem}

 \begin{cor}\label{gl3whittbd}
 	For any \ $-\infty < A_{0}, A_{1}<1$, we have
 	\begin{align}\label{whitest3}
 	\left|W_{-\alpha}(y_{0}, y_{1})\right| \ \ll \ y_{0}^{A_{0}} y_{1}^{A_{1}}, \hspace{20pt} y_{0}, y_{1} \ > \ 0,
 	\end{align}
 	where the implicit constant depends only on $\alpha, A_{0}, A_{1}$. 
 \end{cor}
 
 \begin{proof}
 	Follows directly from  Proposition \ref{vtmellin} by contour shifting. 
 	\end{proof}

 We will need the explicit evaluation of the $GL_{3}(\R)\times GL_{2}(\R)$ Rankin-Selberg integral. It is a consequence of the  \textit{Second Barnes Lemma} stated as follows.

 \begin{comment}
 	\item  For $a, b,c, p, q \in \C$, we have
 	\begin{align}\label{eassecBar}
 		&\int_{-i\infty}^{i\infty} \ \frac{\Gamma(u-a)\Gamma(u-b)\Gamma(u-c)\Gamma(p-u)\Gamma(q-u)}{\Gamma(u+p+q-a-b-c)} \ \frac{du}{2\pi i} \nonumber\\
 		& \hspace{160pt} 	\ = \    \frac{\Gamma(p-a) \Gamma(p-b) \Gamma(p-c)\Gamma(q-a)\Gamma(q-b)\Gamma(q-c)}{\Gamma(p+q-a-b)\Gamma(p+q-b-c)\Gamma(p+q-c-a)}.
 	\end{align}
 \end{comment}
 
 	\begin{lem}\label{secBarn}
 	  For $a,b,c,d,e,f \in \C$ with $f=a+b+c+d+e$, we have
 		\begin{align}\label{secBar}
 			\int_{-i\infty}^{i\infty} \frac{\Gamma(w+a) \Gamma(w+b)\Gamma(w+c) \Gamma(d-w) \Gamma(e-w)}{\Gamma(w+f)} \ &\frac{dw}{2\pi i} \nonumber\\
 			& \hspace{-50pt}  \ = \   \frac{\Gamma(d+a) \Gamma(d+b) \Gamma(d+c) \Gamma(e+a) \Gamma(e+b)\Gamma(e+c)}{\Gamma(f-a) \Gamma(f-b) \Gamma(f-c)}.  
 		\end{align}
 	Recall that  the contours involved must follow the Barnes convention, see  Section \ref{notconv}.
 %Note that this is  the essential ingredient behind Lemma \ref{stadediff1}.  	where the contour follows Barnes' convention. 
 
 \end{lem}

\begin{proof}
	See Bailey \cite{Ba64}. 
\end{proof}

 \begin{prop}\label{stadediff1}
  Let  $W_{\mu}$ and $W_{-\alpha}$  be the Whittaker functions of  $GL_{2}(\R)$ and $GL_{3}(\R)$  respectively.  For $\re s\gg 0$, we have
 	\begin{equation}\label{eqn sta1}
 \mathcal{Z}_{\infty}\left(s; \ W_{\mu},\  W_{-\alpha}\right) \ := \ 	\int_{0}^{\infty} \int_{0}^{\infty} \  W_{\mu}(y_{1}) \cdot W_{-\alpha}(y_{0}, y_{1}) \cdot (y_{0}^2 y_{1})^{s-\frac{1}{2}} \ \frac{dy_{0} dy_{1}}{y_{0}y_{1}^2} \ = \  \frac{1}{4}\cdot  \prod_{\pm} \prod_{k=1}^{3} \ \Gamma_{\R}\left(s\pm \mu-\alpha_{k}\right). 
 	\end{equation}
 \end{prop}

% \textbf{Consistency and , different $I$-function. }  Unfortunately, the normalizations and the conventions of  \cite{Bump88}  (following  \cite{Bump84} closely) are quite different from the ones adopted in this article. 

 \begin{proof}
 See  \cite{Bump88}. 
 \end{proof}

%%%%%%%%%%%%%%%%%%%%%%%%%%%%%%%%%%%%%%%%%%%%%%%%%%%%%%%%%%%%%%%%%%%%%%%%%%%%%%%%%%%%%%%%%%%%%%%%%%%%%%%%%%%%%%%%%%%%%%%%%%

The following pair of integral transforms plays an important role in the archimedean aspect of this article. 

	\begin{defi}\label{defwhittrans}
	Let  $h: (0, \infty) \rightarrow \C$ and $H: i\R \rightarrow \C$ be  measurable functions with $H(\mu)=H(-\mu)$. Let $W_{\mu}(y):= 2\sqrt{y} K_{\mu}(2\pi y)$.  Then the Kontorovich-Lebedev transform of $h$ is defined by
	\begin{equation}\label{eqn whitranseq}
	h^{\#}(\mu) \ := \ \int_{0}^{\infty} h(y)\cdot W_{\mu}(y) \ \frac{dy}{y^2},   
	\end{equation}
	whereas its inverse transform is defined by 
	\begin{align}\label{invers}
	H^{\flat}(y) \ =  \ \frac{1}{4\pi i} \int_{(0)}  \ H(\mu)\cdot W_{\mu}(y) \ \frac{d\mu}{\left| \Gamma(\mu)\right|^2},
	\end{align}
	provided  the integrals converge absolutely. 	Note: the normalization constant $1/4\pi i$ in (\ref{invers}) is consistent with that in \cite{Mo97}, \cite{I02}. 
	\end{defi}

	\begin{comment}
	\begin{rem}
		A trivial bound for $h^{\#}(\mu)$ follows rather immediately from a uniform bound for the $GL(2)$-Whittaler functions obtained in Section 7 of \cite{HaMi}. More precisely, we have
		\begin{align}
		W_{\mu}(y)  \ \ll_{\epsilon} \ e^{-\frac{\pi}{2} |\mu|} |1+2\mu|^{-\frac{1}{3}+\epsilon} y^{\frac{1}{2}-\epsilon} \exp\left(-y/|\mu|\right)
		\end{align}
		for $\mu\in i\R$ with $|\mu|>2$. Then  $h(y)\ll \min\left\{ y,  y^{-1}    \right\}^{M}$ implies that
		\begin{align}\label{whitransbdd}
		h^{\#}(\mu) \ \ll_{\epsilon, M} \  e^{-\frac{\pi}{2} |\mu|} |1+2\mu|^{-\frac{1}{3}+\epsilon}
		\end{align}
		on the vertical line $i\R$. 
		\end{rem}
		\end{comment}
		
		\begin{defi}
				Let $\mathcal{C}_{\eta}$ be the class of holomorphic functions $H$  on the vertical strip  $|\re \mu|< 2\eta$ such that 
				\begin{enumerate}
					\item $H(\mu)=H(-\mu)$,
					
					\item $H$ has rapid decay in the sense that
					\begin{align}\label{rapdeca}
					H(\mu) \ \ll \ e^{-2\pi|\mu|} \hspace{50pt} (|\re \mu| \ < \  2\eta). 
					\end{align}
				\end{enumerate}
				In this article, we take $\eta>40$ without otherwise specifying. 
			\end{defi}
			
		By contour-shifting and Stirling's formula, we have

	\begin{prop}\label{inKLconv}
	For any $H\in \mathcal{C}_{\eta}$, the integral  (\ref{invers}) defining $H^{\flat}$   converges absolutely. Moreover, we have 
				\begin{align}\label{decbdd}
			    H^{\flat}(y) \ \ll \ \min \{y, y^{-1}\}^{\eta} \hspace{20pt} (y \ > \ 0). 
				\end{align}

		\end{prop}
		
		\begin{proof}
See Lemma 2.10 of \cite{Mo97}.
			\end{proof}

	\begin{prop}\label{plancherel}
		Under the same  assumptions of Proposition \ref{inKLconv}, we have
		\begin{align}
		(h^{\#})^{\flat}(g) \ = \ h(g)  \hspace{20pt} \text{ and } \hspace{20pt} 
		(H^{\flat})^{\#}(\mu) \ = \ H(\mu). 
		\end{align}
	\end{prop}
	
	\begin{proof}
		See Lemma 2.10 of \cite{Mo97}. It is a consequence of the Rankin-Selberg calculation for  $GL_{2}(\R)\times GL_{2}(\R)$. 
	\end{proof}

%%%%%%%%%%%%%%%%%%%%%%%%%%%%%%%%%%%%%%%%%%%%%%%%%%%%%%%%%%%%%%%%%%%%%%%%%%%%%%%%%%%%%%%%

\subsection{Automorphic Forms of $GL(2)$ and $GL(3)$}
Let 
 \begin{align*}
 	\mathfrak{h}^2 := \left\{ \begin{pmatrix}
 		1 & u\\
 		& 1
 	\end{pmatrix} \begin{pmatrix}
 		y & \\
 		& 1
 	\end{pmatrix} : u \in \R, \ y >  0  \right\}
 \end{align*}
 with its invariant measure  given by $y^{-2} du dy$.  Let $\Delta:= -y^2\left( \partial_{x}^2 +\partial_{y}^2\right)$. An automorphic form  $\phi:  \mathfrak{h}^2 \rightarrow \C$ of $\Gamma_{2}=SL_{2}(\Z)$ satisfies  $\Delta\phi = \left( \frac{1}{4} -\mu^2\right) \phi$ for some $\mu= \mu(\phi) \in \C$. It is often handy to identify $\mu$ with the pair $(\mu, -\mu)\in \mathfrak{a}_{\C}^{(2)}$.

For $a\in \Z-\{0 \}$, the $a$-th Fourier coefficient of $\phi$, denoted by $\mathcal{B}_{\phi}(a)$, is defined  by
\begin{equation}\label{eiscuspcoeff}
(\widehat{\phi})_{a}(y) \ := \  \int_{0}^{1} \ \phi\left[ \begin{pmatrix}
1 & u\\
& 1
\end{pmatrix} \begin{pmatrix}
y & \\
& 1
\end{pmatrix}\right] e(-au) \ du \ = \  \frac{\mathcal{B}_{\phi}(a)}{\sqrt{|a|}} \cdot  W_{\mu(\phi)}(|a|y). 
\end{equation}
 In the case of the Eisenstein series of $\Gamma_{2}$, i.e., 
		\begin{align}\label{gl2eins}
	\phi \ = \ 	E(z; \mu) \ := \  \frac{1}{2} \ \sum_{\gamma\in U_{2}(\Z)\setminus\Gamma_{2}} I_{\mu}(\im \gamma z) \hspace{20pt} (z \ \in \  \mathfrak{h}^2), 
		\end{align}
	where $I_{\mu}(y):= y^{\mu+\frac{1}{2}}$, it is well-known that $\Delta E(*; \mu) \  = \  \left( \frac{1}{4}-\mu^2\right) E(*; \mu)$ and  the Fourier coefficients $\mathcal{B}(a; \mu)$ of $E(*;\mu)$ is given by  
		\begin{align}\label{eisfournorm}
		\mathcal{B}(a;\mu) \ = \ \frac{|a|^{\mu}\sigma_{-2\mu}(|a|)}{\Lambda(1+2\mu)}, 
		\end{align}	  
		where
		\begin{align*}
			\Lambda(s) \ := \  \pi^{-s/2}\Gamma(s/2)\zeta(s) \hspace{15pt} \text{ and } \hspace{15pt}  \sigma_{-2\mu}(|a|) \ := \  \sum_{d \mid a} d^{-2\mu}.  
		\end{align*}
		The series (\ref{gl2eins}) converges absolutely for $\re \mu>1/2$ and it admits a meromorphic continuation to $\C$.

Next, let  
\begin{align*}
	\mathfrak{h}^{3} \ := \  \left\{  \begin{pmatrix}
		1 & u_{1,2} & u_{1,3} \\
		&     1      & u_{2,3} \\
		&             & 1
	\end{pmatrix} 
	\begin{pmatrix}
		y_{0}y_{1} &          &       \\
		& y_{0} &       \\
		&          &  1
	\end{pmatrix}: u_{i,j} \in \R, \ y_{k} >0 \right\}.  
	\end{align*}
 Let $\Phi: \mathfrak{h}^3 \rightarrow \C$ be a Maass cusp form of $\Gamma_{3}$ as defined in Definition 5.1.3 of \cite{Gold}. In particular,  there exists  $\alpha= \alpha(\Phi)\in \mathfrak{a}_{\C}^{(3)}$ such that  for any $D\in Z(U\mathfrak{gl}_{3}(\C))$  (the center of the universal enveloping algebra of the Lie algebra $\mathfrak{gl}_{3}(\C)$),  we have
\begin{align*}
D\Phi \ = \  \lambda_{D} \Phi \hspace{20pt} \text{ and } \hspace{20pt} 
DI_{\alpha} \ = \ \lambda_{D} I_{\alpha}
\end{align*}
for some $\lambda_{D}\in \C$. The triple $\alpha(\Phi)$ is said to be  the \textit{Langlands parameters} of $\Phi$. %Once again, we follow the convention  of \cite{Gold} (pp. 261). 

\begin{defi}\label{fourcoeff}
	Let $m=(m_{1}, m_{2})\in (\Z-\{0\})^{2}$ and $\Phi:  \mathfrak{h}^3 \rightarrow \C$ be a Maass cusp form of $SL_{3}(\Z)$. 	For any $y_{0}, y_{1}>0$, the integral defined by
	\begin{align}\label{fourcoeff2}
	(\widehat{\Phi})_{(m_{1}, m_{2})}	\begin{psmallmatrix}
	y_{0}y_{1} &          &       \\
	& y_{0} &       \\
	&          &  1
	\end{psmallmatrix} \ &:= \ \int_{0}^{1} \int_{0}^{1}\int_{0}^{1} \Phi\left[ \begin{psmallmatrix}
	1 & u_{1,2} & u_{1,3} \\
	&     1      & u_{2,3} \\
	&             & 1
	\end{psmallmatrix} 
	\begin{psmallmatrix}
	y_{0}y_{1} &          &       \\
	& y_{0} &       \\
	&          &  1
	\end{psmallmatrix}\right] \cdot  e\left(-m_{1} u_{2,3}-m_{2} u_{1,2} \right) \ du_{1,2} \ du_{1,3} \ du_{2,3}.
	\end{align}
 is said to be the  $(m_{1}, m_{2})$-th \textbf{Fourier-Whittaker period} of $\Phi$. 
Moreover,  the $(m_{1}, m_{2})$-th \textbf{Fourier coefficient} of $\Phi$ is the complex number $\mathcal{B}_{\Phi}(m_{1}, m_{2})$  for which
	\begin{align}
		(\widehat{\Phi})_{(m_{1}, m_{2})}	\begin{pmatrix}
		y_{0}y_{1} &          &       \\
		& y_{0} &       \\
		&          &  1
		\end{pmatrix} \ = \  \frac{\mathcal{B}_{\Phi}(m_{1}, m_{2})}{|m_{1}m_{2}|} \  W_{\alpha(\Phi)}^{^{\sgn(m_{2})}}\begin{pmatrix}
		(|m_{1}|y_{0})(|m_{2}|y_{1}) &          &       \\
		& |m_{1}|y_{0} &       \\
		&          &  1
		\end{pmatrix}
	 \end{align}
     holds for any  $y_{0}, y_{1}>0$.   
     
\end{defi}

\begin{rem}\label{heckeigdef}
	\ 
	\begin{enumerate}
		\item The multiplicity-one theorem of Shalika (see Theorem 6.1.6 of \cite{Gold}) guarantees  the well-definedness of the Fourier coefficients for $\Phi$. 
		
		\item If $\Phi$ is Hecke-normalized (c.f. Section \ref{notconv}.(\ref{hecke})), then $\mathcal{B}_{\Phi}(1,n)$ can be shown to be a Hecke eigenvalue of $\Phi$ (see Section 6.4 of \cite{Gold}). %In this case, we also write $\lambda_{\Phi}(n):= \mathcal{B}_{\Phi}(1,n)$. 
	\end{enumerate}

\end{rem}

%%%%%%%%%%%%%%%%%%%%%%%%%%%%%%%%%%%%%%%%%%%%%%%%%%%%%%%%%%%%%%%%%%%%%%%%%%%%%%%%%%%%%%%%%%%%%%%%%%%%%%%%%%%%%%%%%%%%%%%%%%%%%%%%%%%%%%%%%%%%%%%%%%%%%%%%%%%%%%%%%%%%%%%%%%%%%%%%%%%%%%%%%%%%%%%%%%%%%%%%%%%%

\subsection{Automorphic $L$-functions}\label{autoLfunc}
The Maass cusp forms  $\Phi$ and $\phi$ below are  Hecke-normalized and their Langlands parameters are denoted by    $\alpha\in \mathfrak{a}_{\C}^{(3)}$ and $\mu \in \mathfrak{a}_{\C}^{(2)}$  respectively.   Let $\widetilde{\Phi}(g):= \Phi\left( ^{t}g^{-1}\right)$ be the dual form of $\Phi$. It is not hard to show that the Langlands parameters of $\widetilde{\Phi}$ are given by $-\alpha$.

	\begin{defi}\label{DScuspRS}
		Suppose $\Phi$ and  $\phi$ are  Maass cusp forms of   $\Gamma_{3}$ and $\Gamma_{2}$ respectively.  For $\re s \gg 1$, the Rankin-Selberg $L$-function of $\Phi$ and $\phi$ is defined by
	\begin{align}\label{rankse}
		L\left(s, \phi\otimes \Phi\right) \ &:= \    \sum_{m_{1}=1}^{\infty} \ \sum_{m_{2}= 1}^{\infty}  \frac{\mathcal{B}_{\phi}(m_{2})  \mathcal{B}_{\Phi}(m_{1}, m_{2} )}{ \left(m_{1}^2 m_{2}\right)^{s}}. 
	\end{align}

\end{defi}

		Although we will not make use of the Dirichlet series for $L(s, \phi\otimes \Phi)$ throughout this article,  it is commonly used in the literature (especially in the `Kuznetsov-Voronoi' method).  We take this opportunity to indicate our normalization in terms of Dirichlet series so as to faciliate conversion/ comparison, as well as correct some minor inaccuracies  in  Section 12.2 of  \cite{Gold}.

		\begin{prop}\label{ranselmainthm}
		Suppose $\Phi$ and $\phi$ are  Maass cusp forms of   $\Gamma_{3}$ and  $\Gamma_{2}$)  respectively.  In addition, assume that $\phi$ is  even.  Then for	any $\re s\gg 1$, we have
		\begin{align}\label{sameparun}
		  \int_{\Gamma_{2}\setminus GL_{2}(\R)}  \ \phi(g) \widetilde{\Phi}\begin{pmatrix}
				g & \\
				& 1
			\end{pmatrix} |\det g|^{s-\frac{1}{2}} \ dg  \ = \ 	\frac{1}{2} \cdot \Lambda\left(s, \phi\otimes \widetilde{\Phi}\right),
		\end{align}
	where
		\begin{align}
			\Lambda\left(s, \phi\otimes \widetilde{\Phi}\right) \ &:= \ L_{\infty}\left(s, \phi\otimes \widetilde{\Phi}\right)\cdot L\left(s, \phi\otimes \widetilde{\Phi}\right)
		\end{align}
	and 
		\begin{align}
				L_{\infty}\left(s, \phi\otimes \widetilde{\Phi}\right) \ &:= \  \prod_{k=1}^{3} \ \Gamma_{\R}\left(s\pm  \mu- \alpha_{k}\right). 
		\end{align}

	\end{prop}

	\begin{proof}
	The assumption on the parity of $\phi$ is missing in \cite{Gold}. Also, the pairing should be taken over the quotient $\Gamma_{2}\setminus GL_{2}(\R)$ instead of $\Gamma_{2}\setminus\mathfrak{h}^2$ in \cite{Gold}. 
		
		As a brief sketch, we replace $\widetilde{\Phi}\begin{pmatrix}
			g & \\
			& 1
		\end{pmatrix} $ by its Fourier-Whittaker expansion (see Theorem 5.3.2 of \cite{Gold}) on the left side of (\ref{sameparun}) and unfold. Then one may extract the Dirichlet series in (\ref{rankse}) by using (\ref{eiscuspcoeff}) and (\ref{fourcoeff2}).      The integral of Whittaker functions can be computed by   Proposition \ref{stadediff1}.   
	\end{proof}

	In the rest of this article, we will often make use of the shorthands $(\mathbb{P}_{2}^{3}\Phi)(g):= \Phi\begin{pmatrix}
	g & \\
	& 1
\end{pmatrix}$ and the pairing
\begin{align*}
	\left(\phi,\  (\mathbb{P}_{2}^{3}\Phi)\cdot |\det *|^{\overline{s}-\frac{1}{2}}\right) _{\Gamma_{2}\setminus GL_{2}(\R)}
\end{align*}
for  the  integral on the left side of (\ref{sameparun}). By the rapid decay of $\Phi$ at $\infty$, this integral converges absolutely for any $s \in \C$ and uniformly on any compact subset of $\C$. Thus, the $L$-function $L(s, \phi\otimes \widetilde{\Phi})$ admits an entire continuation. 

	\begin{rem}
	\ 
	\begin{enumerate}

		\item When $\phi$ is even, the involution $g\mapsto\  ^{t}g^{-1}$  gives  the functional equation 
		\begin{align*}
			\Lambda\left(s, \phi\otimes \widetilde{\Phi}\right) \ = \ \Lambda\left(1-s, \phi \otimes \Phi\right). 
		\end{align*}

		\item When $\phi$ is odd,  the right side of  (\ref{sameparun}) is identical to  $0$ and hence \textit{does not} provide an integral representation for   $\Lambda(s, \phi\otimes \widetilde{\Phi})$. One must alter Proposition \ref{ranselmainthm} accordingly in this case, say using the raising/ lowering operators, or proceed adelically with an appropriate choice of  test vector at $\infty$.  However, we shall not go into these as our spectral average is taken over even Maass forms of $\Gamma_{2}$ only. 
		
		\item As discussed in  Section \ref{comCI}, the roles of parities and root numbers are rather intricate in the study of moments of $L$-functions, especially regarding the archimedean integral transforms. 
		
	\end{enumerate}
	
\end{rem}

	\begin{defi}\label{JPSSstdL}
		Let $\Phi: \mathfrak{h}^{3}\rightarrow \C$ be a Maass cusp form of $\Gamma_{3}$. For $\re s  \gg  1$,  the standard $L$-function of $\Phi$ is defined by
		\begin{align}\label{gl3stdL}
			L\left(s, \Phi \right)\ := \ \sum_{n=1}^{\infty} \ \frac{\mathcal{B}_{\Phi}(1,n)}{n^{s}}.   
		\end{align}
	\end{defi}

	In the rest of this article, we will not make use of the integral representation of $L(s, \Phi)$, i.e.,  the first line of (\ref{intrexp}) (replacing $\widetilde{\Phi}$ therein by $\Phi$).   It suffices to know $L(s, \Phi)$ admits an entire continuation and  satisfies the following  functional equation: 
	
\begin{prop}
	Let $\Phi: \mathfrak{h}^{3}\rightarrow \C$ be a Maass cusp form of $\Gamma_{3}$. For any $s\in \C$, we have
		\begin{align}\label{JPfunc}
		\Lambda\left(s, \Phi\right) \ = \ \Lambda\left(1-s,  \widetilde{\Phi}\right),
	\end{align}
	where 
	\begin{align}
	\Lambda\left(s, \Phi\right) \ := \ 	L_{\infty}\left(s, \Phi \right) \cdot L\left(s, \Phi\right)
	\end{align}
and
	\begin{align}
L_{\infty}\left(s, \Phi \right) \ := \  \prod_{k=1}^{3}  \ \Gamma_{\R}\left(s+\alpha_{k} \right). 
\end{align}

\end{prop}

	\begin{proof}
		See Chapter 6.5 of \cite{Gold} or \cite{JPSS}. 
	\end{proof}

Furthermore, since $\phi$ and $\Phi$ are assumed to be Hecke-normalized, the standard $L$-functions $L(s, \phi)$ and $L(s, \Phi)$ admit Euler product of the shape
\begin{align}\label{23Eul}
L(s, \phi) \ = \  \prod_{p} \ \prod_{j=1}^{2} \ \left(1-\beta_{\phi, j}(p)p^{-s}\right)^{-1}, \hspace{20pt} 
  L(s, \Phi) \ = \  \prod_{p} \ \prod_{k=1}^{3} \ \left(1-\alpha_{\Phi, k}(p)p^{-s}\right)^{-1}
\end{align}
for $\re s \gg 1$. Then one can show that
\begin{align}\label{RSEul}
L(s, \phi \otimes \Phi) \ = \   \prod_{p} \ \prod_{j=1}^{2}  \ \prod_{k=1}^{3}  \left(1- \beta_{\phi, j}(p)\alpha_{\Phi, k}(p)p^{-s}\right)^{-1}
\end{align}
by Cauchy's identity, cf. the argument of Proposition 7.4.12 of \cite{Gold}.

	\begin{prop}\label{ranseleis}
		For $\re (s\pm \mu) \gg 1$, we have
		\begin{align}\label{samepareis}
		\left( E(*; \mu), \ \left( \mathbb{P}_{2}^{3} \Phi\right) \cdot |\det *|^{\bar{s}-\frac{1}{2}}\right)_{\Gamma_{2}\setminus GL_{2}(\R)} 
		\ = \ \frac{1}{2}   \ \frac{\Lambda\left( s+ \mu, \  \widetilde{\Phi} \right)\Lambda\left( s- \mu, \  \widetilde{\Phi} \right) }{\Lambda(1+2\mu)}.
		\end{align}
	\end{prop}
	
	\begin{proof}
		Parallel to  Proposition \ref{ranselmainthm}. Meanwhile, we make use of (\ref{eisfournorm}). 
	\end{proof}
	
	\begin{rem}
		By analytic continuation,  (\ref{sameparun}) and (\ref{samepareis})  hold for  $s\in \C$  and  away from the poles of $E(*; \mu)$. In fact, the rapid decay of $\Phi$ at $\infty$ guarantees the pairings  converge absolutely. 
	\end{rem}

		%%%%%%%%%%%%%%%%%%%%%%%%%%%%%%%%%%%%%%%%%%%%%%%%%%%%%%%%%%%%%%%%%%%%%%%%%%%%%%%%%%%%%%%%%%%%%%%%%%%%%%%%%%%%%%%%%%%%%%%%
		
		\subsection{Calculation on the Spectral Side}
		
			As indicated in the introduction, our approach differs from the  `Kuznetsov-Voronoi' one right from the start --- we will not make use of the Dirichlet series (\ref{rankse}).  Instead, the moment of $GL(3)\times GL(2)$ $L$-functions is first interpreted in terms of the period integral of Definition \ref{ranselmainthm} using a  Poincar\'e series. %We shall now fix the normalization for the Poincar\'e series used in this article. 

	\begin{defi}\label{poindef}
		Let $a\ge 1$ be an integer and $h\in C^{\infty}(0,\infty)$. The Poincar\'e series of $\Gamma_{2}$   is defined by
		\begin{equation}\label{defpoin}
		P^{a}(z; h) \ := \   \sum_{\gamma\in U_{2}(\Z)\setminus \Gamma_{2}} h(a \im \gamma z) \cdot e\left(a \re \gamma z\right) \hspace{20pt} (z \ \in \  \mathfrak{h}^2)
		\end{equation}
	 provided it converges absolutely. 
	\end{defi}
	
	It is not hard to see that if the bounds 
	\begin{align}\label{poincbdd}
	h(y)  \ \ll \  y^{1+\epsilon} \hspace{10pt} \text{(as \ $y\to 0$)} \hspace{10pt} \text{ and }   \hspace{10pt} h(y) \  \ll \   y^{\frac{1}{2}-\epsilon}  \hspace{10pt} \text{ (as \ $y \to\infty$)}
	\end{align}
are satisfied, then the Poincar\'e series $P^{a}(z; h)$ converges absolutely and represents an $L^2$-function.  In this article, we  take $h:= H^{\flat}$ with $H\in \mathcal{C}_{\eta}$ and $\eta >40$. By Proposition \ref{inKLconv},  conditions (\ref{poincbdd})  clearly holds. We often use the shorthand $P^{a}:= P^{a}(*;h)$. Also, we use  \ $\langle \ \cdot \ ,\  \cdot \ \rangle$ \ to denote the Petersson inner product on $\Gamma_{2}\setminus \mathfrak{h}^{2}$, i.e., 
\begin{align*}
	\left\langle \ \phi_{1} \ ,\  \phi_{2} \ \right\rangle \ := \ \int_{\Gamma_{2}\setminus \mathfrak{h}^{2}} \ \phi_{1}(g) \cdot \ovA{\phi_{2}(g)} \ dg
\end{align*}
with $dg$ being the invariant measure on $\mathfrak{h}^{2}$. 
	
		\begin{lem}\label{stdunfold}
			Let $\phi$ be a Maass cusp form of  $\Gamma_{2}$, $\Delta \phi =\left( \frac{1}{4}-\mu^2\right)\phi$, and $\mathcal{B}_{\phi}(a)$ be  the $a$-th Fourier coefficient of $\phi$.   Then 
			\begin{align*}
			\big\langle \ P^{a} \ ,\  \phi \ \big\rangle 
			\ &=\ |a|^{1/2} \cdot \ovB{\mathcal{B}_{\phi}(a)}  \cdot h^{\#}\left(\ovA{\mu}\right). 
			\end{align*}

		\end{lem}
		
		\begin{proof}
			Replace $P^{a}$ in $\big\langle P^{a},\  \phi \big\rangle$ by its definition and unfold, we easily find that
			\begin{align*}
			\big\langle \ P^{a} \ ,\  \phi \ \big\rangle 
			\  =  \  \int_{0}^{\infty}  h(ay)\cdot  \ovC{\widehat{(\phi)}_{a}(y)} \ \frac{dy}{y^2}. 
			\end{align*}
			The result follows at once  upon plugging-in (\ref{eiscuspcoeff}) and making the change of variable $y\to |a|^{-1}y$.
		\end{proof}
		
	Similarly, the following holds away from the poles of $E(*; \mu)$: 		
		\begin{lem}\label{eisunfold}
			\begin{align}
				\bigg\langle P^{a}, \ E(*;  \mu) \bigg\rangle  \  =  \  |a|^{1/2}\cdot \frac{|a|^{\overline{\mu}}\sigma_{-2\overline{\mu}}(|a|)}{\zeta^{*}(1+2\overline{\mu})} \cdot h^{\#}\left(\ovA{\mu}\right). 
			\end{align}
		\end{lem}

	\begin{prop}[Spectral Expansion]\label{langselde}
	Suppose $f\in L^{2}(\Gamma_{2}\setminus\mathfrak{h}^2)$ and \  $\langle f, 1 \rangle=0$.   Then 
		\begin{align}\label{langsel}
		f(z)  \ = \ \sum_{j= 1}^{\infty} \  \frac{\big\langle f, \phi_{j} \big\rangle}{\big\langle \phi_{j}, \phi_{j}\big\rangle}\cdot  \phi_{j}(z)
	\ 	+  \ \int_{(0)}  \ \bigg\langle f, \ E(*; \mu) \bigg\rangle \cdot E(z; \mu) \  \frac{d\mu}{4\pi i}  \hspace{20pt} (z \ \in \ \mathfrak{h}^2)
		\end{align}
		where  $(\phi_{j})_{j\ge 1}$ is  any orthogonal basis of Maass cusp forms for  $\Gamma_{2}$. 
	\end{prop}
	
	\begin{proof}
		See  Theorem 3.16.1 of \cite{Gold}. 
		\end{proof}

%%%%%%%%%%%%%%%%%%%%%%%%%%%%%%%%%%%%%%%%%%%%%%%%%%%%%%%%%%%%%%%%%%%%%%%%%%%%%%%%%%%%%%%%%%

	\begin{prop}\label{comspec}
		Let $\Phi$ be a  Maass cusp form of $\Gamma_{3}$ and $P^{a}$ be a Poincar\'e series of $\Gamma_{2}$. Then 	
		\begin{align}\label{specside}
		 &\hspace{-50pt}2 \cdot |a|^{-1/2}  \cdot \left(P^a, \ (\mathbb{P}_{2}^{3} \Phi)\cdot  |\det *|^{\overline{s}-\frac{1}{2}}\right)_{\Gamma_{2}\setminus GL_{2}(\R)}  \nonumber\\
	 \hspace{40pt} &	\ = \    \sideset{}{'}{\sum}_{j=1}^{\infty} \  h^{\#}\left(\ovA{\mu_{j}}\right)\cdot  \frac{\ovB{\mathcal{B}_{j}(a)}  \cdot \Lambda\left(s, \phi_{j}\otimes \widetilde{\Phi}\right) }{\langle \phi_{j}, \phi_{j}\rangle} \nonumber\\
	&\hspace{40pt} +  \int_{(0)} \   h^{\#}\left(\mu\right) 
	\frac{ \sigma_{-2\mu}(|a|) |a|^{-\mu} \Lambda\left( s+ \mu,   \widetilde{\Phi} \right)\Lambda\left( 1-s+ \mu,   \Phi \right)}{\left|\Lambda(1+2\mu)\right|^2} \ \frac{d\mu }{4\pi i }
		\end{align}
	for any $s\in \C$, 
	where  the sum is restricted to an orthogonal basis $(\phi_{j})$ of even Hecke-normalized  Maass cusp forms for $\Gamma_{2}$ with $\Delta\phi_{j}= \left( \frac{1}{4}-\mu_{j}^2\right)\phi_{j}$ and $\mathcal{B}_{j}(a):= \mathcal{B}_{\phi_{j}}(a)$. %\footnote{ also the $a$-th Hecke eigenvalue of $\phi_{j}$. }
		
	\end{prop}

	\begin{proof}
		Substitute the spectral expansion of $P^{a}$ as in (\ref{langsel}) into the pairing $\left(P^a, \ (\mathbb{P}_{2}^{3} \Phi)\cdot  |\det *|^{\overline{s}-\frac{1}{2}}\right)_{\Gamma_{2}\setminus GL_{2}(\R)}$. The inner products involved have been computed in Lemmas \ref{stdunfold}$-$\ref{eisunfold} and Definitions \ref{ranselmainthm}$-$\ref{ranseleis}.  
		\end{proof}

		\begin{rem}
			In applications, it is important to have good control in the spectral aspect and the relevant integral transforms. It is crucial to be able to make flexible choices of test functions on the spectral side.  Furthermore, it plays a role in eliminating the extraneous  polar contributions (e.g., those do not occur in the CFKRS predictions) for  the Eisenstein cases.  These are  strengths of the Kuznetsov-based methods over the period methods (cf. \cite{Bl12a, Nu20+, Za21, Za20+}) and might partly explain why the former is  more ubiquitous in the current literature. 
			
			%(See (\ref{contana}))

			Although our method is period-based, we are able to put a large class of test functions on the spectral side as in the Kuznetsov approach, by using the pair of transforms introduced in Definition \ref{defwhittrans}. Such transforms have been generalized to  $GL(n)$ in a simple and explicit fashion in  \cite{GK12}.  They have played important roles in the recent development of the Kuznetsov formulae of higher-rank  (see \cite{GK13}, \cite{GSW21}, \cite{Bu20}). %For further generalizations, see Chapter 15 of  Wallach \cite{Wal92}. 

			Our method  preserves the advantages of both the Kuznetsov and the period approaches ---  the former being the precision in the archimedean aspect whereas  the latter being the structural insights into the nonarchimedean aspect.  
			
		\end{rem}

		%	If one wishes to prove a  smooth Weyl law for the $L$-values $L\big(1/2, \phi_{j}\otimes \widetilde{\Phi}\big)$ in the $GL(2)$ spectral aspect, one may pick
		
		%\ = \ 	H_ {R, \eta}(\mu; \Phi)
		
			 \begin{rem}\label{testfunc}
			 Within our class $\mathcal{C}_{\eta}$ of test functions, a good choice of test function is given by
			 	\begin{align}\label{ourpick}
			 	H(\mu) \ := \ \left(e^{((\mu+iM)/R)^2} + e^{((\mu-iM)/R)^2}\right)\cdot \frac{\Gamma(2\eta+\mu) \Gamma(2\eta-\mu)}{\prod\limits_{i=1}^{3} \  \Gamma\left(\frac{\frac{1}{2}+\mu-\alpha_{i}}{2}\right)\Gamma\left(\frac{\frac{1}{2}-\mu-\alpha_{i}}{2}\right)},
			 	\end{align}
			 	where  $\eta >40$, $M \gg 1$, and $R= M^{\gamma}$ \  ($0< \gamma\le 1$). In   (\ref{ourpick}),
			 	
			 	\begin{itemize}
			 		\item  the factor $e^{((\mu+iM)/R)^2} + e^{((\mu-iM)/R)^2}$ serves as a smooth cut-off for $|\mu_{j}| \in [M-R, M+R]$ and gives the needed  decay  in Proposition \ref{inKLconv};
			 		
			 		\item the factors $\prod\limits_{i=1}^{3}  \Gamma\left(\frac{\frac{1}{2}+\mu-\alpha_{i}}{2}\right)\Gamma\left(\frac{\frac{1}{2}-\mu-\alpha_{i}}{2}\right)$  cancel out the archimedean factors of $\Lambda\big(1/2, \phi_{j}\otimes \widetilde{\Phi}\big)$  on the spectral expansion (\ref{specside}) and in the diagonal contribution (\ref{diagstade});
			 		
			 		\item the factors $\Gamma\left(2\eta+ \mu\right)\Gamma(2\eta-\mu)$ balance off the exponential growth from  $d\mu/|\Gamma(\mu)|^2$, $||\phi_{j}||^{-2}$ and $|\Lambda(1+2i\mu)|^{-2}$.   Also, a large enough region of holomorphy of  (\ref{ourpick}) is maintained so that $h(y):=H^{\flat}(y)$ has  sufficient decay at $0$ and $\infty$. 
			 	\end{itemize}

			 \end{rem}

		\begin{rem}
			Readers may  wonder about  the possiblity of using an automorphic kernel in place of a Poincar\'{e} series in studying the moment of $L$-functions in Theorem \ref{maingl3gl2}. Although this offers extra flexibility in incorporating new structures, the analysis behind the integral transforms (the spherical transforms) becomes quite complicated, see \cite{Bu13} for the case of $GL(3)$. The approach using Poincar\'{e} series  seems to be more adapted to the analytic number theory of higher-rank groups. 
			
		\end{rem}

%%%%%%%%%%%%%%%%%%%%%%%%%%%%%%%%%%%%%%%%%%%%%%%%%%%%%%%%%%%%%%%%%%%%%%%%%%%%%%%%%%%%%%%%%%%%%%%%%%%%%%%%%%%%%%%%%%%%%%%%%%%%%%%%%%%%%%%%%

 \section{Basic Identity for  Dual Moment}\label{Basicide}
 
 \subsection{Unipotent Integration}

We are ready to work on the dual side of our moment formula. As a  simplification of our argument, we shall only consider  $P=P^{a}(*;h)$ with $a=1$ in the following.    Suppose $\re s > 1+ \frac{\theta}{2}$. We begin by replacing $P$ by its definition in the pairing $\left( P,\  \mathbb{P}_{2}^{3} \Phi\cdot |\det *|^{\overline{s}-\frac{1}{2}} \right)_{\Gamma_{2}\setminus GL_{2}(\R)} $. We find upon unfolding: 
  \begin{align}
 \bigg( P,\  \mathbb{P}_{2}^{3} \Phi &\cdot |\det *|^{\overline{s}-\frac{1}{2}} \bigg)_{\Gamma_{2}\setminus GL_{2}(\R)}  \nonumber\\
  &\ = \  \int_{0}^{\infty}  \int_{0}^{\infty} \  h(y_{1})\cdot (y_{0}^2y_{1})^{s-\frac{1}{2}} \cdot  \int_{0}^{1} \hspace{5pt}   \widetilde{\Phi}\left[
  	\begin{pmatrix}
  	1 & u_{1,2} &  \\
  	& 1          & \\
  	&             & 1
  	\end{pmatrix} 
  	 \begin{pmatrix}
  	 y_{0}y_{1} &             &  \\
  	 & y_{0}    &   \\
  	 &             & 1
  	 \end{pmatrix}\right] e(u_{1,2}) \ du_{1,2} \ \frac{dy_{0} dy_{1}}{y_{0}y_{1}^2}. \label{alunf}
  \end{align}
  The main task of this section is to compute the inner, `incomplete' unipotent integral in (\ref{alunf}).    We wish to evaluate it  in terms of the Fourier-Whittaker periods of $\Phi$ (see Definition \ref{fourcoeff}) as they are relevant in the constructions of various $L$-functions associated to $\Phi$, say those discussed in Section \ref{autoLfunc}.
  
  Certainly, this can be obtained by plugging in the \textit{full} Fourier expansion of  \cite{JPSS} (see \cite{Gold} Theorem 5.3.2) and look for possible simplifications. This is in fact not necessary.  We prefer a self-contained and conceptual  treatement. It simply follows from two one-dimensional Fourier expansions and  the automorphy of $\Phi$.  In essence, this is where `summation formulae' take place in our approach, and they are nicely packaged in an elementary, clean, and global  fashion.

 \begin{prop}\label{incomf}
 For any automorphic function $\Phi$ of $\Gamma_{3}$, we have, for any $y_{0}, y_{1}>0$, 
    \begin{align}\label{keylemform}
    \int_{0}^{1} \Phi\left[
    \begin{pmatrix}
    1 & u_{1,2} &  \\
    &    1       & \\
    &             & 1
    \end{pmatrix} 
    \begin{pmatrix}
    y_{0}y_{1} &             &  \\
    & y_{0}    &   \\
    &             & 1
    \end{pmatrix} \right] \  &e(-u_{1,2}) \ du_{1,2} \nonumber\\
    \ &= \  \sum_{a_{0}, a_{1}=-\infty}^{\infty}  (\widehat{\Phi})_{(a_{1},1)}\left[
    \begin{pmatrix}
    1 &             &         \\
    & 1          &  \\
    &   -a_{0} &  1
    \end{pmatrix}
   \begin{pmatrix}
   y_{0}y_{1} &             &  \\
   & y_{0}    &   \\
   &             & 1
   \end{pmatrix}\right]. 
    \end{align}

 \end{prop}

 \begin{proof}
 	
 	Firstly, we Fourier-expand along the abelian subgroup
 	$\left\{\begin{psmallmatrix}
 	1 &    & * \\
 	& 1 &   \\
 	&    & 1
 	\end{psmallmatrix}\right\}$: 
 	\begin{align}
 	 \int_{0}^{1} \Phi\left[
 	\begin{pmatrix}
 	1 & u_{1,2} &  \\
 	& 1 & \\
 	&    & 1
 	\end{pmatrix} 
 	\begin{pmatrix}
 	y_{0}y_{1} &             &  \\
 	& y_{0}    &   \\
 	&             & 1
 	\end{pmatrix}\right] & \ e(- u_{1,2}) \ du_{1,2} \nonumber\\
	&  \hspace{-80pt}  \ = \   \sum_{a_{0}=-\infty}^{\infty} \ \int_{\Z^2 \setminus \R^2} \Phi\left[
 	\begin{pmatrix}
 	1 & u_{1,2} & u_{1,3}  \\
 	& 1          & \\
 	&             & 1
 	\end{pmatrix} 
 	\begin{pmatrix}
 	y_{0}y_{1} &             &  \\
 	& y_{0}    &   \\
 	&             & 1
 	\end{pmatrix} \right] e(- u_{1,2}-a_{0} \cdot u_{1,3}) \ du_{1,2} \ du_{1,3}. 
 	\end{align}

 	Secondly, for each $a_{0}\in \Z$,  consider a unimodular change of variables of the form $	(u_{1,2}, u_{1,3}) \ = \  (u_{1,2}', u_{1,3}') \cdot \begin{psmallmatrix}
 	1     &  \\
 	-a_{0} & 1
 	\end{psmallmatrix}$. One can readily observe that
 	\begin{align*}
 	\begin{pmatrix}
 	1 & u_{1,2} & u_{1,3}  \\
 	& 1          & \\
 	&             & 1
 	\end{pmatrix} 
 	\ = \  
 	\begin{pmatrix}
 	1 &                     &\\
 	&  1          &   \\
 	&   a_{0}   & 1
 	\end{pmatrix}
 	\begin{pmatrix}
 	1 & u_{1,2}' & u_{1,3}'  \\
 	& 1           & \\
 	&              & 1
 	\end{pmatrix} 
 	\begin{pmatrix}
 	1 &            &         \\
 	& 1 &  \\
 	& -a_{0} & 1
 	\end{pmatrix}. 
 	\end{align*}
 	Together with the automorphy of $\Phi$ with respect to  $\Gamma_{3}$, we have
 	\begin{align}
 	& \int_{0}^{1} \Phi\left[
 	\begin{pmatrix}
 	1 & u_{1,2} &  \\
 	& 1          & \\
 	&             & 1
 	\end{pmatrix} 
 \begin{pmatrix}
 y_{0}y_{1} &             &  \\
 & y_{0}    &   \\
 &             & 1
 \end{pmatrix}\right] e(- a_{2}\cdot u_{1,2})\ du_{1,2} \nonumber\\
 	& \hspace{60pt} \ = \  \sum_{a_{0}=-\infty}^{\infty}  \int_{\Z^2\setminus \R^2} \ \Phi\left[
 	\begin{pmatrix}
 	1 & u_{1,2}' & u_{1,3}'  \\
 	& 1           &                \\
 	&              & 1
 	\end{pmatrix} 
 	\begin{pmatrix}
 	1 &            &           \\
 	& 1 &      \\
 	& -a_{0} & 1
 	\end{pmatrix}
 	\begin{pmatrix}
 	y_{0}y_{1} &             &  \\
 	& y_{0}    &   \\
 	&             & 1
 	\end{pmatrix}\right] e(- u_{1,2}') \ du_{1,2}' \ du_{1,3}'. 
 	\end{align}
 	
 	The result follows from a third and final Fourier expansion along the abelian subgroup $\left\{ \begin{psmallmatrix}
 	1 &     &    \\
 	& 1  & *  \\
 	&     & 1
 	\end{psmallmatrix}\right\}$:
 	\begin{align}
 	&\int_{0}^{1} \Phi\left[
 	\begin{pmatrix}
 	1 & u_{1,2} &  \\
 	&    1       & \\
 	&             & 1
 	\end{pmatrix} 
 	\begin{pmatrix}
 	y_{0}y_{1} &             &  \\
 	& y_{0}    &   \\
 	&             & 1
 	\end{pmatrix}\right] e(- u_{1,2}) \ du_{1,2} \nonumber\\
 & \hspace{100pt} 	\ = \  \sum_{a_{0}, a_{1}=-\infty}^{\infty} \int_{0}^{1} \int_{0}^{1} \int_{0}^{1} 
 	\Phi\left[
 	\begin{pmatrix}
 	1 & u_{1,2} & u_{1,3} \\
 	& 1          &  u_{2,3} \\
 	&             &   1
 	\end{pmatrix}
 	\begin{pmatrix}
 	1 &            &           \\
 	& 1 &     \\
 	&-a_{0} & 1
 	\end{pmatrix}
 	\begin{pmatrix}
 	y_{0}y_{1} &             &  \\
 	& y_{0}    &   \\
 	&             & 1
 	\end{pmatrix}\right] \nonumber\\
 	\nonumber\\
 	&\hspace{210pt} \cdot e(- u_{1,2}-a_{1}\cdot u_{2,3}) \ du_{1,2} \ du_{1,3}\  du_{2,3}. \nonumber
 	\end{align}
 \end{proof}

We then explicate  Proposition  \ref{incomf} when  $\Phi$ is a Maass cusp form of $\Gamma_{3}$. This constitutes  the \textit{basic identity} of the present article.  Theorem \ref{maingl3gl2} is a natural consequence of this identity and  the  diagonal/ off-diagonal structures on the dual side become apparent (see Proposition \ref{structure}). 
 
 \begin{cor}\label{incomexpf}
 	Suppose $\Phi$ is a  Maass cusp form of $\Gamma_{3}$. Then
 	\begin{align}\label{incomexpform}
 	\hspace{-10pt}\int_{0}^{1} \Phi\left[
 	\begin{pmatrix}
 	1 & u_{1,2} &  \\
 	&    1       & \\
 	&             & 1
 	\end{pmatrix} 
 	\begin{pmatrix}
 	y_{0}y_{1} &             &  \\
 	& y_{0}    &   \\
 	&             & 1
 	\end{pmatrix}\right] & e(-u_{1,2}) \ du_{1,2}  \nonumber\\
 &\hspace{-90pt} 	\ = \ \sum_{a_{1}\neq 0}  \frac{\mathcal{B}_{\Phi}(a_{1},1)}{|a_{1}|}   \cdot  W_{\alpha(\Phi)}\left( |a_{1}|y_{0}, y_{1}\right) \nonumber\\
 	&\hspace{-10pt} + \sum_{a_{0}\neq 0} \sum_{a_{1}\neq 0}  \  \frac{\mathcal{B}_{\Phi}(a_{1},1)}{|a_{1}|} \cdot
 	 W_{\alpha(\Phi)}\left( \frac{|a_{1}|y_{0}}{1+(a_{0}y_{0})^2}, \  y_{1} \sqrt{1+(a_{0}y_{0})^2}  \right) \nonumber\\
 	 &\hspace{90pt} \cdot e\left(-\frac{a_{0}a_{1}y_{0}^2}{1+(a_{0}y_{0})^2}\right). 
 	\end{align}
 	
 \end{cor}

 \begin{proof}
 	By cuspidality, $(\widehat{\Phi})_{(0,1)}\equiv 0$. The result follows from a straight-forward  linear algebra calculation: % in terms of  the Iwasawa decomposition:
 	\begin{align}\label{centralmaid}
 	\begin{pmatrix}
 	1 &           &        \\
 	& 1        &     \\
 	& -a_{0} & 1 
 	\end{pmatrix} 
 	\begin{pmatrix}
 	y_{0}y_{1} &             &  \\
 	& y_{0}    &   \\
 	&             & 1
 	\end{pmatrix} 	\ \equiv \ 
 	\begin{pmatrix}
 	1 &     &    \\
 	& 1  & -\frac{a_{0}y_{0}^2}{1+(a_{0}y_{0})^2} \\
 	&     &    1
 	\end{pmatrix} 
 	\begin{pmatrix}
 	\frac{y_{0}}{1+(a_{0}y_{0})^2} \cdot 	y_{1} \sqrt{1+(a_{0}y_{0})^2} & & \\
 	& \frac{y_{0}}{1+(a_{0}y_{0})^2}   &    \\
 	&                                                  &   1
 	\end{pmatrix}
 	\end{align}
 under the right quotient by $O_{3}(\R)\cdot \R^{\times}$. This can be verified by the formula stated in Section 2.4 of  \cite{Bu18} or the  mathematica command \textit{IwasawaForm[]}  in the \textit{GL(n)pack} (\textit{gln.m}).  The user manual and the package can be downloaded from Kevin A. Broughan's website: \url{https://www.math.waikato.ac.nz/~kab/glnpack.html}.
 \end{proof}
 
 \begin{comment}
 \diag\left(	\frac{y_{0}}{1+(a_{0}y_{0})^2} \cdot 	y_{1} \sqrt{1+(a_{0}y_{0})^2},  \ \frac{y_{0}}{1+(a_{0}y_{0})^2}, \ 1 \right)
 \end{comment}
 
 \begin{comment}
 \begin{lstlisting}
 SetDirectory[NotebookDirectory[]];<< gln.m
 A={{1,0,0}, {0,1,0}, {0, -a[0],1}}; Y={{y[0]*y[1], 0,0},{0, y[0], 0},{0,0,1}};
 R = IwasawaForm[A.Y]
 \end{lstlisting}
 \end{comment}

 %%%%%%%%%%%%%%%%%%%%%%%%%%%%%%%%%%%%%%%%%%%%%%%%%%%%%%%%%%%%%%%%%%%%%%%%%%%%%%%%%%%%%%%%%%%%%%%%%%%%%%%%%%%%%%%%%%%%%%%%%%%%

 	\subsection{Initial Simplification and Absolute Convergence}\label{inisect}
 	
 	\begin{comment}
 	Since $\Phi$ is a Maass cusp form of $\Gamma_{3}$, the inner product  $	\left \langle P, \  \mathbb{P}_{2}^{3} \Phi\cdot |\det *|^{\overline{s}-\frac{1}{2}} \right\rangle_{L^{2}(\Gamma_{2}\setminus GL_{2}(\R))}$
 	converges absolutely uniformly on every vertical strip and thus is an entire function in $s$. 	However, upon unfolding with the Poincare series, it is essential to restrict ourselves to the half-plane.  In the process of proving Theorem \ref{maingl3gl2} (i.e., Section \ref{inisect}-\ref{2stepana}), 
 	\end{comment}
 	
 	We temporarily restrict ourselves to  the vertical strip $  1+\frac{\theta}{2}< \sigma: =\re s < 4$. As we shall see, this guarantees the absolute convergence of all sums and  integrals.

Suppose $H\in\mathcal{C}_{\eta}$ with $\eta > 40$  (see Proposition \ref{inKLconv}). Then  the bound (\ref{decbdd}) for $h:= H^{\flat}$  implies its Mellin transform $\widetilde{h}(w) := \int_{0}^{\infty} \ h(y) y^{w} \ d^{\times} y$  is holomorphic on the strip $|\re w|< \eta$. Substituting (\ref{incomexpform}) into  (\ref{alunf}), and apply the changes of variables $y_{0}\to |a_{1}|^{-1} y_{0}$, $y_{0}\to |a_{0}|^{-1} y_{0}$ to the first, second piece of the resultant, 
 	\begin{align}
 \left( P,\  \mathbb{P}_{2}^{3} \Phi\cdot |\det *|^{\overline{s}-\frac{1}{2}} \right)_{\Gamma_{2}\setminus GL_{2}(\R)}   \ = \  &2\cdot  L(2s, \Phi) \cdot \int_{0}^{\infty} \int_{0}^{\infty} h(y_{1})\cdot (y_{0}^2 y_{1})^{s-\frac{1}{2}} \cdot  W_{-\alpha(\Phi)}(y_{0}, y_{1}) \ \frac{dy_{0} dy_{1}}{y_{0}y_{1}^2} \label{firstcc} \\
 	& \hspace{100pt} \ + \  OD_{\Phi}(s), \nonumber
 	\end{align}
 	where
 	
 	% In fact, by differentiation under integral sign and the same argument for  Lemma \ref{inKLconv}, we have bounds for the derivatives of $h$. Integrating by parts a few times, one easily sees that $\left|\widetilde{h}(w)\right| \ll (1+ |\im w|)^{-3} $ on $|\re w|<\eta$.  

 	\begin{comment}
 	
 	Let $M>20$ be an absolute constant. We consider the class of  holomorphic functions  $\widetilde{h}(w)$  on the vertical strip $-2M< \re w <2M$ that satisfies the bound $\left|\widetilde{h}(w)\right| \ll (1+ |\im w|)^{-3} $ on such a strip. Let $h: (0, \infty)\rightarrow \C$ be given  by
 	\begin{align}
 	h(y) \ := \  \int_{(\sigma_{w})} \widetilde{h}(w) y^{-w} \ \frac{dw}{2\pi i} \hspace{20pt} ( \sigma_{w} \ \in \ (-2M, 2M) ). 
 	\end{align}
 	Denote by  $\mathcal{H}_{M}$  the class of such $h$'s.  Trivial estimation gives \footnote{ 	In particular, we have  (\ref{poincbdd}). }
 	\begin{align}\label{fastdecay}
 	\left|h(y)\right| \ \ll \  \min\left\{ y,  y^{-1}    \right\}^{M}.
 	\end{align}
 	
 	\end{comment}

 	\begin{defi}\label{ofdef}
 		\begin{align}\label{secondcc}
 		OD_{\Phi}(s) \ := \ &\sum_{a_{0}\neq 0} \sum_{a_{1}\neq 0}   \frac{\mathcal{B}_{\Phi}(1,a_{1})}{|a_{0}|^{2s-1}|a_{1}|} \cdot  \int_{0}^{\infty} \int_{0}^{\infty} h(y_{1})\cdot (y_{0}^2 y_{1})^{s-\frac{1}{2}} \cdot e\left( \frac{a_{1}}{a_{0}}\cdot \frac{y_{0}^2}{1+y_{0}^2}\right) \nonumber\\
 		& \hspace{160pt} \cdot   	W_{-\alpha(\Phi)}\left( \left|\frac{a_{1}}{a_{0}}\right| \cdot \frac{y_{0}}{1+y_{0}^2},  \ y_{1} \sqrt{1+y_{0}^2} \right)    \frac{dy_{0} dy_{1}}{y_{0}y_{1}^2}. 
 		\end{align}
 
 	\end{defi}

  	%$\mathcal{B}_{\Phi}(1, a_{1})  =  \mathcal{B}_{\Phi}(1, -a_{1})$
 	%%%%%%%%%%%%%%%%%%%%%%%%%%%%%%%%%%%%%%%%%%%%%%

 	%%%%%%%%%%%%%%%%%%%%%%%%%%%%%%%%%%%%%%%%%%%%%%%%%%%%%%%%%%%%%%%%%%%%%%%%%%%%%%%%%%%%%%%%%%%%%%%%%%%%%%%%%%%%%%%%%%%%%%%%%%%%%%%%%%%%%%%%%%%%%%%%%%%%%%%%%%%%%%%%%%%%%%%%

 		%the diagonal piece (\ref{firstcc}) converges absolutely. 
 	
 	\begin{prop}\label{diagprop}
 		When $H \in \mathcal{C}_{\eta}$ and \  $4>	\sigma  >  \frac{1+\theta}{2}$,  we have
 		\begin{align}\label{diagstade}
 		\int_{0}^{\infty} \int_{0}^{\infty} h(y_{1})& \cdot (y_{0}^2 y_{1})^{s-\frac{1}{2}} \cdot  W_{-\alpha(\Phi)}(y_{0}, y_{1}) \ \frac{dy_{0} dy_{1}}{y_{0}y_{1}^2}  \nonumber\\
 &	\hspace{50pt} \ =\  	\frac{ \pi^{-3s}}{8}  \cdot  \int_{(0)}  \frac{H(\mu)}{|\Gamma(\mu)|^2} \cdot \prod_{i=1}^{3} \  \Gamma\left(\frac{s+\mu-\alpha_{i}}{2}\right)\Gamma\left(\frac{s-\mu-\alpha_{i}}{2}\right) \ \frac{d\mu}{2\pi i}. 
 		\end{align}
 		   %\footnote{ Note that $\Phi$ is assumed to be tempered. }
 	\end{prop}

 	\begin{proof}
 	From Proposition  \ref{plancherel}, we have
 		\begin{align}
 		\int_{0}^{\infty} \int_{0}^{\infty} h(y_{1}) \cdot (y_{0}^2 y_{1})^{s-\frac{1}{2}} &\cdot  W_{-\alpha(\Phi)}(y_{0}, y_{1}) \ \frac{dy_{0} dy_{1}}{y_{0}y_{1}^2} \nonumber\\
 			\ &  \hspace{30pt } = \ \frac{1}{2}\cdot \int_{(0)} \ \frac{H(\mu)}{|\Gamma(\mu)|^2}\cdot  \int_{0}^{\infty} \int_{0}^{\infty}  W_{\mu}(y_{1})  W_{-\alpha(\Phi)}(y_{0}, y_{1})  (y_{0}^2 y_{1})^{s-\frac{1}{2}} \  \frac{dy_{0} dy_{1}}{y_{0}y_{1}^2}  \  \frac{d\mu}{2\pi i}.   \nonumber 
 		\end{align}
 	The $y_{0}$, $y_{1}$-integrals can be evaluated by Proposition \ref{stadediff1} and  (\ref{diagstade}) follows.  Moreover, the right side of (\ref{diagstade})  is holomorphic on $\sigma >0$.   
 		\end{proof}
 		
 		\begin{comment}
 		Readers may wish to keep a  concrete example of a test function   in mind. For instance,
 		\end{comment}

 		%(\ref{diagstade}) converges absolutely on $\re s>0$ by Stirling's formula and  (\ref{rapdeca}). %Now, the analytic continuation for the diagonal piece to $\re s>0$ follows from that of $L(2s,\Phi)$. 

 	 \begin{prop}\label{odconv}
 	 	The off-diagonal  $OD_{\Phi}(s)$ converges absolutely when $4  >  \sigma  >  1+\frac{\theta}{2}$  and $H \in \mathcal{C}_{\eta}$ ($\eta>40$).    
 	 \end{prop}

 	 \begin{proof}
 	 	Upon inserting absolute values, breaking up the $y_{0}$-integral into $\int_{0}^{1}+\int_{1}^{\infty}$, and applying  the bounds (\ref{whitest3})   and  $|\mathcal{B}_{\Phi}(1, a_{1})| \ll |a_{1}|^{\theta}$, observe that 
 	 	\begin{align*}
 	 OD_{\Phi}(s) \ \ll \ 	&\sum_{a_{0}=1}^{\infty} \sum_{a_{1}=1}^{\infty}   \frac{1}{a_{0}^{2\sigma-1}a_{1}^{1-\theta}} \left(\int_{y_{0}=1}^{\infty}+\int_{y_{0}=0}^{1}\right) \int_{y_{1}=0}^{\infty} |h(y_{1})| (y_{0}^2 y_{1})^{\sigma-\frac{1}{2}} \left(\frac{a_{1}a_{0}^{-1}y_{0}}{1+y_{0}^2}\right)^{A_{0}}  \left(y_{1} \sqrt{1+y_{0}^2} \right)^{A_{1}}  \   \frac{dy_{0} dy_{1}}{y_{0}y_{1}^2},
 	 	\end{align*}
 	 	where the implicit constant depends only on $\Phi$, $A_{0}$, $A_{1}$ with  $-\infty < A_{0}, A_{1}< 1$.  We are allowed to choose different $A_{0}, A_{1}$ in different ranges of the $y_{0}$, $y_{1}$-integrals. 
 	 	
 	 The convergence of both of the series is guaranteed if
 	 	\begin{align}\label{rana0sig}
 	 	A_{0} \ < \ -\theta  \hspace{10pt} \text{ and }  \hspace{10pt} \sigma \ > \ 1-\frac{A_{0}}{2}. 
 	 	\end{align}

 We now show that if  (\ref{rana0sig}) and
 		\begin{align}\label{rana1}
 		A_{1} \ < \ A_{0}-2\sigma+1
 		\end{align}
 		both hold, then the $y_{0}$-integrals converge.  Indeed, observe that $2\sigma+A_{0}-2>-1$ (by (\ref{rana0sig})),  and 
 	\begin{align}
 	\int_{y_{0}=0}^{1}  y_{0}^{2\sigma+A_{0}-2} \left(1+y_{0}^2\right)^{\frac{A_{1}}{2}-A_{0}} \ dy_{0} \ \asymp_{A_{0}, A_{1}} \ 	\int_{y_{0}=0}^{1}    y_{0}^{2\sigma+A_{0}-2} \ dy_{0}. \nonumber
 	\end{align}
So, the last integral converges. Also,  (\ref{rana0sig}) and (\ref{rana1}) imply $A_{1}< \min\{ 1, 2A_{0} \}$ and thus, 
 	\begin{align*}
 	\int_{y_{0}=1}^{\infty}  y_{0}^{2\sigma+A_{0}-2} \left(1+y_{0}^2\right)^{\frac{A_{1}}{2}-A_{0}} \ dy_{0} 
 	\ \le \ \int_{y_{0}=1}^{\infty} y_{0}^{2\sigma+A_{1}-A_{0}-2} \ dy_{0}. 
 	\end{align*}
The last integral converges because of (\ref{rana1}).

 	 	 	 	For the $y_{1}$-intgeral, the integrals 
 	 	 	 	\begin{align*}
 	 	 	 	\int_{y_{1}=1}^{\infty} |h(y_{1})| y_{1}^{\sigma+A_{1}-\frac{5}{2}} \ dy_{1}  \hspace{10pt} \text{ and } \hspace{10pt}  	\int_{y_{1}=0}^{1} |h(y_{1})| y_{1}^{\sigma+A_{1}-\frac{5}{2}} \ dy_{1}
 	 	 	 	\end{align*}
 	 	 	 	converge whenever  $H\in \mathcal{C}_{\eta}$ (we then have (\ref{decbdd})) and 
 	 	 	 	\begin{align}\label{etasuff}
 	 	 	 	\eta  \ > \   \left| \sigma +A_{1}- \frac{3}{2} \right|. 
 	 	 	 	\end{align}

 	 		 		Let $\delta:= \sigma-1-(\theta/2) \  (>0)$. In view of  (\ref{rana0sig}) and (\ref{rana1}), we may take $A_{0} :=   -\theta -\delta$ and  $A_{1}  :=  -2\theta-1-4\delta$. Also, (\ref{etasuff}) trivially holds as $\eta >40$ and $\sigma<4$.  The result follows.	 
 	 \end{proof}

 	 	\begin{comment}

 	 	instead of developing new framework as in \cite{Ne20+, Wu21+, BFW21+}.

 	 	Our method is `period-like' when viewed from the structural perspective, whereas  it is `Kuznetsov-like' when viewed from the analytic (or archimedean) perspective.

 	 	\end{comment}

 	 \begin{rem}
 	 	Readers will have no trouble in realizing   the resemblance of (\ref{mainobj})  to the well-known inner product construction for the Kuznetsov formula.   Indeed, $\mathbb{P}_{2}^{3}\Phi$ can be regarded as an (infinite sum of) Poincar\'e series of $SL_{2}(\Z)$ thanks to its Fourier expansion. (We never adopt this approach in this article.) In a sense, this can be considered as  a $GL(3)\times GL(2)$ analog of the Kuznetsov formulae.  However, there are some differences. One of them has been mentioned: our moment identity is an equality between two unfoldings instead of that between spectral and geometric  expansions. 
 	 	
 	 	The other is on the technical aspect. In the Kuznetsov formula, it is possible to annihilate the  oscillatory factors therein to obtain a primitive form of the trace formula  with some applications, see  \cite{GK13}, \cite{Zh14}, \cite{GSW21}. However, such a treatment is far from sufficient  in our case ---  we have not analytically continued into the critical strip in Proposition  \ref{odconv}!  In other words, the oscillatory factor in $OD_{\Phi}(s)$ is of intrinsic importance to our problem. It arises naturally from the  abstract characterization of Whittaker functions. 
 	 	\end{rem}

 	 \section{Structure of the Off-diagonal}\label{separaOD}

 	Fix  $\epsilon:=1/100$ (say), $0<\phi<\pi/2$,  and consider  the domain $1+ \frac{\theta}{2}+\epsilon < \sigma <4 $ in this section to maintain absolute convergence.   We will stick with this choice of $\epsilon$ for the rest of this article and the number $\phi$ here should not pose any confusion with the basis of cusp forms $(\phi_{j})$ of $\Gamma_{2}$. We define a perturbed version of  $OD_{\Phi}(s)$ as follows:
 	 \begin{align}\label{pertodd}
 	 OD_{\Phi}(s;\phi)  \ := \ &\sum_{a_{0}\neq 0} \sum_{a_{1}\neq 0}   \frac{\mathcal{B}_{\Phi}(1,a_{1})}{|a_{0}|^{2s-1}|a_{1}|} \int_{0}^{\infty} \int_{0}^{\infty} h(y_{1})\cdot (y_{0}^2 y_{1})^{s-\frac{1}{2}}
 	 W_{-\alpha(\Phi)}\left(\left|\frac{a_{1}}{a_{0}}\right|\frac{y_{0}}{1+y_{0}^2},  \ y_{1} \sqrt{1+y_{0}^2} \right)  \nonumber\\
 	 & \hspace{150pt}\cdot e\left(\frac{a_{1}}{a_{0}}\frac{y_{0}^2}{1+y_{0}^2}; \phi\right) \   \frac{dy_{0} dy_{1}}{y_{0}y_{1}^2},
 	 \end{align}
 	 where 
 	 \begin{align}\label{perturb}
 	 e\left( x; \phi\right) \ :=  \ \int_{(\epsilon)} |2\pi x|^{-u} e^{iu\phi \sgn(x)} \Gamma(u) \ \frac{du}{2\pi i} \hspace{30pt} (x\in \R-\{0  \}).
 	 \end{align}
  In Proposition \ref{pullout}, we will show that
 	\begin{align}
 	\lim_{\phi \to \pi/2} \ OD_{\Phi}(s; \phi) \ = \  OD_{\Phi}(s)
 \end{align}
on a smaller region of absolute convergence.

 	 \begin{rem}\label{Goa}
 	 The goals of this section  is to obtain an expression of $OD_{\Phi}(s; \phi)$ that
 	 \begin{itemize}
 	 	\item  reveals the  structure of the dual moment;
 	 	
 	 	\item can be analytically continued into the critical strip;
 	 	
 	 	\item and will allow us to pass to the limit $\phi\to \pi/2$ (in the critical strip). 
 	 \end{itemize}
 	 \end{rem}

 	In view of these, it is natural to work on the dual side of Mellin transforms. Also, we will be able to separate variables as an added benefit. The main result of this section is as follows:

 	 \begin{prop}[Dual Moment]\label{structure}
 	 	Let $H\in \mathcal{C}_{\eta}$ ($\eta>40$) and $\phi \in (0, \pi/2)$. On  the vertical strip
 	 	\begin{align}\label{abdom}
 	 	1+ \frac{\theta}{2} + \epsilon \ < \ \sigma \ < \ 4,
 	 	\end{align}	
 	 	we have
 	 	\begin{align}\label{mainod}
 	 	OD_{\Phi}(s; \phi) \ = \   &\frac{1}{4}  \ \int_{(1+\theta+2\epsilon)} \  \zeta\left(2s-s_{0}\right) L\left(s_{0},  \Phi\right) \cdot \sum_{\delta=\pm} \left(\mathcal{F}_{\Phi}^{(\delta)}H\right)\left(s_{0}, \ s; \  \phi\right) \ \frac{ds_{0}}{2\pi i},
 	 	\end{align}
 	 	where the transform of $H$ is given by
 	 	\begin{align}
 	 	\left(\mathcal{F}_{\Phi}^{(\delta)}H\right)\left(s_{0}, \ s; \ \phi\right) \  :=\  	& \int_{(15)} \int_{(\epsilon)}  \  \widetilde{h}\left(s-s_{1}-\frac{1}{2}\right)    \cdot \mathcal{G}_{\Phi}^{(\delta)}\left( s_{1}, u; \ s_{0}, s;  \ \phi \right) \   \frac{du}{2\pi i}   \frac{ds_{1}}{2\pi i }, \label{doubmbtrans}
 	 	\end{align}
 	 	with $h:= H^{\flat}$,  $G_{\Phi}:= G_{\alpha(\Phi)}$ as defined  in (\ref{vtgamm}), and
 	 	\begin{align}\label{kernelG}
 	 	\mathcal{G}_{\Phi}^{(\delta)}\left( s_{1}, u; \ s_{0}, s;  \ \phi \right) \ := \  &  G_{\Phi}\left(s_{0}-u, s_{1}\right)   \cdot   (2\pi )^{-u} e^{i\delta \phi u}  \Gamma(u) \cdot \frac{\Gamma\left( \frac{u+1-2s+s_{1}-s_{0}}{2}\right) \Gamma\left(\frac{2s-s_{0}-u}{2}\right) }{\Gamma\left(\frac{1+s_{1}}{2}-s_{0}\right)}. 
 	 	\end{align}

 	 \end{prop}

%%%%%%%%%%%%%%%%%%%%%%%%%%%%%%%%%%%%%%%%%%%%%%%%%%%%%%%%%%%%%%%%%%%%%%%%%%%%%%%%%%%%%%%%%%%%%%%%%%%%%%%%%%%%%%%%%%%%%%%%%%%%%%%%%%%%%%%%%%%%%%%%%%%%%%%%%%%%%%%%%%%%%%%%%%%%%%%%%%%%%%%%%%%%%%%%%%%%%%%%%%%%%%%%%%%%%%%%%%%%%%%%%%%%%%%%%%%%%%%%%%%%%%%%%%%%%%%%%%%%%%%%%%%%%%%%%%%%%%%%%%%%%%%%%%%%%%%%%%%%%%%%%%%%%%%%%%%%%%%%%%%%%%%%%%%%%%%%%%%%%%%%

 	 \begin{comment}
 	 
 	 \begin{rem}
 	 When $\phi=\pi/2$, the integral (\ref{perturb}) only converges conditionally which is not ideal for many analytic manipulations. 
 	 \end{rem}
 	 
 	 although we are not shifting the line of integration for the $u$-integral before Section \ref{expevatra}, 
 	 
 	 obtained right after the application of Stirling's estimates (see Lemma \ref{expset}),  the exponential decay contributed from the double Mellin transforms of  $W_{-\alpha(\Phi)}(\cdots)$ in (\ref{pertodd}) as $|\im s_{0}| = |\im u|\to \infty$ vanishes completely!   
 	 
 	 It would be more convenient to perform the desired contour-shifting when $\phi \in (0, \pi/2)$ so that there is  exponential decay provided by the factor $e^{-(\frac{\pi}{2}- \phi) |\im u|}$.  
 	 
 	 \end{comment}    
 	 
 	 	%(and of course $\phi \in (0, \pi/2)$) to ensure the absolute convergence of all of the sums and integrals.
\begin{proof}
 	 Plug-in the expression of  $W_{-\alpha(\Phi)}$ described in  Proposition \ref{vtmellin} into $OD_{\Phi}(s; \phi)$ with
 	 	\begin{align}\label{absconassu}
 	 	\sigma_{1} \ := \ 15 \hspace{10pt} \text{ and } \hspace{10pt}  	1+\theta \ < \ \sigma_{0}\ < \ 2\sigma-1-\epsilon.
 	 	\end{align}
 	  Inserting absolute values to the resulting expression, the sums and integrals are bounded by 
 	 	\begin{align}\label{absedintsum}
 	 	 \sum_{\delta:= \sgn(a_{0}a_{1})= \pm }\left( \sum_{a_{0}\neq 0} \frac{1}{|a_{0}|^{2\sigma-\sigma_{0}-\epsilon}}\right) &\left( \sum_{a_{1}\neq 0} \frac{\left|\mathcal{B}_{\Phi}(1,a_{1})\right|}{\left|a_{1}\right|^{\sigma_{0}+\epsilon}}\right)  \left( \int_{(\sigma_{0})} \int_{(\sigma_{1})} \left| G_{\Phi}(s_{0}, s_{1})\right| \ |ds_{0}| |ds_{1}|\right)   \nonumber\\
 	 	& \hspace{-100pt}  \cdot \left(\int_{(\epsilon)} \left|  e^{i\delta \phi u } \Gamma(u)  \right| \ |du|\right) \left(\int_{0}^{\infty}  y_{0}^{-\sigma_{0}-2\epsilon+2\sigma} (1+y_{0}^2)^{\sigma_{0}+\epsilon-\frac{1+\sigma_{1}}{2}} \ d^{\times} y_{0} \right)\left( \int_{0}^{\infty} |h(y_{1})| \cdot y_{1}^{\sigma-\sigma_{1}-\frac{1}{2}} \ d^{\times} y_{1}\right). 
 	 	\end{align}
 	 	  Observe that:
 	 	\begin{itemize}
 	 		\item  by Stirling's formula,     the $s_{0}$, $s_{1}$, $u$-integrals converge as long as 
 	 		\begin{align}\label{absc1}
 	 		\sigma_{0}, \  \sigma_{1}, \ \epsilon \ > \ 0, \  \hspace{20pt}   \phi  \ \in  \ (0, \pi/2); 
 	 		\end{align}
 	 		
 	 		\item  the $y_{0}$-integral converges as long as
 	 		\begin{align}\label{absc2}
 	 		 	\sigma_{0}+2\epsilon \ < \ 2\sigma \ < \ 	\sigma_{1}-\sigma_{0}+1; 
 	 		\end{align}

 	 		\item  	by  the bound \  $\left|\mathcal{B}_{\Phi}(1, a_{1}) \right|  \ \ll \  |a_{1}|^{\theta}$, the $a_{0}$-sum and the $a_{1}$-sum converge as long as
 	 		\begin{align} \label{absc3}
 	 		2\sigma-1 \  >  \   \sigma_{0}+\epsilon \ > \ 1+\theta. 
 	 		\end{align}
 	 
 	 	\end{itemize}
 	 Under (\ref{absconassu}), items (\ref{absc1}), (\ref{absc2}), (\ref{absc3}) hold. Moreover, the $y_{1}$-integral converges by (\ref{decbdd}) and   $H\in \mathcal{C}_{\eta}$ ($\eta >40$).  Now, upon rearranging  sums and integrals, and notice that  $\mathcal{B}_{\Phi}(1, a_{1})  =  \mathcal{B}_{\Phi}(1, -a_{1})$, we have 
 	 	\begin{align}
 	 	OD_{\Phi}(s; \phi) 
 	 	\ = \   & \ 2 \sum_{\delta=\pm } \int_{(\sigma_{0})}\int_{(\sigma_{1})} \int_{(\epsilon)} \ \frac{G_{\Phi}(s_{0},s_{1})}{4} \cdot   (2\pi)^{-u} e^{i\delta \phi u } \Gamma(u) \left(\int_{0}^{\infty} h(y_{1}) y_{1}^{s-s_{1}-\frac{1}{2}}  d^{\times} y_{1}\right) \nonumber\\
 	 	&\hspace{10pt} \cdot  \left( \int_{0}^{\infty}   y_{0}^{-s_{0}-2u+2s} (1+y_{0}^2)^{s_{0}+u-\frac{1+s_{1}}{2}} \ d^{\times} y_{0} \right) 	\left(\sum_{a_{0}=1}^{\infty} \sum_{a_{1}= 1}^{\infty}   \frac{\mathcal{B}_{\Phi}(1,a_{1})}{a_{0}^{2s-1}a_{1}}   \left(\frac{a_{1}}{a_{0}}\right)^{1-s_{0}-u} \right)  \ \frac{ds_{0}}{2\pi i } \frac{ds_{1}}{2\pi i } \frac{du}{2\pi i}. \label{uponexp}
 	 	\end{align}

 	 	\begin{comment}
 	 	Let $\phi\in (0,\pi/2)$,  $\sigma_{0}, \sigma_{1}, c >0$ be constants satisfying the constraints
 	 	\begin{align}\label{constr}
 	 	\begin{cases}
 	 	1+\theta \ &< \ \sigma_{0} \ < \  2\sigma-1\\
 	 	\sigma_{1} \ &> \ \sigma_{0}+ 2\sigma-1,
 	 	\end{cases}
 	 	\hspace{20pt}  \epsilon  \ < \  \min\left\{ 2\sigma-1-\sigma_{0}, \ \sigma- \frac{\sigma_{0}}{2}   \right\}. 
 	 	\end{align}
 	 	\end{comment}
 	 	
 	 	\begin{comment}
 	 	\begin{align*}
 	 	OD_{\Phi}(\phi, s)  \ = \ &\frac{1}{4} \sum_{a_{0}\neq 0} \sum_{a_{1}\neq 0}   \frac{\mathcal{B}_{\Phi}(1,a_{1})}{|a_{0}|^{2s-1}|a_{1}|} \int_{y_{0}=0}^{\infty}  \int_{y_{1}=0}^{\infty} h(y_{1})\cdot (y_{0}^2 y_{1})^{s-\frac{1}{2}}
 	 	\nonumber\\
 	 	&\cdot \int_{(\sigma_{0})}\int_{(\sigma_{1})} G_{\Phi}(s_{0},s_{1}) \left(\left| \frac{a_{1}}{a_{0}}\right|\frac{y_{0}}{1+y_{0}^2}\right)^{1-s_{0}} \left(y_{1} \sqrt{1+y_{0}^2}\right)^{1-s_{1}}  \nonumber\\
 	 	&\hspace{70pt} \cdot  \int_{(c)}   \left(2\pi \left|\frac{a_{1}}{a_{0}}\right| \frac{y_{0}^2}{1+y_{0}^2}\right)^{-u} e^{i\epsilon \phi u } \Gamma(u) \ \frac{du}{2\pi i} \frac{ds_{0}}{2\pi i } \frac{ds_{1}}{2\pi i }     \frac{dy_{0}\  dy_{1}}{y_{0}y_{1}^2}.  
 	 	\end{align*}
 	 	\end{comment}

 	 Recall the integral identity
 	 	\begin{align}\label{eube}
 	 	\int_{0}^{\infty} y_{0}^{v}(1+y_{0}^2)^{A} \ d^{\times} y_{0} \ = \ \frac{1}{2} \   \frac{\Gamma\left(-A-\frac{v}{2}\right)\Gamma\left(\frac{v}{2}\right)}{ \Gamma(-A)}
 	 	\end{align}
 	 	for $0  <   \re v  <   -2\re A$. It follows that
 	 	\begin{align}\label{rear}
 	 	OD_{\Phi}(s; \phi) \ = \ &\ 2  \sum_{\delta=\pm} \int_{(\sigma_{0})}\int_{(\sigma_{1})} \int_{(\epsilon)} \  \zeta\left(2s-s_{0}-u\right) L\left(s_{0}+u; \Phi\right)\cdot  \widetilde{h}\left(s-s_{1}-\frac{1}{2}\right) \nonumber\\
 	 	&\hspace{70pt} \cdot \frac{G_{\Phi}(s_{0},s_{1})}{4} \cdot   (2\pi)^{-u} e^{i\delta \phi u } \Gamma(u)\cdot   \frac{1}{2} \ \frac{ \Gamma\left(s-\frac{s_{0}}{2}-u\right) \Gamma\left( \frac{1+s_{1}-s_{0}}{2}-s\right)  }{\Gamma\left(\frac{1+s_{1}}{2}-s_{0}-u\right)}  \ \frac{ds_{0}}{2\pi i } \frac{ds_{1}}{2\pi i }  \frac{du}{2\pi i}. 
 	 	\end{align}

 	 	We pick the contour $(\sigma_{0}):=(1+\theta+\epsilon)$ (we thus impose (\ref{abdom})). To  isolate  the nonarchimedean part of $OD_{\Phi}(s; \phi)$,  we make the change of variable $s_{0}' \ = \  s_{0}+u$.  Upon plugging-in the expression for  $G_{\Phi}\left(s_{0}'-u, s_{1}\right)$ (see (\ref{vtgamm})), we obtain  (\ref{mainod})-(\ref{kernelG}).  By the absolute convergence proven above, we also conclude that  the integral transform  $\left(\mathcal{F}_{\Phi}^{(\delta)}h\right)\left(s_{0}', \ s; \ \phi\right)$ is holomorphic on  the domain
 	 	\begin{align}\label{holoregio}
 	 	\sigma \ < \ 4 \hspace{30pt}  \text{ and } \hspace{30pt} 1+\theta+ \epsilon \ < \ \sigma_{0}' \ < \ 2\sigma-1.
 	 	\end{align}
 	 	This completes the proof. 
 	 	\end{proof}

 	 \begin{comment}
 	 \begin{rem} 	
 	 	In the following, we replace $s_{0}'$  by $s_{0}$  and have (\ref{holoregio}) superseding (\ref{absconassu}) correspondingly. 
 	 \end{rem}
 	 \end{comment}
 	 
 	 %%%%%%%%%%%%%%%%%%%%%%%%%%%%%%%%%%%%%%%%%%%%%%%%%%%%%%%%%%%%%%%%%%%%%%%%%%%%%%%%%%%%%%%%%%%%%%%%%%%%%%%%%%%%%%%%%%%%%%%%%%%%%%%%%%%%%%%%%%%%%%%%%%%%%%%%%%%%%%%%%%%%%%%%%%%%%%%%%%%%%%%%%%%%%%%%%%%%%%%%%%%%%%%%%%%%%%%%%%%%%%%%%%%%%%%%%%%%%%%%%%%%%%%%%%%%%%%%%%%%%%%%%%%%%%%%%%%%%%%%%%%%%%%%%%

 	  \begin{prop}\label{pullout}
 	  	For $4> \sigma > (3+\theta)/2$ and  $H\in \mathcal{C}_{\eta}$,   we have
 	  	\begin{align}
 	  	\lim_{\phi \to \pi/2} \ OD_{\Phi}(s; \phi) \ = \  OD_{\Phi}(s).
 	  	\end{align}
 	  \end{prop}

 	  \begin{proof}
 	  	Let $\epsilon:=1/100$, $\sigma_{1} := 15$,  and pick any $\sigma_{0}$ satisfying
 	  	\begin{align}\label{concond}
 	  	\frac{3}{2}+ \theta +\epsilon \ &< \ \sigma_{0} \ < \ 2\sigma-1-\epsilon. 
 	  	\end{align}
 	  	
 	  	Denote by $\mathcal{C}_{\epsilon}$  the indented path consisting of the line segments: 
 	  	\begin{align*}
 	  	-\frac{1}{2}-\epsilon-i\infty \ \rightarrow \ 	-\frac{1}{2}-\epsilon-i \ \rightarrow \  \epsilon-i \ \rightarrow \  \epsilon+i \  \rightarrow \ -\frac{1}{2}-\epsilon+i \ \rightarrow \ -\frac{1}{2}-\epsilon+i\infty.  
 	  	\end{align*}
 	  	Replace $e(x; \phi)$ in (\ref{uponexp}) by the expression: 
 	  	\begin{align}\label{melexp}
 	  	e(x; \phi) \ =  \ \int_{\mathcal{C}_{\epsilon}} |2\pi x|^{-u} e^{iu \phi \sgn(x)} \Gamma(u) \ \frac{du}{2\pi i}. 
 	  	\end{align}
 	  	Note that $	\left| e^{iu\phi \sgn(x)} \Gamma(u) \right|  \ \ll_{\epsilon} \  \left(1 +|\im u|\right)^{-1-\epsilon}$   for $u\in \mathcal{C}_{\epsilon}$ and  $\phi \in (0, \pi/2]$ . Insert absolute values in (\ref{uponexp}). The resulting sums and integrals converge absolutely when  $\phi \in (0, \pi/2]$ and  (\ref{concond}) holds, which can be seen by the same argument following (\ref{absedintsum}).  Apply Dominated Convergence and shift the contour of the $u$-integral  to $-\infty$,  the residual series obtained is exactly $e\left(\frac{a_{1}}{a_{0}}\frac{y_{0}^2}{1+y_{0}^2}\right)$. This completes the proof. 
 	  \end{proof}

 	 	Now, $OD_{\Phi}(s; \phi)$ is in terms of  integrals of Mellin-Barnes type.  Note that the $\Gamma$-factors from Proposition \ref{vtmellin} and  (\ref{perturb})  alone  are not sufficient for our goals (see Remark \ref{Goa} and  (\ref{absc1}), (\ref{absc2}), (\ref{absc3})).  The three extra $\Gamma$-factors  brought by the $y_{0}$-integral, which  `mix' all variables of integrations, will play an important role  in Section \ref{Stirl}-\ref{2stepana}.

 	 	\section{Analytic Properties of the Archimedean Transform}\label{Stirl}
 	 	
 	 	In (\ref{mainod}), the factors $\zeta\left(2s-s_{0}\right)$ and $ L\left(s_{0},  \Phi\right)$ are known to admit holomorphic continuation and have polynomial growth in vertical strips (except on the  line $2s-s_{0}=1$). We also have to study the  archimedean part of (\ref{mainod}), i.e., the integral transform 
 	 	\begin{align}\label{mainintrans}
 	 	\left(\mathcal{F}_{\Phi}^{(\delta)}H\right)\left(s_{0}, \ s; \ \phi\right) \  &:=\  	 \int_{(15)} \int_{(\epsilon)}  \  \widetilde{h}\left(s-s_{1}-\frac{1}{2}\right)    \cdot \mathcal{G}_{\Phi}^{(\delta)}\left( s_{1}, u; \ s_{0}, s;  \ \phi \right) \   \frac{du}{2\pi i}   \frac{ds_{1}}{2\pi i },
 	 	\end{align}
 	 	where $h:= H^{\flat}$ and   $\mathcal{G}_{\Phi}^{(\delta)}(\cdots)$ as defined in (\ref{kernelG}).  In Section \ref{separaOD}, we have shown that when $\phi \in (0, \pi/2)$,  the function  $ (s_{0},s) \mapsto	\left(\mathcal{F}_{\Phi}^{(\delta)}h\right)\left(s_{0}, \ s; \ \phi\right)$
 	 	is holomorphic   on the domain (\ref{holoregio}), i.e., 
 	 	\begin{align*}
 	 	\sigma \ < \ 4 \hspace{30pt}  \text{ and } \hspace{30pt} 1+\theta+ \epsilon \ < \ \sigma_{0} \ < \ 2\sigma-1. 
 	 	\end{align*} 
 	 	
 	 	 In this section, we establish a larger region of holomorphy for  $(s_{0}, s)\mapsto \left(\mathcal{F}_{\Phi}^{(\delta)} H\right)(s_{0}, s; \phi)$ that holds for  $\phi \in (0, \pi/2]$.   We write
 	 	 \begin{align*}
 	 	 s \ = \ \sigma+it, \hspace{15pt}  s_{0}=\sigma_{0}+it_{0},  \hspace{15pt} s_{1}=\sigma_{1}+it_{1},  \hspace{15pt} \text{ and }  \hspace{15pt} u 
 	 	 = \ \epsilon+iv, 
 	 	 \end{align*}
 	 	 with $\epsilon:=1/100$. It is sufficient to consider  $s$  inside the rectangular box  $\epsilon< \sigma<4$ and $|t|\le T$, for any given  $T\ge 1000$.  Moreover,  $\alpha_{k}:= i\gamma_{k} \in i\R$ $(k=1,2,3)$ by our assumptions on $\Phi$. The main result of  this section can be stated as follows: \newline

 	 	 % which will contain our point of interest  $(\sigma_{0}, \sigma)=(1/2, 1/2)$ and 

 	 	\begin{comment}
 	 	Notice that  $\zeta(2s-s_{0})$ and $L(s_{0}, \Phi)$ have polynomial growth outside the regions of absolute convergence. One must first show that there is enough decay 
 	 	\end{comment}
 	 	
 	 	%For the ease of exposition (though not strictly necessary), we make use of the temperedness assumption of $\Phi$ below. 

 	 	\begin{prop}\label{anconpr}
 	 	Suppose  $H\in \mathcal{C}_{\eta}$.  
 	 		
 	 		\begin{enumerate}
 	 		
 	 		\item For any $\phi \in (0, \pi/2]$, the transform $\big(\mathcal{F}_{\Phi}^{(\delta)}H\big)(s_{0}, s; \phi)$ is holomorphic on the domain
 	 		\begin{align}\label{newdomain}
 	 		\sigma_{0} \ > \ \epsilon,  \hspace{15pt}  \sigma \ < \ 4, \hspace{15pt} \text{ and } \hspace{15pt}  2\sigma-\sigma_{0}-\epsilon \ > \ 0.  
 	 		\end{align}
 	 		
 	 		\item Whenever $(\sigma_{0}, \sigma) \in (\ref{newdomain})$,  $|t| <T$, and $\phi \in (0, \pi/2)$, the transform $\left(\mathcal{F}_{\Phi}^{(\delta)}H\right)\left(s_{0}, \ s; \ \phi\right)$ has exponential decay as  $|t_{0}| \to \infty$.   Note: The explicit estimate is stated in the proof below and the implicit constant depends only on $T$ and $\Phi$. 
 	 		\end{enumerate}
 	 		
 	 		\begin{comment}
 	 		\begin{align}\label{estimatrans}
 	 		\left|	\left(\mathcal{F}_{\Phi}^{(\delta)}h\right)\left(s_{0}, \ s; \ \phi\right) \right| \ \ll_{T, \Phi} \  |t_{0}|^{A} e^{-\frac{\pi}{4}  \log^2 (3+|t_{0}|)}
 	 		\end{align}
 	 		for some absolute constant $A>0$. 
 	 		\end{comment}

 	 	\end{prop}

 	 	\begin{comment}
 	 	\begin{align}
 	 	\left(\mathcal{F}_{\Phi}^{(\delta)}h\right)\left(s_{0},  s; \phi \right) \ \ll_{\Phi, M, T} \   |t_{0}|^{8-M} \cdot e^{-\left( \frac{\pi}{2}-\phi\right)|t_{0}|}. 
 	 	\end{align}
 	 	\end{comment}
 	 	
 	 	\begin{rem}\label{nopolerem}
 	 		The domain (\ref{newdomain}) is chosen in a way that the function $(s_{0}, s) \mapsto \mathcal{G}_{\Phi}^{(\delta)}\left( s_{1}, u; \ s_{0}, s;  \ \phi \right)$ is holomorphic on (\ref{newdomain}) when  $\re s_{1}= \sigma_{1} \ge 15$ and $\re u=\epsilon$.  	Moreover, if  we have $15\le \sigma_{1}\le \eta-\frac{1}{2}$ and (\ref{newdomain}), then  $s-s_{1}-1/2$ lies inside the region of holomorphy of $\widetilde{h}$. 
 	 		
 	 		\begin{comment}
 	 		$\re\left(s-s_{1}-\frac{1}{2}\right) \in (-\eta, \eta)$, i.e.,
 	 		
 	 		all of their polar divisors are avoided. by inspecting  the $\Gamma$'s in the numerator of (\ref{kernelG}).   
 	 		\begin{enumerate}
 	 			\item $ \Gamma(u)\prod\limits_{i=1}^{3} \ \Gamma\left( \frac{s_{1}-\alpha_{i}}{2}\right) $ --- because $\re u=\epsilon>0$ and $\re (s_{1}-\alpha_{i})=\sigma_{1}>0$.

 	 			\item $ \prod\limits_{i=1}^{3} \ \Gamma\left( \frac{s_{0}+\alpha_{i}-u}{2}\right) $ ---  because $\re \left(s_{0}+\alpha_{i}-u\right)= \sigma_{0}-\epsilon>0$. 
 	 			
 	 			\item $\Gamma\left(\frac{2s-s_{0}-u}{2}\right)$ --- because $\re (2s-s_{0}-u)= 2\sigma-\sigma_{0}-\epsilon>0$. 
 	 			
 	 			\item $ \Gamma\left( \frac{u+1-2s+s_{1}-s_{0}}{2}\right)$ --- because $\re \left(u+1-2s+s_{1}-s_{0}\right) 
 	 			\ > \ \epsilon+1-2\sigma+15 -(2\sigma-\epsilon) \ > \ 2\epsilon$. 
 	 		\end{enumerate}
 	 		\end{comment}
 	 		
 	 	\end{rem}
 	 	
 	 	\begin{rem}
 	 		As we shall see in Proposition \ref{proprescon},  the region of holomorphy (\ref{newdomain}) is essentially optimal in terms of $\sigma_{0}$.   
 	 		\end{rem}

 	 	\begin{comment}
 	 	we examine the possible position of the contours for which the integrals  (\ref{mainod}) and (\ref{doubmbtrans}) converge absolutely and simultaneouely ---  in a way less restrictive than before. This gives us a clue on how to shift/bend the contours to attain such position. 
 	 	\end{comment}
 	 	
 	 	\begin{proof}
 	 		The proof is based on a careful application of the Stirling  estimate
 	 		\begin{align}\label{stirremind}
 	 		\left|	\Gamma\left( a+ib \right)\right| \ \asymp_{a} \ \left(1+|b|\right)^{a-\frac{1}{2}} e^{-\frac{\pi}{2}|b|} \hspace{20pt} \left(a \ \neq  \ 0, \  -1, \  -2, \  \ldots, \ \  b \ \in \ \R \right)
 	 		\end{align} 
 	 		to the kernel function $\mathcal{G}_{\Phi}^{(\delta)}\left( s_{1}, u; \ s_{0}, s;  \ \phi \right)$. The following set of conditions will be repeated  throughout the proof:
 	 		\begin{align}\label{uniformity}
 	 		\begin{cases}
 	 	\hspace{40pt}  0 \ <  \ \phi  \ \le \ \pi/2, \\
 	 		\sigma_{0} \ > \ \epsilon,   \hspace{10pt}  \sigma \ < \ 4,  \hspace{10pt}   2\sigma-\sigma_{0}-\epsilon \ > \ 0,  \\
 	 	 \re s_{1}\ = \ \sigma_{1} \ \ge  \ 15,  \hspace{10pt}  \re u \ = \ \epsilon. 
 	 		\end{cases}
 	 		\end{align}
 	 		
 	 		Assuming (\ref{uniformity}), apply (\ref{stirremind}) to the kernel function (\ref{kernelG}). It follows that
 	 		\begin{align}\label{giantstirl}
 	 		\left| \mathcal{G}_{\Phi}^{(\delta)}\left( s_{1}, u; \ s_{0}, s;  \ \phi \right)  \right| \ \asymp  \  &\left(1+|v|\right)^{\epsilon-\frac{1}{2}} e^{- \left( \frac{\pi}{2}-\phi\right)|v|} \cdot \prod\limits_{k=1}^{3} \ \left( 1+ |t_{1}-\gamma_{k}|\right)^{\frac{\sigma_{1}-1}{2}} e^{-\frac{\pi}{4} |t_{1}-\gamma_{k}|} \nonumber\\
 	 		&\cdot  \prod\limits_{k=1}^{3} \ \left(1+|t_{0}-v+ \gamma_{k}|\right)^{\frac{\sigma_{0}-\epsilon-1}{2}} e^{-\frac{\pi}{4} \left| t_{0}-v+\gamma_{k}\right|}  \cdot \left(1+ |2t-t_{0}-v|\right)^{\frac{2\sigma-1-\sigma_{0}-\epsilon}{2}} e^{-\frac{\pi}{4} |2t-t_{0}-v|}  \nonumber\\
 	 		&\cdot  \left( 1+|v-2t+t_{1}-t_{0}|\right)^{\frac{\epsilon-2\sigma+\sigma_{1}-\sigma_{0}}{2}} e^{-\frac{\pi}{4} |v-2t+t_{1}-t_{0}|}   \nonumber\\
 	 		& \cdot \left( 1+ |t_{1}-2t_{0}|\right)^{-\left(\frac{\sigma_{1}}{2}-\sigma_{0}\right)} e^{\frac{\pi}{4}|t_{1}-2t_{0}|} \ \cdot \left( 1+ \left| t_{0}+t_{1}-v\right|\right)^{-\frac{\sigma_{0}+\sigma_{1}-\epsilon-1}{2}} e^{\frac{\pi}{4} \left| t_{0}+t_{1}-v\right|},   
 	 		\end{align}
 	 		where the implicit constant depends at most on $\sigma_{1}$. Note that the domain (\ref{newdomain}) for $(\sigma, \sigma_{0})$ is bounded and thus the estimate is uniform in $\sigma, \sigma_{0}, \epsilon$. This will be assumed for all estimates in the rest of this section.  
 	 		
 	 		Let   $	\mathcal{P}_{s}^{\Phi}\left( t_{0}, t_{1}, v\right)$
 	 		be the `polynomial part' of  (\ref{giantstirl}) and 
 	 		\begin{align*}
 	 		\mathcal{E}_{s}^{\Phi}\left(t_{0}, t_{1}, v\right) \ := \  \sum_{k=1}^{3}  \ \left\{  \left| t_{1}-\gamma_{k}\right|+ \left| t_{0}-v+ \gamma_{k}\right| \right\} + |2t-t_{0}-v| + |v-2t+t_{1}-t_{0}| - |t_{1}-2t_{0}|  - |t_{0}+t_{1}-v|. 
 	 		\end{align*}

 	 		We first examine the exponential phase $\mathcal{E}_{s}^{\Phi}\left(t_{0}, t_{1}, v\right)$ of  (\ref{giantstirl})  as it determines the effective support of $\big(\mathcal{F}_{\Phi}^{(\delta)}H\big)(s_{0}, s; \phi)$. By the triangle inequality and the fact  $\gamma_{1}+\gamma_{2}+\gamma_{3}=0$, we have
 	 			\begin{align}\label{simplstir}
 	 			\left|\mathcal{G}_{\Phi}^{(\delta)}\left( s_{1}, u; \ s_{0}, s;  \ \phi \right)   \right|  \ \ll_{\sigma_{1}} \  e^{\pi T}\cdot 	\mathcal{P}_{s}^{\Phi}\left( t_{0}, t_{1}, v\right) \cdot \exp\left(-\frac{\pi}{4} \mathcal{E}(t_{0}, t_{1}, v)\right) \cdot e^{- \left( \frac{\pi}{2}-\phi\right)|v|}
 	 			\end{align}
 	 			with
 	 			\begin{align}
 	 			\mathcal{E}(t_{0}, t_{1}, v) \ := \ 3|t_{1}|+ 3|t_{0}-v| - |t_{1}-2t_{0}|
 	 			+ |v+t_{1}-t_{0}|+ |t_{0}+v| - |t_{0}+t_{1}-v|,
 	 			\end{align}
 	 			whenever we have (\ref{uniformity}) and $|t|\le T$,

 	 	\begin{comment}
 	 		\begin{align*}
 	 		\mathcal{E}_{s}^{\Phi}\left(t_{0}, t_{1}, v\right)	\  \ge  \ &-4|t| + \mathcal{E}(t_{0}, t_{1}, v), 
 	 		\end{align*}
 	 		\end{comment}

 	 		\begin{comment}
 	 		\footnote{ If one also aims to take care of the uniformity in the $t$-aspect for $|t| \gg 1$, the inequality right below is clearly not suffcient for the purpose. }
 	 		\end{comment}
 	 		
 	 		\begin{cla}\label{expset}
 	 			For any $t_{0}, t_{1}, v\in \R$, we have $\mathcal{E}(t_{0}, t_{1}, v) \ \ge \ 0$. Equality holds if and only if
 	 			\begin{align}\label{expzero}
 	 			t_{1} \ = \ 0  \hspace{15pt} \text{ and }  \hspace{15pt} t_{0}-v \ = \ 0.
 	 			\end{align}
 	 		\end{cla}

 	 		\begin{proof}
 	 			Adding up the inequalities  $|t_{1}|+ |t_{0}-v| \ge |t_{0}+t_{1}-v|$ and $ |v+t_{1}-t_{0}|+ |t_{0}+v| \ge |t_{1}-2t_{0}|$, we have
 	 			\begin{align}\label{expdecineq}
 	 			\mathcal{E}(t_{0}, t_{1}, v) \ \ge \ 2 \left(|t_{1}|+|t_{0}-v| \right)\ \ge \ 0. 
 	 			\end{align}
 	 			The equality case is apparent. 
 	 		\end{proof}
 	 		
 	 		% the  function  $	\left(\mathcal{F}_{\Phi}^{(\delta)}H\right)\left(s_{0}, \ s; \ \phi\right)$

 	 		\begin{cla}\label{trivdeclaim}
 	 				When (\ref{uniformity}) and $|t|\le T$ hold,  the integral
 	 				\begin{align}\label{trivexpdint}
 	 				\iint_{  \substack{(\re s_{1}, \re u)=(\sigma_{1}, \epsilon),  \\ (t_{1},v): \text{(\ref{expcutoff}) holds}} } \  \widetilde{h}\left(s-s_{1}-\frac{1}{2}\right)    \cdot \mathcal{G}_{\Phi}^{(\delta)}\left( s_{1}, u; \ s_{0}, s;  \ \phi \right)  \ \frac{du}{2\pi i}  \ \frac{ds_{1}}{2\pi i}
 	 				\end{align}
 	 				has  exponential decay as $|t_{0}|\to \infty$, where   
 	 				\begin{align}\label{expcutoff}
 	 				|t_{1}| \ > \ \log^2 (3+|t_{0}|) \hspace{15pt} \text{ or } \hspace{15pt} |v-t_{0}| \ > \ \log^2 (3+|t_{0}|). 
 	 				\end{align}
 	 			
 	 			%	For any $\phi \in (0, \pi/2]$, we have uniform convergence in  (\ref{mainintrans}) on the bounded domain
 	 		
 	 		\end{cla}
 	 		
 	 		\begin{proof}
 	 			In case of (\ref{expcutoff}), we have
 	 			\begin{align}\label{explowbdd}
 	 			\mathcal{E}(t_{0}, t_{1}, v)  \ > \ \log^2 (3+|t_{0}|) +|t_{1}|+|t_{0}-v|
 	 			\end{align}
 	 			 from (\ref{expdecineq}). The polynomial part $\mathcal{P}^{\Phi}_{s}(t_{0}, t_{1}, v)$ can be crudely bounded by 
 	 			\begin{align}\label{polybd}
 	 			\mathcal{P}^{\Phi}_{s}(t_{0}, t_{1}, v) \ &\ll_{\Phi, \sigma_{1}, T} \ [\left(1+|t_{1}|\right)\left(1+|v-t_{0}|\right)\left(1+|t_{0}|\right)]^{A(\sigma_{1})}, 
 	 			\end{align}
 	 			where  $A(\sigma_{1}) >0$ is some constant.

 	 			\begin{comment}
 	 			Indeed, by Lemma \ref{expset}, we must have $\mathcal{E}(t_{0}, t_{1},v)>0$. It then follows from elementary linear algebra that 
 	 			\begin{align}\label{explobd}
 	 			\mathcal{E}(t_{0}, t_{1},v) \ &\ge \  2c\left(|t_{0}|+|t_{1}|+|v|\right) 
 	 			\end{align}
 	 			for some  small absolute constant $c>0$, see  \cite{HMB92} for example.   Now, (\ref{expcutoff}) implies
 	 			\begin{align}
 	 			\mathcal{E}(t_{0}, t_{1},v)  	\ &\ge \  c(|t_{1}|+|v|)+ c\left( |t_{1}|+ |t_{0}-v|\right)
 	 			\ > \ c \cdot \left[ |t_{1}|+|v|+\log^2 (3+|t_{0}|) \right]. 
 	 			\end{align}
 	 			\end{comment}

 	 			Putting (\ref{explowbdd}),  (\ref{polybd}), and the bound $e^{- \left( \frac{\pi}{2}-\phi\right)|v|}\le 1$ ($\phi \in(0, \pi/2]$)  into (\ref{simplstir}), we obtain
 	 			\begin{align}\label{gamkerdec}
 	 			\hspace{-10pt} \left|\mathcal{G}_{\Phi}^{(\delta)}\left( s_{1}, u; \ s_{0}, s;  \ \phi \right)   \right|  \ \ll_{\Phi, \sigma_{1}, T} \ &\left(1+|t_{0}|\right)^{A(\sigma_{1})}e^{-\frac{\pi}{4} \log^2 (3+|t_{0}|)} \cdot  [\left(1+|t_{1}|\right)\left(1+|v-t_{0}|\right)]^{A(\sigma_{1})}  e^{-\frac{\pi}{4}[|t_{1}|+|t_{0}-v|]}
 	 			\end{align}
 	 			whenever (\ref{expcutoff}), (\ref{uniformity}),  and $|t|\le T$ hold. The boundedness  of $\widetilde{h}$ on vertical strips implies that 	(\ref{trivexpdint})	is 
 	 			\begin{align}\label{nonzeroexp}
 	 		 \ \ll_{\sigma_{1}, \Phi, T} \  
 	 		  (1+|t_{0}|)^{A(\sigma_{1})} e^{-\frac{\pi}{4}  \log^2 (3+|t_{0}|)}. 
 	 			\end{align}
 	 			This proves  Claim \ref{trivdeclaim}. 
 	 				\end{proof}

 	 			Now, let $\phi\in (0, \pi/2]$ and consider $\big(\mathcal{F}_{\Phi}^{(\delta)}H\big)\left(s_{0}, \ s; \ \phi\right)$ as a function on the bounded domain
 	 			\begin{align}\label{bddomuni}
 	 			(\sigma_{0},\sigma) \ \in  \  (\ref{newdomain}), \hspace{20pt}  |t|, \ |t_{0}|  \ \le  \ T. 
 	 			\end{align}
 	 			When $|t_{1}|> \log^2 (3+T)$ or $|v|> T+\log^2 (3+T)$, observe that (\ref{expcutoff}) is satisfied and from (\ref{gamkerdec}), 
 	 			\begin{align}\label{holoestim}
 	 			\left|\mathcal{G}_{\Phi}^{(\delta)}\left( s_{1}, u; \ s_{0}, s;  \ \phi \right)   \right|  \ \ll_{\Phi, T} \    [\left(1+|t_{1}|\right)\left(1+|v|\right)]^{A(15)} \cdot   e^{-\frac{\pi}{4}[|t_{1}|+|v|]}.
 	 			\end{align}
 	 			The last function is clearly jointly integrable with respect to $t_{1}$, $v$, and by Remark \ref{nopolerem},   $\big(\mathcal{F}_{\Phi}^{(\delta)}H\big)\left(s_{0}, \ s; \ \phi\right)$ is a holomorphic function on (\ref{bddomuni}). Since the choice of $T$ is arbitrary, we arrive at the first conclusion of Proposition \ref{anconpr}.

 	 		 %Thus, converges uniformly on (\ref{bddomuni}).  

 	 		%%%%%%%%%%%%%%%%%%%%%%%%%%%%%%%%%%%%%%%%%%%%%%%%%%%%%%%%%%%%%%%%%%%%%%%%%%%%%%%%%%%%%%%%%%%%%%%%%%%%%%%%%%%%%%%%%%%%%%%%%%%%%%%%%%%%%%%%%%%%%%%%%%%%%%%%%%%%%%%%%%%%%%%%%%%%%%%%%%%%%%%%%%%%%%%%%%%%%%%%%%%%%%%%%%%%%%%%%%%%%%%%%%%%%%%%%%%%%%%%%%%%%%%%%%%%%%%%%%%%%%%%%%%%%%%%%%%%%%%%%%%
 	 		In the remaining part of this section, we prove the second assertion of Proposition \ref{anconpr}.  We estimate the contribution from
 	 		\begin{align}\label{almostexpzero}
 	 		|t_{1}| \ \le \ \log^2 (3+|t_{0}|) \hspace{15pt} \text{ and } \hspace{15pt} |v-t_{0}| \ \le \ \log^2 (3+|t_{0}|). 
 	 		\end{align}
 	 			(The complementary part has been treated in Claim \ref{trivdeclaim}.)

 	 	It suffices to restrict ourselves to the effective support (\ref{expzero}). The polynomial part can be essentially computed by substituting $t_{1}:=0$ and $v:=t_{0}$.  More precisely, when  (\ref{almostexpzero}) and  $|t_{0}| \gg_{T} 1$ hold,  there are only two possible scenarios for the factors $1+|(\cdots)|$ in (\ref{giantstirl}): either $ 1+ |(\cdots)| \ \asymp \ |t_{0}|$, or  \ $	\log^{-C} (3+|t_{0}|) \ \ll \   1+|(\cdots)| \ \ll \ \log^{C} (3+|t_{0}|) $ for some absolute constant $C>0$. 
 	 		
 	 		In case of (\ref{almostexpzero}),  apply the bounds 	$e^{-\frac{\pi}{4} \mathcal{E}(t_{0}, t_{1},v)}  \le    1$ and  $e^{-\left( \frac{\pi}{2}-\phi\right)|v|}    \le    e^{-\frac{1}{2}\left( \frac{\pi}{2}-\phi\right) |t_{0}|}$ for $|t_{0}|\gg   1$ to  (\ref{simplstir}). As a result, if we also have (\ref{uniformity}),  $|t| <T$, and  $|t_{0}|  >8T$, then 
 	 		\begin{align}
 	 		\left| \mathcal{G}_{\Phi}^{(\delta)}\left( s_{1}, u; \ s_{0}, s;  \ \phi \right)  \right| \ &\ll_{\sigma_{1}, \Phi, T} \  |t_{0}|^{7-\frac{\sigma_{1}}{2}} e^{- \frac{1}{2}\left( \frac{\pi}{2}-\phi\right)|t_{0}|}  \log^{B(\sigma_{1})} |t_{0}|    \label{exzerest}
 	 		\end{align}
 	 		and 
 	 		\begin{align}\label{zeroexp}
 	 \hspace{-80pt} 	\iint_{\substack{ (\re s_{1}, \re u)= (\sigma_{1}, \epsilon), \\(t_{1},v): \ \text{(\ref{almostexpzero}) holds}} } \  \widetilde{h}\left(s-s_{1}-\frac{1}{2}\right)    &\cdot \mathcal{G}_{\Phi}^{(\delta)}\left( s_{1}, u; \ s_{0}, s;  \ \phi \right)  \  \frac{du}{2\pi i}  \ \frac{ds_{1}}{2\pi i} \nonumber\\
 	 		\ & \hspace{40pt}  \ll_{  \sigma_{1}, \Phi, T} \  |t_{0}|^{7-\frac{\sigma_{1}}{2}} e^{-\frac{1}{2}\left( \frac{\pi}{2}-\phi\right)|t_{0}|}  \log^{4+B(\sigma_{1})} |t_{0}|,   
 	 		\end{align} 
 	 		where $B(\sigma_{1})>0$ is some constant. If  $\phi < \pi/2$, then there is exponential decay in (\ref{zeroexp}) as  $|t_{0}|\to \infty$. Therefore, the second conclusion of the proposition follows from  (\ref{zeroexp}) and  (\ref{nonzeroexp}) (putting $\sigma_{1}=15$).  
 	 			\end{proof}

 	 		\begin{comment}
 	 		\begin{align}
 	 		&\iint_{(t_{1},v): \ \text{(\ref{almostexpzero}) holds} } \  \widetilde{h}\left(s-s_{1}-\frac{1}{2}\right)    \cdot \mathcal{G}_{\Phi}^{(\delta)}\left( s_{1}, u; \ s_{0}, s;  \ \phi \right)  \ dv  \ dt_{1} \nonumber\\
 	 		\ \ll_{\sigma_{1}, \Phi, T} \ &  e^{-\left( \frac{\pi}{2}-\phi\right)|t_{0}|} \cdot  |t_{0}|^{7- \frac{\sigma_{1}}{2}} \cdot \log^{4+B} |t_{0}|.   
 	 		\end{align} 
 	 		for some absolute constant $B>0$. 
 	 		\end{comment}

 %%%%%%%%%%%%%%%%%%%%%%%%%%%%%%%%%%%%%%%%%%%%%%%% 	 %%%%%%%%%%%%%%%%%%%%%%%%%%%%%%%%%%%%%%%%%%%%%%%% 	 %%%%%%%%%%%%%%%%%%%%%%%%%%%%%%%%%%%%%%%%%%%%%%%% 	

 	\section{Analytic Continuation of the Off-diagonal (Proof of Theorem \ref{maingl3gl2})}\label{2stepana}
 	
 	Recall that
 	\begin{align}
 		OD_{\Phi}(s; \phi) \ = \   &\frac{1}{4}  \ \int_{(1+\theta+2\epsilon)} \  \zeta\left(2s-s_{0}\right) L\left(s_{0},  \Phi\right) \cdot \sum_{\delta=\pm} \left(\mathcal{F}_{\Phi}^{(\delta)}H\right)\left(s_{0}, \ s; \  \phi\right) \ \frac{ds_{0}}{2\pi i}
 	\end{align}
 for 	$1+ \frac{\theta}{2} + \epsilon  <  \sigma  <  4$ and $\phi\in (0, \pi/2)$, see Proposition  \ref{structure}.

 	\subsection{Step 1: }\label{firstcont} We first obtain a holomorphic continuation of $OD_{\Phi}(s; \phi)$ up to $\re s > \frac{1}{2}+\epsilon$ by shifting the $s_{0}$-integral to the left. 
 	
 	Fix any $\phi\in (0,\pi/2)$ and  $T\ge 1000$.  We first restrict ourselves to 
 	\begin{align}\label{inidom}
 	1+ \frac{\theta}{2}+2\epsilon \ < \  \sigma \ < \  4, \hspace{15pt}  |t|\ < \  T.
 	\end{align}
 	Clearly, the pole $s_{0}=2s-1$ of $\zeta(2s-s_{0})$ is on the right of the contour $\re s_{0}=1+\theta+ 2\epsilon$ of  the integral (\ref{mainod}).  
 	
 	Let  $T_{0}\gg 1$. The  rectangle with vertices $2\epsilon\pm iT_{0}$ and $(1+\theta+2\epsilon)\pm iT_{0}$ in the $s_{0}$-plane lies inside the region of holomorphy (\ref{newdomain}) of $(\mathcal{F}^{(\delta)}_{\Phi} H)(s_{0}, s; \phi)$. 
 	The contribution from the horizontal segments $[2\epsilon\pm iT_{0}, (1+\theta+2\epsilon)\pm iT_{0}]$ tends to $0$ as $T_{0}\to \infty$ by the exponential decay of 	$(\mathcal{F}_{\Phi}^{(\delta)}H)(s_{0}, \ s; \ \phi)$ (see Proposition \ref{anconpr}) which surely counteracts the polynomial growth from $L(s_{0}, \Phi)$ and $\zeta(2s-s_{0})$. 	As a result, we may   shift the line of integration to $\re s_{0}= 2\epsilon$ and no pole is crossed. Hence, 
 	\begin{align}\label{shifleft}
 	OD_{\Phi}(s; \phi) \ &= \   \frac{1}{4}  \ \int_{(2\epsilon)} \zeta\left(2s-s_{0}\right) L\left(s_{0},  \Phi\right) \cdot \sum_{\delta=\pm} \left(\mathcal{F}_{\Phi}^{(\delta)}H\right)\left(s_{0}, \ s; \  \phi\right) \ \frac{ds_{0}}{2\pi i}
 	\end{align}
 	on (\ref{inidom}). The right side of (\ref{shifleft})  is holomorphic on
 	\begin{align}\label{uponshif}
 	\frac{1}{2}+\epsilon \ < \ \sigma \ < \  4, \hspace{15pt} |t| \ < \  T
 	\end{align}
 	and serves as an analytic continuation of $OD_{\Phi}(s; \phi)$ to (\ref{uponshif}) by using Proposition \ref{anconpr}. Note that $\sigma > \frac{1}{2}+\epsilon $ implies the holomorphy of $\zeta(2s-s_{0})$.

 	\begin{comment}
 	\begin{enumerate}
 	\item  there is an absolute constant $A>0$ such that
 	\begin{align}
 	L(s_{0}, \Phi) \ \ll_{\Phi} \  (1+T_{0})^{A}
 	\end{align}
 	(say by the convexity bound); 
 	
 	\item since 
 	\begin{align*}
 	2\sigma-\sigma_{0} \ > \ 2\left( 1+ \frac{\theta}{2}+2\epsilon\right)-(1+ \theta+2\epsilon) \ = \  1+2\epsilon,
 	\end{align*}
 	we have 
 	\begin{align*}
 	\zeta(2s-s_{0}) \ \ll \ 1; 
 	\end{align*}
 	
 	\item by the estimate (\ref{estimatrans}), if $M\ge A+9$, then
 	\begin{align}
 	\left(\mathcal{F}_{\Phi}^{(\delta)}h\right)\left(s_{0},  s\right) \ \ll_{\Phi, T} \   (1+ T_{0})^{-A-1}; 
 	\end{align}
 	
 	\item altogether, 
 	\begin{align}
 	\zeta\left(2s-s_{0}\right) L\left(s_{0},  \Phi\right) \cdot \sum_{\delta=\pm} \left(\mathcal{F}_{\Phi}^{(\delta)}h\right)\left(s_{0}, \ s; \  \phi\right) \ \ll_{\Phi, T}  \  (1+T_{0})^{-1} \ \to \ 0
 	\end{align}
 	as $T_{0}\to \infty$. 
 	\end{enumerate}
 	\end{comment}

 	%%%%%%%%%%%%%%%%%%%%%%%%%%%%%%%%%%%%%%%%%%%%%%%%%%%%%%%%%%%%%%%%%%%%%%%%%%%%%%%%%%%%%%%%%%%%%%%%%%%%%%%%%%%%%%%%%%%%%%%%%%%%%%%%%

 	\subsection{Step 2: Crossing the Polar Line (Shifting the $s_{0}$-integral again)}\label{secondcont}
 	Consider  a subdomain of (\ref{uponshif}): 
 	\begin{align}\label{restrdis}
 	\frac{1}{2}+ \epsilon \ < \ \sigma \ < \ \frac{3}{4}, \hspace{15pt} |t| \ < \ T. 
 	\end{align}
 Different from Step $1$, the pole $s_{0}=2s-1$   is now inside the rectangle with vertices $2\epsilon\pm iT_{0}$ and $\frac{1}{2}\pm iT_{0}$ provided $T_{0} > 4T$. Such a rectangle lies in the region of holomorphy (\ref{newdomain}) of $(\mathcal{F}^{(\delta)}_{\Phi} H)(s_{0}, s; \phi)$. When $\phi< \pi/2$,  the exponential decay of $(\mathcal{F}^{(\delta)}_{\Phi} H)(s_{0}, s; \phi)$ once again allows us to  shift the line of integration from $\re s_{0}=2\epsilon$ to $\re s_{0}=1/2$, crossing the pole of  $\zeta(2s-s_{0})$ which has residue $-1$. In other words, 
 	\begin{align}\label{desiredexpr}
 	OD_{\Phi}(s; \phi) \ = \  & \frac{1}{4} \  L(2s-1, \Phi) \  \sum_{\delta=\pm}  \left(\mathcal{F}_{\Phi}^{(\delta)}H\right)\left(2s-1, \ s; \  \phi\right) \nonumber\\
 	& \hspace{40pt} + \frac{1}{4} \   \int_{(1/2)} \zeta\left(2s-s_{0}\right) L\left(s_{0},  \Phi\right) \cdot \sum_{\delta=\pm} \left(\mathcal{F}_{\Phi}^{(\delta)}H\right)\left(s_{0}, \ s; \  \phi\right) \ \frac{ds_{0}}{2\pi i}. 
 	\end{align}
 	
 	On the line $\re s_{0}=1/2$, observe that $ s\mapsto  \big(\mathcal{F}_{\Phi}^{(\delta)}H\big)\left(s_{0}, \ s; \  \phi\right)$ is holomorphic on $\sigma> \frac{1}{4}+\frac{\epsilon}{2}$ by  (\ref{newdomain}) ; whereas $s \mapsto \zeta(2s-s_{0})$  is holomorphic on   $\sigma< 3/4$ as $2\sigma-\sigma_{0}<1$.  As a result, the function $s\mapsto \int_{(1/2)} (\cdots) \ \frac{ds_{0}}{2\pi i}$ in (\ref{desiredexpr}) is holomorphic on the vertical strip
 	\begin{align}
 	\frac{1}{4}+ \frac{\epsilon}{2} \ < \ \sigma \ < \ \frac{3}{4},
 	\end{align}
 	which is sufficient for our purpose. 
 	
 	Proposition \ref{anconpr} only asserts that the function $s\mapsto\big(\mathcal{F}^{(\delta)}_{\Phi} H\big)\left(2s-1, s; \phi\right)$ is holomorphic on $\frac{1}{2}+\epsilon < \sigma< 4$.  However, it actually admits a continuation to the domain $\epsilon< \sigma < 4$ as we will see in Proposition \ref{proprescon}.

 	%%%%%%%%%%%%%%%%%%%%%%%%%%%%%%%%%%%%%%%%%%%%%%%%%%%%%%%%%%%%%%%%%%%%%%%%%%%%%%%%%%%%%%%%%%%%%%%%%%%%%%%%%%%%%%%%%%%%%%%%%%%%%%%%%%%%%%%%%%%%%%%%%%%%%%%%%%%%%%%%%%%%%%%%%%%%%%%%%%%%%%%%%%%%%%%%%%%%%%%%%%%%%%%%%%%%%%%%

 	\subsection{Step 3:  Putting Back $\phi \to \pi/2$ --- Shifting the $s_{1}$-integral and Refining Step 1-2}
 	
 	By using estimate (\ref{gamkerdec}) and Dominated Convergence, 
 	\begin{align}
 	\lim_{\phi \to \pi/2} \  \left(\mathcal{F}_{\Phi}^{(\delta)}H\right)\left(2s-1, \ s; \  \phi\right) \ = \   \left(\mathcal{F}_{\Phi}^{(\delta)}H\right)\left(2s-1, \ s; \  \pi/2\right)
 	\end{align}
 	for $\frac{1}{2}+\epsilon < \sigma< 4$ and $|t|<T$. However, for the  continuous part of (\ref{desiredexpr}), we need a follow-up of Proposition \ref{anconpr} in order to pass to the limit $\phi \to \pi/2$. Essentially,  thanks to the structure of the $\Gamma$'s in Proposition \ref{vtmellin} and the analytic properties of $\widetilde{h}$,   it is possible to shift the line of integration of the $s_{1}$-integral to gain sufficient polynomial decay.

 	\begin{prop}\label{polyestimatrans}
 		Let $H\in \mathcal{C}_{\eta}$. There exists a constant  $B=B_{\eta}$ such that  whenever $(\sigma_{0}, \sigma) \in (\ref{newdomain})$, $|t| <T$, and $|t_{0}| \gg_{T} 1$,  we have the estimate
 		\begin{align}\label{refbdd}
 		\left|	\left(\mathcal{F}_{\Phi}^{(\delta)}H\right)\left(s_{0}, \ s; \ \pi/2\right) \right| \ \ll \  |t_{0}|^{8-\frac{\eta}{2}}   \log^{B} |t_{0}|,
 		\end{align}
 		where the implicit constant depends only on $\eta$, $T$, $\Phi$. 
 		
 	\end{prop}
 	
 	\begin{proof}
 		On domain (\ref{newdomain}), observe that the vertical strip $\re s_{1}\in [15, \eta-\frac{1}{2}]$ contains no pole of the function $s_{1} \mapsto \mathcal{G}_{\Phi}^{(\delta)}\left( s_{1}, u; \ s_{0}, s;  \ \phi \right)$,  and it lies within the region of holomorphy of $\widetilde{h}$  (see Remark \ref{nopolerem}).  The estimate (\ref{gamkerdec}) allows us to shift the line of integration from $\re s_{1}=15$  to $\re s_{1}=\eta-\frac{1}{2}$  in  (\ref{doubmbtrans}). Notice that the estimates done in Proposition \ref{anconpr} works for $\phi=\pi/2$ too. In particular, from (\ref{zeroexp}) and (\ref{nonzeroexp}),    the bound (\ref{refbdd})  follows by taking $\sigma_{1}:=\eta-\frac{1}{2}$ therein (upon the contour shift). This completes the proof. 
 			\end{proof}

 		%When $\phi=\pi/2$,  the exponential decay vanishes completely (see (\ref{expzero})). 
 		
 		% because of $H\in \mathcal{C}_{\eta}$ and $\re (s-s_{1}-1/2) \in (-\eta, \eta)$ (whereas $\epsilon< \sigma< 4$).

 	%Recall that $OD_{\Phi}(s)$ is related to  $OD_{\Phi}(s; \phi)$ via  

 	Suppose  $(3+\theta)/2 \ < \ \sigma \ < \ 4$. By Proposition \ref{pullout},  equation (\ref{mainod}) and equation (\ref{shifleft}), we have
 	\begin{align}
 	OD_{\Phi}(s)  \ &= \   \lim_{\phi \to \pi/2} \ OD_{\Phi}(s; \phi)  \nonumber\\
 	\ &= \ \lim_{\phi \to \pi/2}  \ \frac{1}{4} \  \int_{(1+\theta+2\epsilon)} \zeta\left(2s-s_{0}\right) L\left(s_{0},  \Phi\right) \cdot \sum_{\delta=\pm} \left(\mathcal{F}_{\Phi}^{(\delta)}H\right)\left(s_{0}, \ s; \  \phi\right) \ \frac{ds_{0}}{2\pi i} \nonumber\\
 	\ &= \  \lim_{\phi \to \pi/2} \ \frac{1}{4} \   \int_{(2\epsilon)} \zeta\left(2s-s_{0}\right) L\left(s_{0},  \Phi\right) \cdot \sum_{\delta=\pm} \left(\mathcal{F}_{\Phi}^{(\delta)}H\right)\left(s_{0}, \ s; \  \phi\right) \ \frac{ds_{0}}{2\pi i}.  \label{beforelim}
 	\end{align}
 	Proposition \ref{polyestimatrans} ensures enough polynomial decay and hence the absolute convergence of (\ref{prelimshif}) at $\phi=\pi/2$: 
 	\begin{align}
 	OD_{\Phi}(s)  \ = \   \frac{1}{4}  \ \int_{(2\epsilon)} \zeta\left(2s-s_{0}\right) L\left(s_{0},  \Phi\right) \cdot \sum_{\delta=\pm} \left(\mathcal{F}_{\Phi}^{(\delta)}H\right)\left(s_{0}, \ s; \  \pi/2\right) \ \frac{ds_{0}}{2\pi i}. \label{prelimshif}
 	\end{align}
 	Now,  (\ref{prelimshif}) serves as an analytic continuation of  $OD_{\Phi}(s)$  to the domain $1/2+ \epsilon < \sigma<4$. %by dominated convergence and 
 	
 	On the smaller domain $1/2+ \epsilon< \sigma< 3/4$,  the  expressions (\ref{beforelim}) and (\ref{desiredexpr}) are equal. Then
 	\begin{align}
 	OD_{\Phi}(s) \ = \  (\ref{beforelim})
 	\ = \  &  \frac{1}{4} \    L(2s-1, \Phi) \  \sum_{\delta=\pm}  \left(\mathcal{F}_{\Phi}^{(\delta)}H\right)\left(2s-1, \ s; \  \pi/2\right) \nonumber\\
 	& \hspace{30pt} + \frac{1}{4}  \  \int_{(1/2)} \  \zeta\left(2s-s_{0}\right) L\left(s_{0},  \Phi\right) \cdot \sum_{\delta=\pm} \left(\mathcal{F}_{\Phi}^{(\delta)}H\right)\left(s_{0}, \ s; \  \pi/2\right) \ \frac{ds_{0}}{2\pi i} \label{finalmome}
 	\end{align}
 	by Dominated Convergence and Proposition \ref{anconpr}. 	The last integral is holomorphic on $ \frac{1}{4}+ \frac{\epsilon}{2}  <  \sigma  <  \frac{3}{4}$.
 	
 	In the following, we write $	\left(\mathcal{F}_{\Phi} H\right)\left(s_{0},  s \right)  := \left(\mathcal{F}_{\Phi}^{+} H\right)\left(s_{0}, s; \  \pi/2 \right) + \left(\mathcal{F}_{\Phi}^{-} H\right)\left(s_{0}, s; \  \pi/2 \right)$.  The duplication and the reflection formula of  $\Gamma$-functions in the form
 	\begin{align*}
 	2^{-u}\Gamma(u)  \ = \   \frac{1}{2\sqrt{\pi}} \cdot \Gamma\left(\frac{u}{2}\right) \Gamma\left(\frac{u+1}{2}\right) \hspace{15pt} \text{ and } \hspace{15pt} \Gamma\left(\frac{1+u}{2}\right)\Gamma\left(\frac{1-u}{2}\right)\ = \ \pi \sec \frac{\pi u}{2},
 	\end{align*}
 	leading to 
 	\begin{align}\label{streamlinebarnes}
 	\left(\mathcal{F}_{\Phi} H\right)\left(s_{0},  s \right) 
 	\ = \  &  \sqrt{\pi} \  \int_{(\eta-1/2)}  \widetilde{h}\left(s-s_{1}-\frac{1}{2}\right) \pi^{-s_{1}}   \frac{\prod\limits_{i=1}^{3} \ \Gamma\left( \frac{s_{1}-\alpha_{i}}{2}\right)  }{ \Gamma\left(\frac{1+s_{1}}{2}-s_{0}\right)} \nonumber\\
 	& \hspace{50pt} \cdot   \int_{(\epsilon)}  \   \frac{ \Gamma\left( \frac{u}{2}\right) \Gamma\left( \frac{s_{1}-(s_{0}-u)}{2}+ \frac{1}{2}-s\right)  \cdot   \prod\limits_{i=1}^{3} \ \Gamma\left( \frac{(s_{0}-u)+\alpha_{i}}{2}\right)  \Gamma\left(s-\frac{s_{0}+u}{2}\right) }{\Gamma\left( \frac{1-u}{2}\right) \Gamma\left(\frac{(s_{0}-u)+s_{1}}{2}\right)}  \  \frac{du}{2\pi i}   \ \frac{ds_{1}}{2\pi i }. 
 	\end{align}
 	In Section \ref{expevatra}, we shall work with this expression further.

 	\begin{comment}
 	When $s_{0}=2s-1$ and 	 $\frac{1}{2}+\epsilon < \sigma< 4$, the expression further simplifies to
 	\begin{align}\label{resiman}
 	\left(\mathcal{F}_{\Phi} H\right)(2s-1, s)  	
 	\ =\  &   \sqrt{\pi} \int_{(\eta-1/2)}  \widetilde{h}\left(s-s_{1}-\frac{1}{2}\right) \pi^{-s_{1}}   \frac{\prod\limits_{i=1}^{3} \ \Gamma\left( \frac{s_{1}-\alpha_{i}}{2}\right)  }{ \Gamma\left(\frac{1+s_{1}}{2}+1-2s\right)} \nonumber\\
 	& \hspace{50pt} \cdot   \int_{(\epsilon)}  \   \frac{  \Gamma\left( \frac{u}{2}\right) \Gamma\left( \frac{u+s_{1}}{2}+ 1-2s\right)  \cdot   \prod\limits_{i=1}^{3} \ \Gamma\left( s-\frac{1}{2}+\frac{\alpha_{i}}{2}-\frac{u}{2}\right) }{\Gamma\left(s-\frac{1}{2}+\frac{s_{1}-u}{2}\right)}   \  \frac{du}{2\pi i}   \ \frac{ds_{1}}{2\pi i }.
 	\end{align}
 	\end{comment}
 	
 %	In fact, the $s_{1}$-integral can also be expressed in the form of a Barnes integral. However, there are minor technicalities in doing so. This is actually not needed for the analytic continuation of $OD_{\Phi}(s)$. This is postponed to Section \ref{expevatra} so as not to interrupt the flow of our argument. 

 	%%%%%%%%%%%%%%%%%%%%%%%%%%%%%%%%%%%%%%%%%%%%%%%%%%%%%%%%%%%%%%%%%%%%%%%%%%%%%%%%%%%%%%%%%%%%%%%%%%%%%%%%%%%%%%%%%%%%%%%%%%%%%%%%%%%%%%%%%%%%%%%%%%%%%%%%%%%%%%%%%%%%%%%%%%%%%%%%%%%%%%%%%%%%
 	
 	\subsection{Step 4: Continuation of the Residual Term --- Shifting the $u$-integral}\label{shiftuint}
 	
 	\begin{prop}\label{proprescon}
 		Let $H\in \mathcal{C}_{\eta}$.  The function $s\mapsto \left(\mathcal{F}_{\Phi} H\right)(2s-1, s) $ can be holomorphically continued to the vertical strip $\epsilon< \sigma<4$  except at the three simple poles:  $s=(1-\alpha_{i})/2$ \ ($i=1,2,3$), where $ (\alpha_{1}, \alpha_{2}, \alpha_{3})$ are the Langlands parameters of the  Maass cusp form $\Phi$. 
 	\end{prop}
 	
 	\begin{proof}
 		We will prove a stronger result in Proposition \ref{secMTcom}. However, a simpler argument suffices for the time being. Suppose  $\frac{1}{2}+\epsilon < \sigma< 4$ and $s_{0}=2s-1$.   In	(\ref{streamlinebarnes}), we shift the line of integration from $\re u=\epsilon$ to $\re u=-1.9$:
 		\begin{align}
 		\left(\mathcal{F}_{\Phi} H\right)(2s-1, s)  \ = \  &2\sqrt{\pi} \  \prod_{i=1}^{3} \ \Gamma\left(s-\frac{1}{2}+ \frac{\alpha_{i}}{2}\right)  \   \int_{(\eta-\frac{1}{2})}  \widetilde{h}\left(s-s_{1}-\frac{1}{2}\right)   \frac{ \pi^{-s_{1}} \prod\limits_{i=1}^{3} \ \Gamma\left( \frac{s_{1}-\alpha_{i}}{2}\right)  \Gamma\left(\frac{s_{1}}{2}+1-2s\right)}{ \Gamma\left(\frac{1+s_{1}}{2}+1-2s\right) \Gamma\left(s-\frac{1}{2}+\frac{s_{1}}{2}\right)} \ \frac{ds_{1}}{2\pi i} \nonumber \\
 		& \hspace{40pt} +   \sqrt{\pi} \  \int_{(\eta-\frac{1}{2})}   \int_{(-1.9)}  \ \text{(Same as the integrand of  (\ref{streamlinebarnes})) }  \  \frac{du}{2\pi i}   \ \frac{ds_{1}}{2\pi i }. \nonumber
 		\end{align}
 	By Stirling's formula and the same argument following (\ref{holoestim}), the integrals above represent holomorphic functions on $\epsilon< \sigma<4$. 
 		\end{proof}

 		\subsection{Step 5: Conclusion}

 	 Apply Proposition \ref{proprescon} to (\ref{finalmome}) and observe  that the poles of $s\mapsto (\mathcal{F}_{\Phi}H)(2s-1, s)$ are exactly the trivial zeros of the the arithmetic factor $L(2s-1, \Phi)$ in (\ref{desiredexpr}). We conclude that the product of functions  $s\mapsto L(2s-1, \Phi)\cdot  (\mathcal{F}_{\Phi}H)\left(2s-1, s \right)$ is holomorphic on $\epsilon< \sigma< 4$ and thus (\ref{finalmome}) provides a holomorphic continuation of $OD_{\Phi}(s)$ to the vertical strip $\frac{1}{4}+ \frac{\epsilon}{2}  <  \sigma  <  \frac{3}{4}$.  By the rapid decay of $\Phi$ at $\infty$, the pairing \ $s\mapsto \left( P, \ \mathbb{P}_{2}^{3} \Phi\cdot |\det *|^{\overline{s}-\frac{1}{2}} \right)_{\Gamma_{2}\setminus GL_{2}(\R)}$ represents an entire function. Putting (\ref{comspec}),  (\ref{diagstade}) and (\ref{finalmome}) together, we arrive at Theorem \ref{maingl3gl2}.

 	 \begin{rem}
 	 	Readers might have noticed that the analytic continuation procedure in our case (for a moment of automorphic $L$-functions of degree $6$) is much more involved than the ones for degree $4$ (cf. the second moment formula of $GL(2)$ of Iwaniec-Sarnak/ Motohashi).  To some extent, this is hinted by the presence of  off-diagonal main terms in our case when  $\Phi$ is specialized to be an Eisenstein series, whereas this does not happen in the degree $4$ cases.  See  pp. 35 of \cite{CFKRS05} for  further discussions. 
 	 	
 	 	However, the subtle arithmetic differences of the off-diagonals  are the deeper causes.  More specifically, the arithmetic in  Iwaniec-Sarnak/ Motohashi is given by a shifted Dirichlet series of two divisor functions  and  the holomorphy of the dual side in  the critical strip simply rests on the absolute convergence of such a Dirichlet series. However, the absolute convergence provided by Proposition \ref{structure} is very much insufficient in our case --- we must move the contour judiciously so that the $L$-functions present in the off-diagonal take value on $\re s_{0}=1/2$ (when $s=1/2$).

 	 	\end{rem}

 	\section{Explication  of the off-diagonal ---  Main Terms and Integral Transform}\label{expevatra}

 	\begin{comment}
 	This section is independent of Section \ref{2stepana}. We may work with a slightly larger region (\ref{newdomain}) as described in Section \ref{Stirl}. 
 	\end{comment}
 	
 	The power of spectral summation formulae (including Theorem \ref{maingl3gl2}) is encoded in the archimedean transformations involved. Therefore, it is important to obtain very explicit expressions for the transformations, usually in terms of  \textit{special functions}. It is well-known that the special functions for $GL(2)$  possess lots of symmetries and identities under various transforms. However, this ceases to be true when it comes to higher-rank groups and there remain plenty of prospects for in-depth investigations. 
 		
 		  Nevertheless, there have been some successes in higher-rank groups. For example, Stade \cite{St01, St02} was able to compute the Mellin transforms and certain Rankin-Selberg integrals of Whittaker functions for $GL_{n}(\R)$; Goldfeld et. al.  \cite{GK13, GSW21, GSW23+} obtained (harmonic-weighted)  spherical Weyl laws of  $GL_{3}(\R)$, $GL_{4}(\R)$ and $GL_{n}(\R)$ with strong power-saving error terms; and there is the work of Buttcane \cite{Bu13, Bu16} on Kuznetsov formulae for $GL(3)$. What lies at the core of  the aforementioned results are various \textit{Mellin-Barnes}  integrals which represent the special functions of higher-rank.  Judging from their experiences, this way of handling the archimedean aspects of problems is more  likely to  generalize. 
 		  
 		    In this final section, we  continue such investigation   and record several formulae for the archimedean transform  $\left(\mathcal{F}_{\Phi}H\right)\left(s_{0},  s \right)$.

 		   % In fact, this was adopted during our course of proving Theorem \ref{maingl3gl2}. 

 		  %  it is necessary to establish certain analytic properties  and preliminary estimates of the transform using the Mellin-Barnes representation in our case,  see Section  \ref{Stirl}. 

 		\begin{lem}\label{mellingl2}
 			Suppose $H\in \mathcal{C}_{\eta}$ and $h:= H^{\flat}$. On the vertical strip $-\frac{1}{2}< \re w < \eta$, we have
 			\begin{align}
 			\widetilde{h}(w)  \ := \ \int_{0}^{\infty} h(y)y^{w} \ d^{\times} y  \ = \  \frac{\pi^{-w-\frac{1}{2}}}{4} \int_{(0)}  \ H(\mu) \cdot  \frac{\Gamma\left(\frac{w+\frac{1}{2}+\mu}{2}\right) \Gamma\left(\frac{w+\frac{1}{2}-\mu}{2}\right)}{\left| \Gamma(\mu)\right|^2} \ \frac{d\mu}{2\pi i}, \label{mellin}
 			\end{align}
 			
 		\end{lem}
 		
 		\begin{proof}
 		Since $H\in \mathcal{C}_{\eta}$,  both sides of (\ref{mellin}) converge absolutely on the strip $-1/2<\re w<\eta$ by Stirling's formula and Proposition \ref{plancherel}.  Substituting the definition of $h$ as in (\ref{invers}) into $\widetilde{h}(w)$, the result follows from equation  (\ref{melwhitgl2}).

 		\end{proof}

 \begin{comment}
 \begin{align}\label{transhift}
 \left(\mathcal{F}_{\Phi} h\right)\left(s_{0},  s \right) 
 \ = \  &  \sqrt{\pi} \int_{(2M-1/2)}  \widetilde{h}\left(s-s_{1}-\frac{1}{2}\right) \pi^{-s_{1}}   \frac{\prod\limits_{i=1}^{3} \ \Gamma\left( \frac{s_{1}-\alpha_{i}}{2}\right)  }{ \Gamma\left(\frac{1+s_{1}}{2}-s_{0}\right)} \nonumber\\
 & \hspace{40pt} \cdot   \int_{(\epsilon)}  \   \frac{ \Gamma\left( \frac{u}{2}\right) \Gamma\left( \frac{u+s_{1}-s_{0}}{2}+ \frac{1}{2}-s\right)  \cdot   \prod\limits_{i=1}^{3} \ \Gamma\left( \frac{s_{0}+\alpha_{i}-u}{2}\right)  \Gamma\left(s-\frac{s_{0}+u}{2}\right) }{\Gamma\left( \frac{1-u}{2}\right) \Gamma\left(\frac{s_{0}+s_{1}-u}{2}\right)}  \  \frac{du}{2\pi i}   \ \frac{ds_{1}}{2\pi i }
 \end{align}
 \end{comment}

 	\begin{comment}
 	We should first evaluate/ remove the $s_{1}$-integral. \footnote{ Because that's the new stuff introduced and we had to shift both left and right for different purposes. }
 	\end{comment}

 \subsection{The Off-diagonal Main Term in Theorem \ref{maingl3gl2}}
 
 In this subsection, we will show that  the off-diagonal main term of Theorem \ref{maingl3gl2} (i.e.,  $ L(2s-1, \Phi) \cdot  \left(\mathcal{F}_{\Phi} H\right)\left(2s-1,  s\right)/2 $) matches up with the prediction of \cite{CFKRS05}. It suffices to prove the following proposition which is an idenitity of Mellin-Barnes integral and is of archimedean nature.  (Then the desired matching follows immediately from the functional equation (\ref{JPfunc}).)

 Note that its proof is a bit more involved than the one of Proposition \ref{stadediff1} ---   the $u$-integral was introduced for various technical reasons in Section \ref{separaOD}. Nevertheless, upon bringing in new $\Gamma$-factors  the $u$-integral turns out to  contain nice symmetries and in turn leads to a number of cancellations/ reductions  in a row.

% the main tool here is the second Barnes lemma. 

 	\begin{thm}\label{secMTcom}
 		Suppose $\frac{1}{2}+\epsilon  <  \sigma  < 1$. Then 
 		\begin{align}\label{secMTcomid}
 			\left(\mathcal{F}_{\Phi} H\right)\left(2s-1,  s \right) 
 			\ &= \ \pi^{\frac{1}{2}-s} \cdot \prod_{i=1}^{3} \ \frac{\Gamma\left(s-\frac{1}{2}+ \frac{\alpha_{i}}{2}\right)}{\Gamma\left(1-s- \frac{\alpha_{i}}{2}\right)} \cdot \int_{(0)} \  \frac{H(\mu)}{\left| \Gamma(\mu)\right|^2}   \cdot \prod\limits_{i=1}^{3}  \ \prod\limits_{\pm} \  \Gamma\left( \frac{1-s+ \alpha_{i}\pm \mu}{2}\right) \ \frac{d\mu}{2\pi i}. 
 			\end{align}
 		\end{thm}
 		
 		\begin{proof}
 			Suppose $\frac{1}{2}+\epsilon  <  \sigma  < 4$.  When $s_{0}=2s-1$, observe that the factor $\Gamma\left(\frac{1-u}{2}\right)$ in the denominator of (\ref{streamlinebarnes}) cancels with the factor $\Gamma\left(s-\frac{s_{0}+u}{2}\right)$ in the numerator of (\ref{streamlinebarnes}). This gives
 			\begin{align}\label{firuseBar}
 				\left(\mathcal{F}_{\Phi} H\right)\left(2s-1,  s \right)  \ = \   &\sqrt{\pi} \  \int_{(\eta-1/2)}  \ \widetilde{h}\left(s-s_{1}-\frac{1}{2}\right)  \frac{\pi^{-s_{1}}   \prod\limits_{i=1}^{3} \ \Gamma\left( \frac{s_{1}-\alpha_{i}}{2}\right)  }{  \Gamma\left(\frac{1+s_{1}}{2}+1-2s\right)} \nonumber\\
 				 & \hspace{60pt} \cdot \int_{(\epsilon)}  \   \frac{ \Gamma\left( \frac{u}{2}\right) \Gamma\left( \frac{u+s_{1}}{2}+ 1-2s\right)  \cdot   \prod\limits_{i=1}^{3} \ \Gamma\left(s-\frac{1}{2}+ \frac{\alpha_{i}-u}{2}\right)   }{ \Gamma\left(s-\frac{1}{2}+\frac{s_{1}-u}{2}\right)}  \  \frac{du}{2\pi i} \ \frac{ds_{1}}{2\pi i}. 
 			\end{align}
 			
 					We make the change of variable $u\to -2u$ and take
 				\begin{align*}
 				\left(a,b,c; d,e \right) \ =  \  \left(s-\frac{1}{2}+ \frac{\alpha_{1}}{2}, \ s-\frac{1}{2}+ \frac{\alpha_{2}}{2}, \ s-\frac{1}{2}+ \frac{\alpha_{3}}{2};  \ 0, \ \frac{s_{1}}{2}+1-2s \right)
 				\end{align*}
 			in (\ref{secBar}).  Notice that 
 			\begin{align*}
 				(a+b+c)+d+e \ = \  3\left(s-\frac{1}{2}\right) + \frac{s_{1}}{2}+1 -2s \ = \   s-\frac{1}{2}+\frac{s_{1}}{2} \ (:= \ f)
 			\end{align*}
 			because of $\alpha_{1}+\alpha_{2}+\alpha_{3}=0$. We find the $u$-integral is equal to 
 				\begin{align}
 				2 \cdot  \prod_{i=1}^{3} \  \frac{\Gamma\left(s-\frac{1}{2}+ \frac{\alpha_{i}}{2}\right) \Gamma\left( \frac{1}{2}-s+ \frac{s_{1}+\alpha_{i}}{2}\right)}{\Gamma\left( \frac{s_{1}-\alpha_{i}}{2}\right)}.
 				 \end{align}
 			 
 			Notice  that the three $\Gamma$-factors in denominator of the last expression cancel with the three in the numerator of the first line of  (\ref{firuseBar}).  Hence, we have
 					\begin{align}
 					\left(\mathcal{F}_{\Phi} H\right)\left(2s-1,  s \right)  \ = \   & 2\sqrt{\pi} \cdot  \prod_{i=1}^{3} \ \Gamma\left(s-\frac{1}{2}+ \frac{\alpha_{i}}{2}\right)   \int_{(\eta-1/2)}  \ \widetilde{h}\left(s-s_{1}-\frac{1}{2}\right)  \frac{\pi^{-s_{1}} \prod\limits_{i=1}^{3} \Gamma\left( \frac{1}{2}-s+ \frac{s_{1}+\alpha_{i}}{2}\right)    }{  \Gamma\left(\frac{1+s_{1}}{2}+1-2s\right)} \  \frac{ds_{1}}{2\pi i}. 
 					\end{align}
 					
 					We must now further restrict to $\frac{1}{2}+\epsilon < \sigma< 1$. 
 					We shift the line of integration to the left from $\re s_{1}=\eta-1/2$ to $\re s_{1}=\sigma_{1}$ satisfying 
 					\begin{align*}
 					2\sigma-1  \ < \sigma_{1} \ < \ \sigma. 
 					\end{align*}
 				It is easy to see no pole is crossed and  we may now apply Lemma \ref{mellingl2}:
 				\begin{align}
 					\left(\mathcal{F}_{\Phi} H\right)\left(2s-1,  s \right) 
 					\ &= \   \frac{\pi^{\frac{1}{2}-s}}{2}  \cdot \prod_{i=1}^{3} \ \Gamma\left(s-\frac{1}{2}+ \frac{\alpha_{i}}{2}\right)  \nonumber\\
 					&\hspace{40pt} \cdot \int_{(0)} \  \frac{H(\mu)}{\left| \Gamma(\mu)\right|^2}  \cdot  \int_{(\sigma_{1})} \ \frac{\prod\limits_{i=1}^{3} \Gamma\left( \frac{1}{2}-s+ \frac{s_{1}+\alpha_{i}}{2}\right) \cdot \Gamma\left(\frac{s-s_{1}+\mu}{2}\right) \Gamma\left(\frac{s-s_{1}-\mu}{2}\right) }{\Gamma\left( \frac{1+s_{1}}{2}+1-2s\right)} \ \frac{ds_{1}}{2\pi i} \ \frac{d\mu}{2\pi i}. 
 				\end{align}
 				
 					For the $s_{1}$-integral, apply the change of variable $s_{1}\to 2s_{1}$ and (\ref{secBar}) the second time but with 
 					\begin{align}
 					\left( a, b, c; d,e\right) \ = \ \left(\frac{1}{2}-s+ \frac{\alpha_{1}}{2}, \ \frac{1}{2}-s+ \frac{\alpha_{2}}{2}, \ \frac{1}{2}-s+ \frac{\alpha_{3}}{2};  \ \frac{s+\mu}{2}, \  \frac{s-\mu}{2} \right).
 					\end{align}
 				Oberserve that 
 				\begin{align*}
 					(a+b+c) + (d+e) \ = \  3\left(\frac{1}{2}-s\right) + s  \ := \ \frac{3}{2}-2s \ (:=f).   
 				\end{align*}
 			The $s_{1}$-integral is thus equal to 
 			\begin{align*}
 			\prod\limits_{i=1}^{3}	\frac{   \ \prod\limits_{\pm} \  \Gamma\left( \frac{1-s+ \alpha_{i}\pm \mu}{2}\right) }{\Gamma\left(1-s- \frac{\alpha_{i}}{2}\right)}
 			\end{align*}
 		and the   result follows. 

 				\end{proof}

 %%%%%%%%%%%%%%%%%%%%%%%%%%%%%%%%%%%%%%%%%%%%%%%%%%%%%%%%%%%%%%%%%%%%%%%%%%%%%%%%%%%%%%%%%%%%%%%%%%%%%%%%%%%%%%%%%%%%%%%%%%%%%%%%%%%%%%%%%%%%%%%%%%%%%%%%%%%%%%%%%%%%%%%%%%%%%%%%%%%%%%%%%%%%%%%%%%%%%%%%%%%%%%%%%%%%%%%%
 
  %in  the Eisenstein case of $GL(3)$, i.e.,
  
  %labelling % concise
  
 %  and remarkably general (at least heuristically) 

\subsection{Integral Transform}\label{manyintrans}

  Based on the experience of Stade \cite{St01, St02},  we \textit{do not} expect the Mellin-Barnes integrals of $(\mathcal{F}_{\Phi}H)(s_{0},s)$ (see (\ref{imporker}) below) to be completely reducible as in Proposition \ref{secMTcom}   if $(s_{0},s)$ is in a \textit{general position}. However, a necessary condition for the reductions to take place is that the Mellin-Barnes integrals  are in certain special forms. The most concrete way to see this is to express the integrals in terms of  \textit{hypergeometric functions}. 
  
  We define
 	\begin{align}
 	 _{4}\widehat{F}_{3}\left( 
 	\setlength\arraycolsep{1pt}	\begin{matrix}
 	A_{1} &          & A_{2}&           & A_{3} &         & A_{4} \\
 	          &B_{1} &          & B_{2} &           & B_{3}&
 	\end{matrix}  \ \bigg| \ z \right) 
 	  \ := \ &\frac{\Gamma(A_{1}) \Gamma(A_{2}) \Gamma(A_{3}) \Gamma(A_{4})}{\Gamma(B_{1}) \Gamma(B_{2}) \Gamma(B_{3})} \cdot \  _{4}F_{3} \left( 	\setlength\arraycolsep{1pt}
 	  \begin{matrix}
 	  A_{1} &          & A_{2}&           & A_{3} &         & A_{4} \\
 	  &B_{1} &          & B_{2} &           & B_{3}&
 	  \end{matrix}  \ \bigg| \ z \right)   \nonumber\\
 	\ := \ & \sum_{n=0}^{\infty} \ \frac{\Gamma(A_{1}+n) \Gamma(A_{2}+n) \Gamma(A_{3}+n) \Gamma(A_{4}+n)}{\Gamma(B_{1}+n) \Gamma(B_{2}+n) \Gamma(B_{3}+n)} \frac{z^{n}}{n!}. 
 	\end{align}
The series converges absolutely	when $|z|<1$ and $A_{1}, A_{2}, A_{3}, A_{4} \not \in \Z_{\le 0}$;  and  on $|z|=1$ if 
\begin{align*}
\re \left( B_{1}+B_{2}+B_{3}-A_{1}-A_{2}-A_{3}-A_{4}\right) \ > \ 0.
\end{align*} 
In fact, our hypergeometric functions are of \textit{Saalsch\"utz} type, i.e., $B_{1}+B_{2}+B_{3}-A_{1}-A_{2}-A_{3}-A_{4}=1$.  Only such special type of hypergeometric functions at $z=1$ possess many functional relations and integral representations, see \cite{M12}.

 	\begin{prop}\label{step1eva}
 	 Suppose $H\in \mathcal{C}_{\eta}$ and $h:= H^{\flat}$. On the region
 	$	\sigma_{0} \ > \ \epsilon$, $ \sigma \ < \ 4$, and  $2\sigma-\sigma_{0}-\epsilon \ > \ 0$, we have $ \left(\mathcal{F}_{\Phi} H\right)\left(s_{0}, s \right)$ equal to $ 2\pi^{3/2}$ times 
 		\begin{align}\label{off4F3}
 	&	     \int_{(\eta-1/2)} \  \widetilde{h}\left(s-s_{1}-\frac{1}{2}\right)\cdot   \frac{ \prod\limits_{i=1}^{3} \ \Gamma\left( \frac{s_{1}-\alpha_{i}}{2}\right)  }{ \Gamma\left(\frac{1+s_{1}}{2}-s_{0}\right)} \cdot  \pi^{-s_{1}}  \sec \frac{\pi}{2}\left( 2s+s_{0}-s_{1}\right) \nonumber\\
 		&\hspace{50pt} \cdot   \ _{4}\widehat{F}_{3}\left( 	\setlength\arraycolsep{1pt}\begin{matrix}
 		s- \frac{s_{0}}{2} &      & \frac{s_{0}+\alpha_{1}}{2} && \frac{s_{0}+\alpha_{2}}{2} &&  \frac{s_{0}+\alpha_{3}}{2} \\
 		&	1/2 && \frac{s_{0}+s_{1}}{2} &&  s+\frac{1}{2}+\frac{s_{0}-s_{1}}{2}
 		\end{matrix} \ \bigg| \  1\right)  \ \frac{ds_{1}}{2\pi i} \nonumber\\
 		&\hspace{20pt} - \int_{(\eta-1/2)} \  \widetilde{h}\left(s-s_{1}-\frac{1}{2}\right)\cdot  \frac{\prod\limits_{i=1}^{3} \ \Gamma\left( \frac{s_{1}-\alpha_{i}}{2}\right)  }{ \Gamma\left(\frac{1+s_{1}}{2}-s_{0}\right)} \cdot   \pi^{-s_{1}}   \sec \frac{\pi}{2}\left( 2s+s_{0}-s_{1}\right) \nonumber\\
 		&\hspace{80pt} \cdot   \ _{4}\widehat{F}_{3}\left(	\setlength\arraycolsep{.001pt} \medmuskip=.001mu
 		\begin{smallmatrix}
 	               \frac{1}{2}-s_{0}+\frac{s_{1}}{2} & &	\frac{1}{2}-s+ \frac{s_{1}+\alpha_{1}}{2} &                        & \frac{1}{2}-s+ \frac{s_{1}+\alpha_{2}}{2} &        & \frac{1}{2}-s+ \frac{s_{1}+\alpha_{3}}{2}  \\
 		                                                             &\frac{1}{2}-s +s_{1} &                                        &  1-s-\frac{s_{0}-s_{1}}{2} &                  & \frac{3}{2}-s- \frac{s_{0}-s_{1}}{2}  &
 		\end{smallmatrix}  \ \bigg| \ 1 \right)   \ \frac{ds_{1}}{2\pi i}. 
 		\end{align}

 	\end{prop}
 	
 	\begin{comment}
 		On  the vertical strip  $	\frac{1}{2}+\epsilon   <   \sigma  <   4$, we have
 		\begin{align}
 		\left(\mathcal{F}_{\Phi} h\right)\left(2s-1, s \right)
 		\ = \   &2\pi^{3/2}  \int_{(\eta -1/2)}  \widetilde{h}\left(s-s_{1}-\frac{1}{2}\right) \cdot \pi^{-s_{1}}   \frac{\prod\limits_{i=1}^{3} \ \Gamma\left( \frac{s_{1}-\alpha_{i}}{2}\right)  }{ \Gamma\left(\frac{1+s_{1}}{2}-2s+1\right)} \cdot  \sec \frac{\pi}{2}\left( 4s-1-s_{1}\right) \nonumber\\
 		&\hspace{50pt} \cdot   \ _{3}\widehat{F}_{2}\left(	\setlength\arraycolsep{1pt} \begin{matrix}
 		s-\frac{1}{2}+\frac{\alpha_{1}}{2} & &s-\frac{1}{2}+\frac{\alpha_{2}}{2} & & s-\frac{1}{2}+\frac{\alpha_{3}}{2} \\
 		& s-\frac{1}{2}+\frac{s_{1}}{2} &&  2s-\frac{s_{1}}{2}
 		\end{matrix} \ \bigg| \  1\right)     \ \frac{ds_{1}}{2\pi i}  \nonumber\\
 		& - 2\pi^{3/2}  \int_{(\eta -1/2)}  \widetilde{h}\left(s-s_{1}-\frac{1}{2}\right) \cdot \pi^{-s_{1}}   \frac{\prod\limits_{i=1}^{3} \ \Gamma\left( \frac{s_{1}-\alpha_{i}}{2}\right)  }{ \Gamma\left(\frac{1+s_{1}}{2}-2s+1\right)} \cdot  \sec \frac{\pi}{2}\left( 4s-1-s_{1}\right) \nonumber\\
 		& \hspace{60pt} \cdot  \ _{3}\widehat{F}_{2}\left(	\setlength\arraycolsep{1pt}  \medmuskip=.1mu
 		\begin{matrix}
 		\frac{1}{2}-s+ \frac{s_{1}+\alpha_{1}}{2} &         & \frac{1}{2}-s+ \frac{s_{1}+\alpha_{2}}{2} &                & \frac{1}{2}-s+ \frac{s_{1}+\alpha_{3}}{2} \\
 		&\frac{1}{2}-s +s_{1} &         &   2-2s+ \frac{s_{1}}{2} &
 		\end{matrix}  \ \bigg| \ 1 \right)\   \frac{ds_{1}}{2\pi i}.  \label{secmain}
 		\end{align}
 	\end{comment}

 	\begin{proof}
 		By Stirling's formula, we can shift the line of integration of the $u$-integral in (\ref{streamlinebarnes}) to $-\infty$.  The residual series obtained can then be identified in terms of hypergeometric series as asserted in the present proposition. This can also be verified by  \textit{InverseMellinTransform[]} command in mathematica. More systematically, one  rewrites the $u$-integral in the form of a Meijer's $G$-function. The conversion between Meijer's $G$-functions and generalized hypergeometric functions is known as \textit{Slater's theorem}, see Chapter 8 of \cite{PBM90}. 
 			\end{proof}
 		
 		%The second conclusion (\ref{secmain}) follows from the first (\ref{off4F3}) by specializing  $s_{0}=2s-1$. This completes the proof. 

 	\begin{comment}
 	\begin{lstlisting}
 	InverseMellinTransform[ Gamma[u/2]*Gamma[(s[1] - s[0] + u)/2 + 1/2 - s]*
 	Gamma[(s[0] - u + a[1])/2]*Gamma[(s[0] - u + a[2])/2]*Gamma[(s[0] - u + a[3])/2]*
 	Gamma[s - (s[0] + u)/2]/(Gamma[(1 - u)/2]*Gamma[(s[0] - u + s[1])/2]), u, 1]
 	\end{lstlisting}
 	\end{comment}

\begin{comment}
	\begin{lstlisting}
	InverseMellinTransform[Gamma[u]*Cos[Pi/2*u]*
	Gamma[(u + 1 - 2*s + s[1] - s[0])/2]*
	Gamma[(s[0] + a[1] - u)/2]*Gamma[(s[0] + a[2] - u)/2]*
	Gamma[(s[0] + a[3] - u)/2]*
	Gamma[s - (s[0] + u)/2]/Gamma[(s[0] + s[1] - u)/2], u, 2]
	\end{lstlisting}

\end{comment}

% \cite{BF18, BF21},

 Recently, the articles \cite{BBFR20},  \cite{BFW21+}  have brought in the powerful asymptotic analysis of hypergeometric functions into the study of moments and obtain  sharp estimates in the spectral aspect. Also, our class of admissible test functions in Theorem \ref{maingl3gl2} is large enough for such prospects for the family of  $GL(3)\times GL(2)$ $L$-functions, see Remark \ref{testfunc}.

Next, we prove the existence of a kernel function for the integral transform  $\left(\mathcal{F}_{\Phi} H\right)\left(s_{0},  s \right)$ when integrating against the test function $H(\mu)$ chosen on the spectral side. (The resulting formula  also serves as  an intermediate step to lead to a more useful formula for $\left(\mathcal{F}_{\Phi} H\right)\left(s_{0},  s \right)$.)  The proof will require some care but fortunately it is not too hard for our particular case.   However, readers should be cautious that this is not always true for other instances of spectral summation formulae.  For example,  the existence of kernel can be rather non-trivial  in the spectral Kuznetsov formulae of $GL(2)$ and $GL(3)$  as pointed out by  \cite{Bu16} and \cite{Mo97}.

	\begin{prop}\label{kerextr}
		Suppose $H\in \mathcal{C}_{\eta}$.  On the domain 
		\begin{align}\label{expdom}
		\sigma_{0} \ > \ \epsilon \ :=  \ 1/100, \hspace{10pt}    \sigma \ < \ 4,  \hspace{10pt} 
		2\sigma-\sigma_{0}-\epsilon \ > \ 0, \hspace{10pt} 
		\sigma_{0}+2\sigma-1-\epsilon \ > \ 0,  \hspace{10pt} 
		1+\epsilon -\sigma_{0}-\sigma \ > \ 0, 
		\end{align}
		we have
		\begin{align}
		\left(\mathcal{F}_{\Phi} H\right)\left(s_{0},  s \right) 
		\ = \  &  \frac{\pi^{\frac{1}{2}-s}}{4} \  \int_{(0)} \ \frac{H(\mu)}{\left| \Gamma(\mu)\right|^2} 	\cdot 	\mathcal{K}(s_{0},s; \alpha, \mu) \ \frac{d\mu}{2\pi i},
		\end{align}
		where the kernel function $	\mathcal{K}(s_{0},s; \alpha, \mu)$ is given explicitly by the double Barnes integrals
		\begin{align}\label{imporker}
		\mathcal{K}(s_{0},s; \alpha, \mu) \ := \   &\int_{-i\infty}^{i\infty} \int_{-i\infty}^{i\infty}  \  \frac{\Gamma\left(\frac{s-s_{1}+\mu}{2}\right) \Gamma\left(\frac{s-s_{1}-\mu}{2}\right) \prod\limits_{i=1}^{3} \ \Gamma\left( \frac{s_{1}-\alpha_{i}}{2}\right)  }{ \Gamma\left(\frac{1+s_{1}}{2}-s_{0}\right)} \nonumber\\
		& \hspace{40pt} \cdot    \frac{ \Gamma\left( \frac{u}{2}\right) \Gamma\left( \frac{s_{1}-s_{0}+u}{2}+ \frac{1}{2}-s\right)    \prod\limits_{i=1}^{3} \ \Gamma\left( \frac{s_{0}-u+\alpha_{i}}{2}\right)  \Gamma\left(s-\frac{s_{0}+u}{2}\right) }{\Gamma\left( \frac{1-u}{2}\right) \Gamma\left(\frac{s_{0}-u+s_{1}}{2}\right)}  \  \frac{du}{2\pi i}   \ \frac{ds_{1}}{2\pi i },
		\end{align} 
		and the contours follow  the  Barnes convention. 
		
	\end{prop}

\begin{rem}
	\
	\begin{enumerate}
		\item 	The domain (\ref{expdom}) is certainly non-empty as it includes our point of interest $(\sigma_{0}, \sigma)=(1/2, 1/2)$. 
		
		\item  The contours of (\ref{imporker}) may be taken explicitly as the vertical lines  $\re u =\epsilon$ and $\re s_{1}=\sigma_{1}$  with
		\begin{align}\label{pickcont}
			\sigma_{0}+2\sigma-1-\epsilon \ < \  \sigma_{1} \ < \  \sigma.
		\end{align}
	\end{enumerate}
\end{rem}

	%Regard $(s_{0},s)$ as \textit{fixed} \footnote{ Contrary to the previous sections --- the contour for the $s_{1}$-integral are allowed to depend on the choices of  $\sigma_{0}$ and $\sigma$  during this evaluation.   }

	\begin{proof}
		Suppose
		\begin{align}\label{stirdom}
		\sigma_{0} \ > \ \epsilon,  \hspace{15pt}  \sigma \ < \ 4, \hspace{15pt} \text{ and } \hspace{15pt}  2\sigma-\sigma_{0}-\epsilon \ > \ 0
		\end{align}
		as in Proposition \ref{anconpr}. Recall the expression (\ref{streamlinebarnes}) for $\left(\mathcal{F}_{\Phi}H\right)(s_{0}, s)$. This time, we shift  the line of integration of the $s_{1}$-integral to $\re s_{1}=\sigma_{1}$ satisfying
		\begin{align}\label{valid}
		\sigma_{1} \ < \ \sigma
		\end{align}
		and no pole is crossed during this shift as long as
		\begin{align}\label{nopolecr}
		\sigma_{1} \ > \ 0 \hspace{15pt} \text{ and } \hspace{15pt}  \sigma_{1} \ > \ \sigma_{0}+2\sigma-1-\epsilon. 
		\end{align}
		
		Now, assume  (\ref{expdom}). The restrictions (\ref{stirdom}), (\ref{valid}), (\ref{nopolecr}) hold and such a line of integration for the $s_{1}$-integral exists. Upon shifting the line of integration to such a position, substituting  (\ref{mellin}) into (\ref{streamlinebarnes}) and the result follows. 
	\end{proof}

	The second step is to apply a very useful rearrangement  of  the $\Gamma$-factors in the $(n-1)$-fold Mellin transform of the  $GL(n)$ spherical Whittaker function  as discovered in Ishii-Stade \cite{IS07}. We shall only need the case of $n=3$ which we describe as follows.  Recall 
	\begin{align}\label{oriGam}
		G_{\alpha}(s_{1},s_{2})  \ := \   \pi^{-s_{1}-s_{2}}  \cdot  \frac{\prod\limits_{i=1}^{3} \Gamma\left( \frac{s_{1}+\alpha_{i}}{2}\right) \Gamma\left( \frac{s_{2}-\alpha_{i}}{2}\right)}{\Gamma\left(\frac{s_{1}+s_{2}}{2}\right)}
	\end{align}
from Proposition \ref{vtmellin}.  The First Barnes Lemma, i.e., 
	\begin{align}\label{firBarn}
		\int_{-i\infty}^{i\infty} \ \Gamma\left(w+ \alpha\right)\Gamma\left(w+\mu\right)  \Gamma\left( \gamma-w\right)\Gamma\left(\delta-w\right) \ \frac{dw}{2\pi i} 
		\ = \ \frac{\Gamma\left( \alpha+\gamma\right) \Gamma\left(\alpha+\delta\right) \Gamma\left(\mu+\gamma\right) \Gamma\left(\gamma+\delta\right)}{\Gamma\left( \alpha+\mu+\gamma+\delta\right)},
	\end{align}
	can be applied \textit{in reverse} such that  (\ref{oriGam}) can be rewritten as
	\begin{align}\label{ISrec}
		G_{\alpha}(s_{1},s_{2})  \  = \  \pi^{-s_{1}-s_{2}}  &\cdot \Gamma\left(\frac{s_{1}+\alpha_{1}}{2} \right) \Gamma\left(\frac{s_{2}-\alpha_{1}}{2} \right) \nonumber\\
		&\cdot \int_{-i\infty}^{i\infty} \ \Gamma\left(z+ \frac{s_{1}}{2}-\frac{\alpha_{1}}{4}\right)\Gamma\left(z+\frac{s_{2}}{2}+\frac{\alpha_{1}}{4}\right) \Gamma\left(\frac{\alpha_{2}}{2}+\frac{\alpha_{1}}{4}-z  \right) \Gamma\left(\frac{\alpha_{3}}{2}+\frac{\alpha_{1}}{4}-z  \right) \ \frac{dz}{2\pi i},
	\end{align}
see  Section 2 of \cite{IS07}. 
	Although (\ref{ISrec}) is  less symmetric than (\ref{oriGam}),  it is really (\ref{ISrec}) that displays the recursive structure of the $GL(3)$ Whittaker function in terms of the $K$-Bessel function.

	%(One can always regain)
	
	%From (\ref{imporker}), $\alpha\mapsto \mathcal{K}(s_{0},s; \alpha, \mu)$ is invariant under the Weyl group action of $GL(3)$. One may symmetrize the right side of (\ref{ISform}) by averaging over the Weyl group. 

	\begin{thm}\label{beausym}
	Suppose  $\re s_{0}= \re s=1/2$ \ and \  $\re \alpha_{i}=\re \mu =0$. Then $\mathcal{K}(s_{0},s; \alpha, \mu) $ is equal to 
			\begin{align}\label{ISform}
				  &4\cdot  \gamma\left(-\frac{s_{0}+\alpha_{1}}{2}\right) \cdot \prod_{\pm} \  \Gamma\left( \frac{s\pm\mu-\alpha_{1}}{2}\right)  \nonumber\\
			&\hspace{15pt}  \cdot \int_{-i\infty}^{i\infty}  \  \int_{-i\infty}^{i\infty}  \   \Gamma\left(s+t\right)  \Gamma\left( \frac{1-\alpha_{1}}{2}+t\right)\cdot  \Gamma\left(\frac{\alpha_{2}}{2}+\frac{\alpha_{1}}{4}-z  \right) \Gamma\left(\frac{\alpha_{3}}{2}+\frac{\alpha_{1}}{4}-z  \right)  \cdot \prod_{\pm} \  \Gamma\left( \frac{-s\pm\mu}{2}+ \frac{\alpha_{1}}{4}  +z-t\right) \nonumber\\
			&\hspace{80pt}  \cdot \frac{\gamma\left(t+ \frac{s_{0}}{2}\right) \gamma\left(\frac{\alpha_{1}}{4}-z- \frac{s_{0}}{2}\right) }{ \gamma\left(\frac{\alpha_{1}}{4}+t-z\right)} \     \frac{dz}{2\pi i} \  	\frac{dt}{2\pi i},
		\end{align}
	where the contours may be explicitly taken as the vertical lines $\re t = a$ and $\re z=b$ satisfying  
	\begin{align}\label{IScont}
		-1/2 \ <  \ a  \ <  \ -1/4, \hspace{10pt} -1/4 \ <  \ b \ < \ 0, \hspace{10pt} \text{ and } \hspace{10pt} b-a  \ > \ 1/4
		\end{align}
	and 
	\begin{align}
		\gamma(x) \ :=  \ \frac{\Gamma(-x)}{\Gamma\left(\frac{1}{2}+x\right)}. 
	\end{align}

	\end{thm}

	\begin{comment}
		where the contours may be explicitly taken as the vertical lines $\re t = a$ and $\re z=b$ with $-1/4< a, b<0$ and $a<b$. 
		
		\begin{align}
		\hspace{-10pt}	\mathcal{K}(s_{0},s; \alpha, \mu) \ = \   & 4 \ \frac{\Gamma\left( \frac{s_{0}+\alpha_{1}}{2}\right)}{\Gamma\left( \frac{1-s_{0}-\alpha_{1}}{2}\right)} \ \Gamma\left( \frac{s+\mu-\alpha_{1}}{2}\right) \Gamma\left( \frac{s-\mu-\alpha_{1}}{2}\right) \nonumber\\
		& \cdot \int_{-i\infty}^{i\infty}  \ \frac{dt}{2\pi i} \  \Gamma\left(s+t-\frac{s_{0}}{2}\right)  \Gamma\left( \frac{1-\alpha_{1}}{2}+t-\frac{s_{0}}{2}\right) \frac{ \Gamma(-t)}{\Gamma\left( \frac{1}{2}+t\right)}  \nonumber\\
		&	\hspace{30pt} \cdot \int_{-i\infty}^{i\infty}  \ \frac{dz}{2\pi i} \  \Gamma\left(\frac{\alpha_{2}}{2}+\frac{\alpha_{1}}{4}-z  \right) \Gamma\left(\frac{\alpha_{3}}{2}+\frac{\alpha_{1}}{4}-z  \right)  \frac{ \Gamma\left( z+ \frac{s_{0}}{2}- \frac{\alpha_{1}}{4}\right)}{\Gamma\left( \frac{1}{2}-z-\frac{s_{0}}{2}+ \frac{\alpha_{1}}{4} \right)} \nonumber\\
		&\hspace{60pt}\cdot  \Gamma\left( \frac{-s+\mu}{2}+ \frac{\alpha_{1}}{4} +\frac{s_{0}}{2} +z-t\right) \Gamma\left( \frac{-s-\mu}{2}+ \frac{\alpha_{1}}{4}+ \frac{s_{0}}{2}+ z-t\right) \frac{\Gamma\left(\frac{1}{2}+\frac{\alpha_{1}}{4} -\frac{s_{0}}{2}-z+t\right)}{\Gamma\left( - \frac{\alpha_{1}}{4}+\frac{s_{0}}{2}+z-t\right)},
	\end{align}
\end{comment}

	\begin{rem}\label{whygdBarnes}
		\
	\begin{enumerate}
		\item  The assumptions of Theorem \ref{beausym}  have already covered the most interesting case for the moments in Theorem \ref{maingl3gl2}, i.e.,  on the critical line and for the tempered forms, but they are by no means essential. They were imposed to obtain a clean description of the contours (\ref{IScont}).

		\item Furthermore,  if we have either of the followings: 
		\begin{enumerate}
			\item  the cusp form $\Phi$ is considered to be fixed  and  the implicit constants in the estimates are allowed to depend on $\Phi$, 
			
			\item or $\Phi= E_{\min}^{(3)}(*; \alpha)$ where the `shifts' $\alpha_{i}$'s  are considered to be small as in \cite{CFKRS05} (i.e.,  $\ll 1/\log R$ in view of Remark \ref{testfunc}), 
		\end{enumerate}
then  it suffices to consider the case when $\alpha_{1}=\alpha_{2}= \alpha_{3}=0$ by continuity.  With $s=1/2$, this results in a simpler-looking formula for  (\ref{ISform}), i.e., 
	\begin{align*}
	\hspace{20pt}&4\cdot  \gamma\left(-\frac{s_{0}}{2}\right) \cdot \prod_{\pm} \  \Gamma\left( \frac{1/2\pm\mu}{2}\right) \int_{-i\infty}^{i\infty}  \  \int_{-i\infty}^{i\infty}  \   \Gamma\left(\frac{1}{2}+t\right)^2\cdot  \Gamma(-z)^2 \cdot \prod_{\pm} \  \Gamma\left( \frac{-1/2\pm\mu}{2}+z-t\right) \nonumber\\
	&\hspace{200pt}  \cdot  \frac{\gamma\left(t+ \frac{s_{0}}{2}\right) \gamma\left(-z- \frac{s_{0}}{2}\right) }{ \gamma\left(t-z\right)}  \  \frac{dz}{2\pi i} \  	\frac{dt}{2\pi i},
\end{align*}

\item In terms of analytic applications for $GL(n)$ involving Whittaker functions, experience has shown that the new formula of \cite{IS07} is more useful than the ones obtained previously.  For example, 

\begin{enumerate}
	\item The formula (\ref{ISrec}) was used in Buttcane \cite{Bu20} (cf. Theorem 2 therein) to simplify (considerably) the archimedean Rankin-Selberg calculation of $GL(3)$ previously  done  by Stade \cite{St93}.
	
	\item  In the recent work on the orthogonality relation for $GL(n)$ (see  \cite{GSW23+}),  such formula is crucial for strong bounds for the Whittaker functions and  the inverse Whittaker transform of their test function. 
\end{enumerate}
(This was pointed out to the author by Prof. Eric Stade and Prof. Dorian Goldfeld.  The author would like to thank their comments here.)

\item Indeed, a simple application of Stirling's formula shows that the integrand of the  Mellin-Barnes representation (\ref{ISform}) now decays exponentially as long as $|\im z|, |\im t| \to \infty$  \textit{regardless} of the size of $|\im s_{0}|$.  This favourable feature is not shared by the one of  (\ref{mainintrans}). 

	\end{enumerate}

	\end{rem}
	
		%(Note that  (\ref{ISexp}) will lead to another proof of  Proposition \ref{secMTcom} and \ref{Eissimpl}.) More generally,
		
		%	[Alternatively, one may wish to bring in the spectral transform $H(\mu)$ and to manipulate the $s_{1}$-integral.  Explain why not good enough. No decay, too narrow for contour shift]

	\begin{proof}[Proof of Theorem \ref{beausym}]
	Substitute  (\ref{ISrec}) into (\ref{imporker})   rearrange the integrals, we find that
	\begin{align}
	\mathcal{K}(s_{0},s; \alpha, \mu) \ := \   &\int_{-i\infty}^{i\infty} \  \frac{\Gamma\left(\frac{s-s_{1}+\mu}{2}\right) \Gamma\left(\frac{s-s_{1}-\mu}{2}\right) \Gamma\left(\frac{s_{1}-\alpha_{1}}{2} \right)   }{ \Gamma\left(\frac{1+s_{1}}{2}-s_{0}\right)}   \nonumber\\
	& \hspace{20pt} \cdot  \int_{-i\infty}^{i\infty} \  \Gamma\left(\frac{\alpha_{2}}{2}+\frac{\alpha_{1}}{4}-z  \right) \Gamma\left(\frac{\alpha_{3}}{2}+\frac{\alpha_{1}}{4}-z  \right)  \Gamma\left(z+\frac{s_{1}}{2}+\frac{\alpha_{1}}{4}\right) \  \nonumber\\
	&\hspace{50pt} \cdot \int_{-i\infty}^{i\infty}   \frac{ \Gamma\left( \frac{u}{2}\right) \Gamma\left( \frac{s_{1}-(s_{0}-u)}{2}+ \frac{1}{2}-s\right)      \Gamma\left(s-\frac{s_{0}+u}{2}\right)	\Gamma\left(\frac{s_{0}-u+\alpha_{1}}{2} \right)   \Gamma\left(z+ \frac{s_{0}-u}{2}-\frac{\alpha_{1}}{4}\right)}{\Gamma\left( \frac{1-u}{2}\right) }  \nonumber\\
	&\hspace{330pt}  \frac{du}{2\pi i} \ 	\frac{dz}{2\pi i} \ \frac{ds_{1}}{2\pi i }. \nonumber\\
	\label{ISexp}
	\end{align} 
	
	The innermost $u$-integral is of  $_{3}F_{2}(1)$-type (non-Saalsch\"{u}tz) in place of the original $_{4}F_{3}(1)$-type (Saalsch\"{u}tz). This offers some extra flexibility to transform the integrals further.  We will apply the following transformation identity for Barnes integrals of  $_{3}F_{2}(1)$-type (see Bailey \cite{Ba64}):
		\begin{align}\label{Ishii3F21}
		&\int_{-i\infty}^{i\infty} \ \frac{\Gamma(a+u)\Gamma(b+u)\Gamma(c+u)\Gamma(f-u)\Gamma(-u)}{\Gamma(e+u)} \ \frac{du}{2\pi i} \nonumber\\
		\ &\hspace{40pt} =  \ \frac{\Gamma(b)\Gamma(c)\Gamma(f+a)}{\Gamma(f+a+b+c-e)\Gamma(e-b)\Gamma(e-c)} \nonumber\\
		& \hspace{100pt}\cdot  \int_{-i\infty}^{i\infty} \ \frac{\Gamma(a+t)\Gamma(e-c+t)\Gamma(e-b+t)\Gamma(f+b+c-e-t)\Gamma(-t)}{\Gamma(e+t)} \ \frac{dt}{2\pi i}. 
		\end{align}

	Make a change of variable $u \to -2u$ and take
	\begin{align}
	a \  = \  s-\frac{s_{0}}{2}, & \hspace{15pt} b \ = \ \frac{s_{0}+\alpha_{1}}{2},  \hspace{15pt} c \ = \  z+\frac{s_{0}}{2}-\frac{\alpha_{1}}{4},  \nonumber\\
	f \ = \  &\frac{s_{1}-s_{0}}{2}+ \frac{1}{2}-s, \hspace{15pt} e \ = \  1/2
	\end{align}
	in (\ref{Ishii3F21}),  the $u$-integral of (\ref{ISexp}) can be written as
	\begin{align}
	2 &\cdot  \frac{\Gamma\left( \frac{s_{0}+\alpha_{1}}{2}\right) \Gamma\left(z+ \frac{s_{0}}{2}-\frac{\alpha_{1}}{4}\right) \Gamma\left(\frac{1+s_{1}}{2}-s_{0}\right)}{\Gamma\left( \frac{s_{1}}{2}+z+ \frac{\alpha_{1}}{4}\right) \Gamma\left( \frac{1-s_{0}-\alpha_{1}}{2}\right) \Gamma\left( \frac{1}{2}-z-\frac{s_{0}}{2}+ \frac{\alpha_{1}}{4}\right)} \nonumber\\
	& \hspace{20pt} \cdot \int_{-i\infty}^{i\infty} \ \frac{\Gamma\left(t+s-\frac{s_{0}}{2}\right) \Gamma\left(t+\frac{1}{2}-z-\frac{s_{0}}{2}+\frac{\alpha_{1}}{4}\right) \Gamma\left(t+\frac{1}{2}-\frac{s_{0}+\alpha_{1}}{2}\right) \Gamma\left( \frac{s_{0}+s_{1}}{2} +z-s+ \frac{\alpha_{1}}{4}-t\right) \Gamma(-t)}{\Gamma\left( \frac{1}{2}+t\right)} \ \frac{dt}{2\pi i}. 
	\end{align}
	
	Putting this back into (\ref{ISexp}). Observe that two pairs of $\Gamma$-factors involving $s_{1}$ will be cancelled and we can then execute the $s_{1}$-integral. More precisely, 
	\begin{align}
		\frac{1}{2} \cdot \mathcal{K}(s_{0},s; \alpha, \mu) \ = \ & \frac{\Gamma\left( \frac{s_{0}+\alpha_{1}}{2}\right)}{\Gamma\left( \frac{1-s_{0}-\alpha_{1}}{2}\right)}  \ \cdot \int_{-i\infty}^{i\infty}  \ \frac{dt}{2\pi i} \ \frac{\Gamma\left(t+s-\frac{s_{0}}{2}\right)  \Gamma\left(t+\frac{1}{2}-\frac{s_{0}+\alpha_{1}}{2}\right)  \Gamma(-t)}{\Gamma\left( \frac{1}{2}+t\right)}  \nonumber\\
		& \hspace{20pt} \cdot \int_{-i\infty}^{i\infty}  \ \frac{dz}{2\pi i} \  \frac{\Gamma\left(\frac{\alpha_{2}}{2}+\frac{\alpha_{1}}{4}-z  \right) \Gamma\left(\frac{\alpha_{3}}{2}+\frac{\alpha_{1}}{4}-z  \right)  \Gamma\left( z+ \frac{s_{0}}{2}- \frac{\alpha_{1}}{4}\right)}{\Gamma\left( \frac{1}{2}-z-\frac{s_{0}}{2}+ \frac{\alpha_{1}}{4}\right)} \ \Gamma\left(t+\frac{1}{2}-z-\frac{s_{0}}{2}+\frac{\alpha_{1}}{4}\right)   \nonumber\\
		&\hspace{-10pt} \cdot \int_{-i\infty}^{i\infty} \  \frac{ds_{1}}{2\pi i} \ \Gamma\left( \frac{s_{0}+s_{1}}{2} +z-s+ \frac{\alpha_{1}}{4}-t\right) \Gamma\left(\frac{s-s_{1}+\mu}{2}\right) \Gamma\left(\frac{s-s_{1}-\mu}{2}\right) \Gamma\left(\frac{s_{1}-\alpha_{1}}{2} \right).    
		\end{align}
		Applying (\ref{firBarn}) once again, we obtain 
		\begin{align}
		\hspace{-10pt}	\frac{1}{4} \cdot \mathcal{K}(s_{0},s; \alpha, \mu) \ = \   &\frac{\Gamma\left( \frac{s_{0}+\alpha_{1}}{2}\right)}{\Gamma\left( \frac{1-s_{0}-\alpha_{1}}{2}\right)} \ \Gamma\left( \frac{s+\mu-\alpha_{1}}{2}\right) \Gamma\left( \frac{s-\mu-\alpha_{1}}{2}\right) \nonumber\\
			& \cdot \int_{-i\infty}^{i\infty}  \ \frac{dt}{2\pi i} \  \Gamma\left(s+t-\frac{s_{0}}{2}\right)  \Gamma\left( \frac{1-\alpha_{1}}{2}+t-\frac{s_{0}}{2}\right) \frac{ \Gamma(-t)}{\Gamma\left( \frac{1}{2}+t\right)}  \nonumber\\
		&	\hspace{30pt} \cdot \int_{-i\infty}^{i\infty}  \ \frac{dz}{2\pi i} \  \Gamma\left(\frac{\alpha_{2}}{2}+\frac{\alpha_{1}}{4}-z  \right) \Gamma\left(\frac{\alpha_{3}}{2}+\frac{\alpha_{1}}{4}-z  \right)  \frac{ \Gamma\left( z+ \frac{s_{0}}{2}- \frac{\alpha_{1}}{4}\right)}{\Gamma\left( \frac{1}{2}-z-\frac{s_{0}}{2}+ \frac{\alpha_{1}}{4} \right)} \nonumber\\
			&\hspace{60pt}\cdot  \Gamma\left( \frac{-s+\mu}{2}+ \frac{\alpha_{1}}{4} +\frac{s_{0}}{2} +z-t\right) \Gamma\left( \frac{-s-\mu}{2}+ \frac{\alpha_{1}}{4}+ \frac{s_{0}}{2}+ z-t\right) \nonumber\\
			&\hspace{70pt} \cdot\frac{\Gamma\left(\frac{1}{2}+\frac{\alpha_{1}}{4} -\frac{s_{0}}{2}-z+t\right)}{\Gamma\left( - \frac{\alpha_{1}}{4}+\frac{s_{0}}{2}+z-t\right)}. 
		\end{align}
A final cleaning can be performed via the change of variables $t\to t+\frac{s_{0}}{2}$. This leads to (\ref{ISform}) and completes the proof. 

	\end{proof}

%%%%%%%%%%%%%%%%%%%%%%%%%%%%%%%%%%%%%%%%%%%%%%%%%%%%%%%%%%%%%%%%%%%%%%%%%%%%%%%%%%%%%%%%%%%%%%%%%%%%%%%%%%%%%%%%%%%%%%%%%%%%%%%%%%%%%%%%%%%%%%%%%%%%%%%%%%%%%%%%%%%%%%%%%%%%%%%%%%%%%%%%%%%%%%%%%%%%%%%%%%%%%%%%%%%%%%%%%%%%%%%%%%%%%%%%%%%%%%%%%%%%%%%%%%%%%%%%%%%%%%%%%%%%%%%%%%%%%%%%%%%%%%%%%%%%%%%%%%%%%%%%%%%%%%%%%%%%%%%%%%%%%%%%%%%%%%%%%%%%%%%%%%%%%%%%%%%%%%%%%%%%%%%%%%%%%%%%%%%%%%%%%%%%%%%%%%%%%%%%%%%%%%%%%%%%%%%%%%%%%%%%%%%%%%%%%%%%%%%%%%%%%%%%%%%%%%%%%%%%%%%%%%%%%%%%%%%%%%%%%%%%%%%%%%%%%%%%%%%%%%%%%%%%%%%%%%%%%%%%%%%%%%%%%%%%%%%%%%%%%%%%%%%%%%%%%%%%%%%%%%%%%%%%%%%%%%%%%%%%%%%%%%%%%%%%%%%%%%%%%%%%

		\section{Notes}
		
			\begin{rem}[Note added in Dec. 2021]
			The first version of our preprint appeared on Arxiv in December 2021.  Peter Humphries has kindly informed the author that the moment of Theorem \ref{maingl3gl2} arises naturally from the context of the $L^{4}$-norm problem of $GL(2)$ Maass forms and  can also be investigated under another set of `Kuznetsov-Voronoi' method (see \cite{BK19a, BK19b, BLM19}) that is distinct from \cite{Li09, Li11}. This is his on-going work with Rizwanur Khan.  
		\end{rem}

		\begin{rem}[Note added in Oct. 2022/ Apr. 2023]
			The preprint of Humphries-Khan has now appeared, see \cite{HK22+}.  The spectral moments considered in  \cite{HK22+} and the present paper are distinct in a number of ways.   In one case, our  spectral moments  coincide when both  $\Phi=\widetilde{\Phi}$ and $s=1/2$ hold true, but otherwise extra twistings by root numbers are present in the one considered by \cite{HK22+}. This would then lead to different conclusions in view of the Moment Conjecture of \cite{CFKRS05} (see the discussions in Section \ref{comCI}). In the other case, our spectral moments differ by a full holomorphic spectrum and thus give rise to distinct conclusions in applications toward non-vanishing (say).  All these result in  different ways of making choices of test functions, as well as different shapes of the dual sides.  The self-duality assumption was used in \cite{HK22+} to annihilate two of the terms in their proof, but no such treatment is necessary for our method. 
			
			There is also the recent preprint of Bir\'{o} \cite{Bi22+} which studies another instance of reciprocity closely related to ours, but with the decomposition `$4=2\times 2$' on the dual side instead.  His integral construction consists of  a product of an automorphic kernel with a copy of  $\theta$-function and Maass cusp form of $SL_{2}(\Z)$  attached to each variable. The integration is taken over both variables and over  the quotient $\Gamma_{0}(4)\setminus \mathfrak{h}^2$.  See equation (3.15) therein. 
		\end{rem}

		\section{Acknowledgement}
		
		This paper is an extension of the author's thesis \cite{Kw22}. It is a great pleasure to thank  Jack Buttcane, Peter Humphries,  Eric Stade,  and my Ph.D. advisor Dorian Goldfeld  for helpful and interesting discussions, as well as the reviewers for a careful reading of the article and their valuable comments.  Part of the work was completed during the author's stay at the American Institute of Mathematics (AIM) in Summer 2021. I would like to thank AIM for the generous hospitality.

\ \\
\end{document}